\documentclass[12pt, twoside]{article}
\usepackage{amsfonts,amsmath,amssymb,amsthm}
\usepackage{mathrsfs,mathtools}
\usepackage{enumerate}
\usepackage{xspace,MACROS/mydef} 
\usepackage{hyperref}
\usepackage{esint}
\usepackage{graphicx}
\usepackage{algorithm}
\usepackage{algpseudocode}
\usepackage{color}
\DeclareMathAlphabet{\mathpzc}{OT1}{pzc}{m}{it}

\usepackage[textsize=small]{todonotes}
\theoremstyle{definition}
\newtheorem{theorem}{Theorem}[section]

\newtheorem{proposition}[theorem]{Proposition}
\newtheorem{lemma}[theorem]{Lemma}
\newtheorem{definition}[theorem]{Definition}

\newtheorem{ass}[theorem]{Assumption}
\numberwithin{equation}{section}

\newcommand{\subjclass}[1]{\bigskip\noindent\emph{2010 Mathematics Subject Classification:}\enspace#1}
\newcommand{\keywords}[1]{\noindent\emph{Keywords:}\enspace#1}

\begin{document}


\baselineskip=17pt


\title{Fractional Elliptic Quasi-Variational Inequalities: \\ Theory and Numerics}

\author{Harbir Antil\\
Department of Mathematical Sciences,  George Mason University, \\ Fairfax, VA 22030, USA.\\
\texttt{hantil@gmu.edu}\\
Carlos N. Rautenberg\\
Institut f\"ur Mathematik,  Humboldt Universit\"at zu Berlin, \\ Unter den Linden 6, D-10099 Berlin, Germany.\\
\texttt{carlos.rautenberg@mathematik.hu-berlin.de}}

\date{}

\maketitle


\begin{abstract}
This paper introduces an elliptic quasi-variational inequality (QVI) problem class with fractional diffusion of order $s \in (0,1)$, studies existence and uniqueness of solutions and develops a solution algorithm. As the fractional diffusion prohibits the use of standard tools to approximate the QVI, instead we realize it as a Dirichlet-to-Neumann map for a 
problem posed on a semi-infinite cylinder. We first study existence and uniqueness of solutions for this extended QVI and then transfer the results to the fractional QVI: This introduces a new paradigm in the field of fractional QVIs. Further, we truncate the semi-infinite cylinder and show that the solution to the truncated problem converges to the solution of the extended problem, under fairly mild assumptions, as the truncation parameter $\tau$ tends to infinity. Since the constraint set changes with the solution, we develop an argument  using Mosco convergence. We state an algorithm to solve the truncated problem and show its convergence in function space. Finally, we conclude with several illustrative numerical examples. 


\end{abstract}

\section{Introduction}
\label{sec:introduccion}
The purpose of this work is twofold: 1) To introduce a new class of quasi-variational inequalities (QVIs) involving a fractional power of an elliptic operator and study existence and uniqueness of solutions; 2) To develop a solution algorithm suitable for numerical implementation. The problem class of interest is the following: Let $\Omega$ be an open, bounded and connected domain of $\R^n$, $n\geq1$, with Lipschitz boundary $\partial\Omega$  and  $\fsf \in L^\infty(\Omega)$ non-negative be given. Consider the following \emph{fractional QVI}:
\begin{equation}
\label{fractional:QVI}\tag{$\mathbb{P}$}
\text{Find }\usf \in K(\usf): \left\langle \mathcal{L}^s \usf, \usf - \vsf \right\rangle_{-s,s} \le \langle \fsf, \usf - \vsf \rangle_{-s,s}  \text{ in } \Omega, \qquad \forall \vsf \in K(\usf) ,  
\end{equation}
where $\wsf\mapsto K(\wsf)$ is defined as 
\begin{equation}\label{eq:Kset}
K(\wsf) := \{ \vsf \in \Hs \; | \; \vsf \le \Psi(\wsf) \quad \text{a.e. in } \Omega\},
\end{equation}
$ \Hs$ is defined in section \ref{sub:notation}, $\Psi(\usf):\Omega\to\mathbb{R}$ is measurable and non-negative for all $\usf\in \Hs$  and  additional assumptions are later specified on $\Psi$ in section \ref{s:FQVIvsQVI}.

The operator $\calLs$, $s \in (0,1)$, is a fractional power of the second order, symmetric and uniformly elliptic operator $\mathcal{L}$, supplemented with homogeneous Dirichlet boundary conditions (see \cite{antil2017fractional} for a discussion on nonhomogeneous boundary conditions). That is, $\mathcal{L} w := - \DIV_{x} (A \nabla_{x} w ) + c w,$
where $0 \leq c \in L^\infty(\Omega)$,  and $A(x) = A^{ij}(x) = A^{ji}(x)$, $i,j=1,\dots,n$, is bounded, measurable in $\Omega$ and satisfies the uniform ellipticity condition $\Lambda_1 |\xi|^2 \le A(x) \xi \cdot \xi \le \Lambda_2 |\xi|^2$, for all $\xi \in \mathbb{R}^n$ for almost every $x \in \Omega$, for some ellipticity constants $0 < \Lambda_1 \le \Lambda_2$. 

Fractional derivatives have been around for as long as the standard derivatives. The recent popularity of this topic can be attributed to advancements in computing (fractional operators usually lead to dense systems) and a few applications, for instance image processing and phase field models \cite{HAntil_SBartels_2017a}, turbulence \cite{wow,NEGRETE} etc. 
In particular, the study of constrained optimization problems such as the fractional obstacle problem (both elliptic and parabolic, and $\Psi$ independent of $\usf$) have been the focal point of recent research. Such problems appear, for instance, in finance as a pricing model for American options, we refer to \cite{MR2064019} for modeling and \cite{MR3100955,otarola2015finite} for a functional analytic and numerical treatment of the underlying variational inequalities.

When $s = 1$, QVIs are known to appear in many applications: They arise for instance in game theory, solid mechanics, elastoplasticity and superconductivity. We refer to \cite{MR1951035, MR0464857, harker1991generalized, MR2525009, MR1401169, MR1765905,MR3299010,MR3307334,MR2997552,MR2648182,MR1765540,MR2947539,MR2410744,Laetsch1975286} and the monographs \cite{MR745619, MR2363978} as well as the references therein for diverse theoretical approaches and possible applications.  The development of approximation methods and solution algorithms for QVIs require problem-tailored approaches due to their non-convex and non-smooth nature. Although some work has been done in finite dimensions,   the literature in infinite dimensions is rather scarce. The first sequential method of approximation of solutions was developed by Bensoussan (an account can be found in \cite{MR652685,MR756234, MR635927}) where ordering properties are exploited and convergence rates for such problems were obtained in \cite{MR538562,MR533179}.  The semismooth Newton in combination with fixed point approaches have been developed in \cite{MR3648950,MR3023771,MR3119319} for gradient and obstacle type constraints and several approaches involving dualization of the problems were pioneered by Prigozhin and Barrett (see \cite{MR3335194,MR3231973,MR3082292,MR2652615}).

In view of the aforementioned applications of fractional order PDEs and QVIs,  it is only natural to merge these ideas together  which then leads to \eqref{fractional:QVI}. To the best of our knowledge this is the first work that addresses the well-posendess of \eqref{fractional:QVI} and develops solution algorithms for such a problem. Further, we provide a precise example of application in what follows, and for the sake of simplicity we consider it on an unbounded domain. 
Let 
$\Omega= \mathbb{R}^{n}$ with $n\ge 1$ denote the location of a semi-permeable membrane that forms the base of the cylinder 
$\C = \Omega \times (0,\infty)$, where the latter contains a slightly incompressible fluid. 
We denote by $\ue$ the negative pressure of the fluid  and by $\Psi$ the negative 
osmotic pressure. In other words, the flow across the membrane occurs only when $\ue = \Psi$ 
on $\Omega \times \{0\}$ and there is no flow if $\ue < \Psi$ on $\Omega \times \{0\}$. 
In case the (average) pressure in $\C$ has an impact on the pressure of a domain $\C_{\mathrm{out}}$ which contains a certain solution, and where 
$\C_{\mathrm{out}}$ is such that $\C_{\mathrm{out}}\cap \C = \emptyset$ and $\partial\C \cap \partial\C_{\mathrm{out}} = \Omega$, then $\Psi$ is a function of $\ue$ within $\C$. Such mechanisms are usually in place on biological systems where homeostasis is of the utmost importance. 
In particular, at equilibrium, this implies that $\ue\leq \Psi(\ue)$ on $\Omega\times \{0\}$, 
and 
\begin{equation*}
\int_{\C}  \nabla \ue \cdot \nabla (\ve-\ue)\geq 0,
\end{equation*}
for all $\ve$ such that $\ve \leq \Psi(\ue)$ on $\Omega\times \{0\}$ which can be equivalently formulated as \eqref{fractional:QVI} when $s=1/2$ by an analogous result to Lemma \ref{lemma:bijections} for unbounded domains.

We emphasize that \eqref{fractional:QVI} is nonlocal and many of the classical techniques dealing with QVIs are not applicable. Indeed, existence of solutions for QVIs involve, in general, ordering properties of the associated monotone operator, in this case $\mathcal{L}^s$, and/or compactness properties of the obstacle map $\Psi$. Even though, it does not hold $\left\langle \mathcal{L}^s \usf^+, \usf^-\right\rangle_{-s,s}= 0$ for each $\usf\in \Hs$ and $s\neq 1$, it is  available that $\left\langle \mathcal{L}^s \usf^+, \usf^-\right\rangle_{-s,s}\leq 0$ for all $\usf\in \Hs$: By equation (1.3) in \cite{MR3489634} (see also \cite{SoVo} where this result was first shown using probabilistic arguments in smooth domains) we have
\begin{align*}
\left\langle \mathcal{L}^s \usf^+, \usf^-\right\rangle_{-s,s}&=\int_{\Omega}\int_{\Omega} (\usf^+(x)-\usf^+(z))(\usf^-(x)-\usf^-(z))K_s(x,z)\mathrm{d}x\mathrm{d}z\\
&=-\int_{\Omega}\int_{\Omega} (\usf^+(x)\usf^-(z)+\usf^+(z)\usf^-(x))K_s(x,z)\mathrm{d}x\mathrm{d}z\leq 0,
\end{align*}
since $K_s(x,z)\geq 0$. Hence, it is possible to pursue the proof of existence of solutions based on the property above. However, since we are interested in creating a numerical method to solve the original QVI of interest, we will exploit the analogous property on an extended domain by invoking the so-called Caffarelli-Silvestre or Stinga-Torrea extension. 

The extension idea was introduced by Caffarelli and Silvestre in $\mathbb{R}^n$ \cite{LCaffarelli_LSilvestre_2007a} 
and its extension to bounded domains is given in 
\cite{ACapella_JDavila_LDupaigne_YSire_2011a, PRStinga_JLTorrea_2010a}. We refer to the extension in bounded domains as the Stinga-Torrea extension. This idea was applied to the fractional obstacle problem in \cite{caffarelli2008regularity, MR3100955} of both elliptic and parabolic type. In a nutshell, the Caffarelli-Silvestre extension says that fractional powers of the spatial operator $\mathcal{L}$ can be realized as an operator that maps a Dirichlet boundary condition to a Neumann condition via an extension problem on the semi-infinite cylinder $\C = \Omega \times (0,\infty)$. 

Related to the nonlocal QVI given in \eqref{fractional:QVI}, we introduce the following \emph{extended QVI} problem which is local in nature and includes one extra spatial dimension $y$:
\begin{equation}\label{local:QVI}\tag{P}
\text{Find }\ue \in \mathcal{K}(\ue): a(\ue ,\ue-\ve)  \le  \langle \fsf , \tr (\ue-\ve) \rangle_{-s ,s}, \quad \forall \ve \in \mathcal{K}(\ue),
\end{equation}
where $\we\mapsto\mathcal{K}(\we)$ is defined as
\begin{equation}\label{Kset}
\mathcal{K}(\we) = \{ \ve \in \HL(y^{\alpha},\C) \; | \; \tr \ve \le \Psi(\tr \we) \quad \text{a.e. in } \Omega\},
\end{equation}
where $ \HL(y^{\alpha},\C)$ and $ \tr $ are defined in section \ref{s:alphaHext} and the bilinear form $a$ is given by 
\begin{equation}\label{eq:bilinearform}
 a(\we,\ve):=  \frac{1}{d_s} \int_{\C} {y^{\alpha}\mathbf{A}(x,y)} \nabla \we \cdot \nabla \ve
 + y^{\alpha} c(x) \we \ve , 
\end{equation}
for $\we, \ve \in \HL(y^{\alpha},\C) $ with 
$\alpha = 1-2s \in (-1,1)$, and $d_s = 2^{\alpha}\Gamma(1-s)/\Gamma(s)$. 
Moreover, $\mathbf{A}(x,y) =  \textrm{diag}
\{A(x),1\}$. We will call $y$, the \emph{extended
variable}, and the dimension $n+1$ in $\R_+^{n+1}$, the \emph{extended dimension}
of problem \eqref{local:QVI}. 
We expect that $\usf$ solving 
\eqref{fractional:QVI} fulfills $\usf = \ue|_{\Omega \times \{ 0 \}}$, where $\ue$ solves \eqref{local:QVI}; further, in Lemma~\ref{lemma:bijections} we  prove that the solution set of \eqref{fractional:QVI} and the one of \eqref{local:QVI} have the same cardinality. The result of Lemma~\ref{lemma:bijections} is in accordance with \cite{LCaffarelli_LSilvestre_2007a,ACapella_JDavila_LDupaigne_YSire_2011a,PRStinga_JLTorrea_2010a} but requires extra care and does not follow immediately from these well-known papers. However, it has serious consequences: It allows us to transfer the well-posedness of \eqref{local:QVI} to the seemingly intractable \eqref{fractional:QVI}. This initiates a new paradigm in the field of QVIs.

The paper is organized as follows: The material in Section~\ref{sec:Prelim} is well-known and is provided only so that the paper is self-contained. In Section~\ref{sub:notation} we set up the notation and define the fractional powers of $\mathcal{L}$ based on spectral theory. We supplement this definition with fractional Sobolev spaces. In Section~\ref{s:alphaHext} we state the Stinga-Torrea extension. In Section~\ref{s:Linf} we state the $L^\infty$-regularity result for solution to the linear problem. Our main work starts from Section~\ref{s:FQVIvsQVI} where we first state a general result, Lemma~\ref{lemma:bijections}, which allows us to establish the relation between \eqref{fractional:QVI} and \eqref{local:QVI}. With the help of this result, in conjunction with Assumption~\ref{ass1} $(i)$, we show the existence of solutions to \eqref{fractional:QVI} and \eqref{local:QVI}  in Theorem~\ref{thm:P}.  The Assumption~\ref{ass1} $(ii)$, in addition, leads to uniqueness of solutions to these problems. Since we are interested in developing a numerical method to solve \eqref{local:QVI}, owing to the fact that $\C$ is unbounded, in Section~\ref{tQVI} we propose a truncated problem on a bounded domain $\C_\tau = \Omega \times (0,\tau)$ with $\tau < \infty$. In Theorem~\ref{thm:convergence} we prove the convergence of truncated solutions to $\ue$, solving \eqref{local:QVI}, as $\tau \rightarrow \infty$ under fairly mild assumptions. Such a result is made possible because of our Mosco convergence result in Lemma~\ref{lemma:Mosco}. In Section~\ref{s:algo} we develop an algorithm in function space to solve the truncated problem. We prove the convergence of this algorithm in Theorem~\ref{thm:algoconv}. Finally, we conclude with several illustrative examples in Section~\ref{s:numerics}.

\section{Notation and preliminaries}
\label{sec:Prelim}

To some extent, in this section, we will use the notation from \cite{HAntil_JPfefferer_MWarma_2016a}.

\subsection{Spectral Fractional Operator}
\label{sub:notation}

Let $\Omega$ be an open, bounded and connected domain of $\R^n$, $n\geq1$, with Lipschitz boundary $\partial\Omega$. For any $s\geq0$, we introduce the following fractional order Sobolev space
\begin{align*}
\mathbb H^s(\Omega):=\left\{\usf=\sum_{k=1}^\infty \usf_k\varphi_k\in L^2(\Omega):\;\;\|\usf\|_{\mathbb H^s(\Omega)}^2:=\sum_{k=1}^\infty \lambda_k^s\usf_k^2<\infty\right\},
\end{align*}
where $\lambda_k$ are the eigenvalues of $\mathcal{L}$ with associated normalized 
(in $L^2(\Omega)$) eigenfunctions $\varphi_k$ and
\begin{align*} 
\usf_k=(\usf,\varphi_k)_{L^2(\Omega)}=\int_{\Omega}\usf\varphi_k\;dx.
\end{align*}

It is well-known that
\begin{equation}\label{inf}
\mathbb H^s(\Omega)=
\begin{cases}
H_0^s(\Omega)=H^s(\Omega)\;\;\;&\mbox{ if }\; 0<s<\frac 12,\\
H_{00}^{\frac 12}(\Omega)\;\;&\mbox{ if }\; s=\frac 12,\\
H_0^s(\Omega)\;\;&\mbox{ if }\; \frac 12<s<1.
\end{cases}
\end{equation}
Here 
\begin{align*}
H_0^s(\Omega):=\overline{\mathcal D(\Omega)}^{H^s(\Omega)},
\end{align*}
where $\mathcal D(\Omega)$ denotes the space of infinitely continuously differentiable functions with compact support in $\Omega$, and
\begin{align*}
H_{00}^{\frac 12}(\Omega):=\left\{u\in 
H^{\frac 12}(\Omega):\;\int_{\Omega}\frac{u^2(x)}{\mbox{dist}(x,\partial\Omega)}\;dx<\infty\right\} ,
\end{align*}
is the so-called Lions-Magenes space \cite{Tartar} with norm
\[
\|u\|_{H_{00}^{\frac 
12}(\Omega)}=\left(\|u\|_{H^{\frac12}(\Omega)}^2+\int_{\Omega}\frac{u^2(x)}{\mbox{dist}(x,\partial\Omega)}\;dx\right)^{\frac
	12}.
\]
We denote by $\mathbb{H}^{-s}(\Omega)$ the dual of $\mathbb{H}^s(\Omega)$ and by $\langle \cdot,\cdot \rangle_{-s,s}$ we denote the duality pairing between $\mathbb{H}^{-s}(\Omega)$ and $\mathbb{H}^{s}(\Omega)$. We next define the fractional powers of $\mathcal{L}$ (cf.~\cite{MR3489634}).

\begin{definition}
The spectral fractional operator $\mathcal{L}^s$ is defined on the space $C_0^\infty(\Omega)$ by
\begin{align}\label{def:second_frac}
\mathcal{L}^s\usf=\sum_{k=1}^\infty\lambda_k^s \usf_k\varphi_k\qquad \text{ with  } \usf_k=\int_{\Omega}\usf\varphi_k.
\end{align}
\end{definition}
By density, the operator $\mathcal{L}^s$ extends to an operator mapping from $\mathbb{H}^s(\Omega)$ to $\mathbb{H}^{-s}(\Omega)$.

\subsection{$\alpha$-Harmonic Extension}
\label{s:alphaHext}

The extension approach in $\mathbb{R}^n$ is due to Caffarelli and Silvestre and its restriction to bounded domains was given by Stinga and Torrea. The key property of the extension is that it \emph{localizes} the nonlocal operator $\mathcal{L}^{ s}$ at the expense of an extra spatial dimension. 
Before we discuss it, we introduce some notation. We denote by $\C := \Omega \times (0,\infty)$ the semi-infinite cylinder with lateral boundary $\partial_L \C := \partial\Omega \times [0,\infty)$. We let $\tau > 0$ denote a
truncation of cylinder $\C$ to $\C_\tau := \Omega \times [0,\tau]$ {and define $\partial_L \C_\tau := (\partial\Omega \times [0,\tau])\cup(\Omega\times \{\tau\})$}. Notice that both $\C$ and $\C_\tau$
are the objects in $\mathbb{R}^{n+1}$. Furthermore, we let $y$ denote the extended variable 
such that a vector $x' \in \mathbb{R}^{n+1}$ admits the following representation: 
$x' = (x_1,\dots,x_n,x_{n+1}) = (x,x_{n+1}) = (x,y)$ with $x_i \in \mathbb{R}$ for $i=1,\dots,n$, $x \in \mathbb{R}^n$ and $y\in \mathbb{R}$. 

Let $\mathcal{D} \subset \mathbb{R}^{n+1}$ be an open set, such as $\C$ or $\C_\tau$. Next we define the weighted spaces with weight function $\tau^\alpha$ with $\alpha \in (-1,1)$. These weighted spaces are necessary to tackle the singular/degenerate nature of the extended problem. We refer to \cite[Section~2.1]{Turesson}, \cite{KO84} and \cite[Theorem~1]{GU} for a more detailed discussion on such spaces. We denote by $L^2(\tau^\alpha,\D)$ to the space of measurable functions defined on $\D$ with finite norm $\|\we\|_{L^2(\tau^\alpha,\D)} := \|\tau^{\frac{\alpha}{2}}\we\|_{L^2(\D)}$. 
Further, let $H^1(\tau^\alpha,\D)$ denote the space of measurable functions 
{$\we \in L^2(\tau^\alpha,\D)$ with weak gradients $\nabla \we$ in $  L^2(\tau^\alpha,\D)$}, and endowed with the norm
\[
 \|\we\|_{H^1(\tau^\alpha,\D)} := \left( \|\we\|_{L^2(\tau^\alpha,\D)}^2 + \|\nabla \we\|_{L^2(\tau^\alpha,\D)}^2 \right)^{\frac12}
  . 
\]
We are now ready to define the Sobolev space on $\C$ that is of interest to us
\[
 \mathring{H}^1_L(\tau^\alpha,\C) 
  := \{\we \in H^1(\tau^\alpha,\C) \; | \; \we = 0 \mbox{ on } \partial_L \C \} .
\]
The space $\mathring{H}^1_L(\tau^\alpha,\C_\tau)$ is defined in a similar manner. We will 
denote the trace of a function on $\Omega$ by $\tr$.

Consider a function $\usf : \Omega \rightarrow \mathbb{R}$. We then define an $\alpha$-harmonic extension of $\usf$ (cf. \cite{LCaffarelli_LSilvestre_2007a, PRStinga_JLTorrea_2010a}) to the cylinder $\C$, as the function $\ue$ that solves 
\begin{equation}
\label{alpha_harm_Lyu}
\begin{dcases}
  -\DIV \left( y^{\alpha} \mathbf{A} \nabla \ue \right) + y^{\alpha} c\ue = 0 & \textrm{in } \C,\\
  \ue = 0 \quad  \text{on } \partial_L \C, &
  \ue = \usf \quad \text{on } \Omega \times \{0\}. \\
\end{dcases}
\end{equation}
Given  $\usf\in \Hs$, this problem has a unique solution $\ue \in \HL(y^\alpha,\C)$; in fact,  $\ue \in \HL(y^\alpha,\C)$ solves problem  \eqref{alpha_harm_Lyu} if and only if it solves the minimization problem
\begin{align*}
&\min\int_{\C}  {y^{\alpha}(\nabla \we,\mathbf{A}(x,y)}\nabla \we)
 + y^{\alpha} c(x)|\we|^2 \diff x \diff y\quad \text{ over } \we \in \HL(y^\alpha,\C),\\
&\text{subject to } \tr \we= \usf,
\end{align*}
where the objective functional is coercive, continuous and strictly convex (hence weakly lower semicontinuous). We define the solution mapping $\usf \mapsto \ue$ of \eqref{alpha_harm_Lyu} as $H_\alpha: \Hs\to \HL(y^\alpha,\C)$, i.e., $\ue=H_\alpha \usf$.

Towards this end, the fundamental result of \cite{PRStinga_JLTorrea_2010a}, see also \cite[Lemma~2.2]{ACapella_JDavila_LDupaigne_YSire_2011a}, can be stated as follows:

\begin{theorem}[Stinga--Torrea extension]
\label{TH:CS}
If $s\in(0,1)$, $\usf \in \Hs$, and $\ue$ solves \eqref{alpha_harm_Lyu} then 
\[
  d_s \mathcal{L}^s \usf = \partial_{\nu}^{\alpha} \ue,
\]
in the sense of distributions, where $d_s = 2^{\alpha}\Gamma(1-s)/\Gamma(s)$ and 
$\partial_{\nu}^{\alpha} \ue$ is the 
conormal exterior derivative of $\ue$ at $\Omega \times \{ 0 \}$ given by 
$\partial_{\nu}^{\alpha} \ue = -\lim_{y \rightarrow 0^+} y^\alpha \ue_y,$ 
where the limit must be understood in the distributional sense. 
\end{theorem}

Note that the above result for $\Omega\equiv \mathbb{R}^n$ was obtained by Caffarelli and Silvestre in \cite{LCaffarelli_LSilvestre_2007a}. In particular, the Stinga-Torrea extension entails (see \cite[Theorem 1.1]{PRStinga_JLTorrea_2010a} and \cite[Lemma~2.2] {ACapella_JDavila_LDupaigne_YSire_2011a}) that
\begin{equation}\label{eq:extension}
\langle \mathcal{L}^s \usf, \tr \we \rangle_{-s,s}= a(H_\alpha \usf,\we), \quad \forall \we\in \HL(y^\alpha,\C),
\end{equation}
for $a(\cdot,\cdot)$ defined in \eqref{eq:bilinearform}, and where $H_\alpha \usf$ 
denotes the $\alpha$-Harmonic Extension of $\usf$ as defined in the previous paragraphs.

\subsection{Boundedness of the solution to the linear problem}\label{s:Linf}

In what follows we need that the solution to the following linear problem is essentially bounded: Find $\usf \in \Hs$ such that
\begin{align}\label{eq:lin}
\begin{aligned}
\mathcal{L}^s \usf &= \fsf \quad \mbox{in } \Omega . 
\end{aligned}
\end{align}

The $L^\infty(\Omega)$ characterization of the solution of the above problem is expected in several settings. We state the following result that can be found in \cite{HAntil_MWarma_2017a} (see also
\cite{HAntil_JPfefferer_MWarma_2016a}). 

\begin{theorem}[Lipschitz domains]\label{theo-bound-1}
Let 
$\Omega$ be Lipschitz and $\fsf \in L^p(\Omega)$ with 
\begin{equation*}
\begin{array}{ll}
 p >\frac{n}{2s} & \mbox{if } n > 2s , \\
 p > 1 & \mbox{if } n = 2s , \\
 p = 1 & \mbox{if } n < 2s ,
\end{array}
\end{equation*}
$0 \le c \in L^\infty(\Omega)$, and denote by $\usf$ to the solution of \eqref{eq:lin}. Then $\usf \in L^\infty(\Omega)$ and there 
exists a constant $C = C(n,s,p,\Omega)$ such that 
$\|\usf\|_{L^\infty(\Omega)} \le C \|g\|_{L^p(\Omega)}$. 
\end{theorem}


For the paper remainder, we will assume that the conditions of Theorem~\ref{theo-bound-1} 
hold true,  i.e., the solution to \eqref{eq:lin} belongs to $L^\infty(\Omega)$.

\section{Solutions to \eqref{fractional:QVI} and \eqref{local:QVI}}\label{s:FQVIvsQVI}

In this section we address the existence and uniqueness of solutions to the QVIs determined by \eqref{fractional:QVI} and \eqref{local:QVI}, and the relationship between their solution sets. 
 As mentioned before, existence and uniqueness of solutions for QVIs involve, in general, ordering properties of the associated monotone operator and/or compactness of the obstacle map $\Psi$.  
 Before considering existence of solutions, we first study the relationship between solution set of \emph{fractional QVI} in \eqref{fractional:QVI} and \emph{extended QVI} in \eqref{local:QVI} (in case they exist).
\begin{lemma}\label{lemma:bijections}
Let $\mathbf{S}_{\mathbb{P}}$ and $\mathbf{S}_P$ denote the set of solutions to \eqref{fractional:QVI} and \eqref{local:QVI}, respectively. Then, the maps
\begin{equation*}
\tr:\mathbf{S}_P\to \mathbf{S}_{\mathbb{P}} \qquad\text{and}\qquad H_\alpha:\mathbf{S}_{\mathbb{P}}\to \mathbf{S}_P, 
\end{equation*}
are bijections.

\end{lemma}

\begin{proof}
Suppose that $\usf\in \Hs$ solves \eqref{fractional:QVI} and let $\ue$ be its canonical extension, i.e., $\ue:=H_\alpha \usf$. By definition of the bilinear form $a(\cdot, \cdot)$ and the extension result \eqref{eq:extension}, we observe that for any $\ve\in \mathcal{K}(\ue)$
\begin{align*}
a(\ue, \ve-\ue)&=\langle  \mathcal{L}^s \usf, \tr (\ve-\ue) \rangle_{-s,s}=\langle  \mathcal{L}^s \usf, \tr \ve-\usf \rangle_{-s,s}\geq \langle   \fsf, \tr \ve-\usf \rangle_{-s,s}\\
&= \langle   \fsf, \tr (\ve-\ue) \rangle_{-s,s}
\end{align*}
where we have used that $\tr\ve\in K(\tr \ue)$ and $\tr\ue=\usf$. Since $\ue\in \mathcal{K}(\ue)$ given that $\tr\ue\in K(\tr \ue)$, and $\ve\in \mathcal{K}(\ue)$ is arbitrary, then  $\ue$ solves \eqref{local:QVI}. This shows that $H_\alpha(\mathbf{S}_{\mathbb{P}})\subset \mathbf{S}_P$, and the injectivity of $H_\alpha$ follows since $H_\alpha(\vsf_1)=H_\alpha(\vsf_2)$ implies that $\vsf_1=\vsf_2$ because $\tr H_\alpha(\vsf_i)=\vsf_i$ with $i=1,2$.

Suppose that $\ue$ solves \eqref{local:QVI}. We first prove that $\ue=H_\alpha(\tr \ue)$, i.e., the solution to   \eqref{local:QVI} is identical to the canonical extension of its trace. Consider $\mathcal{R}\in \HL(y^\alpha,\C)$ and $\tr \mathcal{R}=0$ a.e. in $\Omega$, then $\ve:=\ue\pm \mathcal{R}$ satisfies $\ve\in \mathcal{K}(\ue)$. Hence, considering this $\ve$ in \eqref{local:QVI}, we observe that 
\begin{equation*}
a(\ue,\mathcal{R})=0, \qquad \mathcal{R}\in \HL(y^\alpha,\C),\:\: \tr \mathcal{R}=0 \text{ a.e. in }\Omega.
\end{equation*}
That is, $\ue$ solves the problem: Find $\we\in \HL(y^\alpha,\C)$ such that 
\begin{equation*}
\begin{dcases}
  -\DIV \left( y^{\alpha} \mathbf{A} \nabla \we \right) + y^{\alpha} c\we = 0 & \textrm{in } \C,\\
  \we = 0 \quad  \text{on } \partial_L \C, &
  \we = \tr \ue\quad \text{on } \Omega \times \{0\}. \\
\end{dcases}
\end{equation*}
Therefore, $\ue=H_\alpha(\tr \ue)$.
The extension result \eqref{eq:extension} implies  that
\begin{align*}
\langle  \mathcal{L}^s (\tr \ue),  \tr \ve-\tr\ue\rangle_{-s,s}=a(\ue, \ve-\ue)\geq \langle   \fsf, \tr \ve-\tr\ue \rangle_{-s,s}
\end{align*}
for each $\ve\in \mathcal{K}(\ue)$. Since, $\tr: \mathcal{K}(\ue)\to  K(\tr \ue)$ is surjective, then $\tr\ue$ solves \eqref{local:QVI}.  Hence, $\tr(\mathbf{S}_P)\subset  \mathbf{S}_{\mathbb{P}}$ and if $\ve_i\in \mathbf{S}_P$ and $\tr \ve_1=\tr \ve_2$, then $\ve_1=\ve_2$ since we have proven that $\ve_i=H_\alpha \tr \ve_i$ for $i=1,2$. That is, $\tr$ is injective and the surjectivity of the map follows since $\tr H_\alpha \vsf=\vsf$ for any $\vsf\in \mathbf{S}_{\mathbb{P}}$. Finally, the surjectivity of $H_\alpha$ follows as we have proven that if $\ve\in \mathbf{S}_{P}$, then $\ve$ is the canonical extension of its trace $\tr \ve\in \mathbf{S}_{\mathbb{P}}$, so that $H_\alpha (\tr \ve)=\ve$.
\end{proof}



The previous results allows us to study problem \eqref{local:QVI} and subsequently transfer solution properties to \eqref{fractional:QVI}. We start by considering the following assumption on the obstacle map $\Psi$: 
\begin{ass}[first assumption on $\Psi$]\label{ass1}
\begin{itemize}
\item[$(i)$.] If $0\leq \usf_1\leq \usf_2$, then $0\leq \Psi(\usf_1)\leq \Psi(\usf_2)$ a.e. in $\Omega$.
\item[$(ii)$.] For every non-negative $\usf$ and $\zeta\in [0,1)$, there exists $\beta\in (\zeta,1)$ such that $\Psi(\zeta \usf)\geq \beta \Psi(\usf)$ a.e. in $\Omega$.
\end{itemize}
\end{ass}

We refer to $(i)$ in Assumption~\ref{ass1} as the non-decreasing property of $\Psi$. This will be used to show existence of solutions. On the other hand, $(ii)$ in Assumption~\ref{ass1} will be used to show uniqueness (cf.~\cite{Laetsch1975286} where this property was used for the first time). Unless otherwise stated, we shall not use $(ii)$, however, $(i)$ in Assumption~\ref{ass1} is assumed to hold true for the remainder of the paper.

Both items, $(i)$ and $(ii)$, in Assumption~\ref{ass1} are satisfied by the map
\begin{equation}\label{eq:icPsi}
\Psi(\usf)(x):=\nu+\inf_{\xi\geq 0,\: x+\xi\in \Omega}\usf(x+\xi),
\end{equation}
for some $\nu\geq0$ with $\xi:=\{\xi_i\}$ and where $\xi\geq 0$ means $\xi_i\geq0$ for $i=1,2,\ldots, n$. This map arises in optimal impulse control problems.

%
%


We are now in position to present an existence and uniqueness result.

\begin{theorem}\label{thm:P} 
The set of solutions $\mathbf{S}_P$ of \eqref{local:QVI} is non-empty, it satisfies $\tr \mathbf{S}_P\subset L^\infty(\Omega)$ and if $\ue\in \mathbf{S}_P$ then
\begin{equation*}
0\leq \tr \ue \leq \usf^*, \qquad \text{ a.e. in } \Omega
\end{equation*}
where $\usf^*$ solves (weakly) the problem: Find $\usf\in \Hs$ such that $\mathcal{L}^s \usf=\fsf$. If in addition to $(i)$ the obstacle map $\Psi$ satisfies $(ii)$ in Assumption \ref{ass1},  then $\mathbf{S}_P$ is a singleton.
\end{theorem}
\color{black}
\begin{proof}
For a given $\fsf\in L^\infty(\Omega)$, and any $\we\in \HL(y^{\alpha},\C)$, let $T(\fsf, \we)$ denote the solution to the variational inequality
\begin{equation}\label{vi}
\text{Find }\ue \in \mathcal{K}(\we): a(\ue ,\ue-\ve)  \le  \langle \fsf , \tr (\ue-\ve) \rangle_{-s ,s}, \quad \forall \ve \in \mathcal{K}(\we).
\end{equation}
Since $a: \HL(y^{\alpha},\C)\times  \HL(y^{\alpha},\C)\to \mathbb{R}$ is bilinear, continuous and coercive, then $T(\fsf, \we)\in  \HL(y^{\alpha},\C)$ is uniquely defined (see \cite{MR1786735}).

Note that if $\ue\in \HL(y^{\alpha},\C)$, then $\ue^+=\max(0,\ue)$ and $ \ue ^-=-\min(0,\ue)$ belong to $\HL(y^{\alpha},\C)$ and also $a(\ue^+, \ue ^-)=0$. Additionally, if  $\ve\leq \we$ a.e., $\ve_0\in \mathcal{K}(\ve)$, and $\we_0\in \mathcal{K}(\we)$, it follows that  $\min(\ve_0,\we_0)\in \mathcal{K}(\ve)$ and $\max(\ve_0,\we_0)\in\mathcal{K}(\we)$, which yields (see \cite[Theorem 5.1, Chapter 4]{rodrigues1987obstacle}) that
\begin{equation*}
 \fsf_1\leq \fsf_2, \:\we_1\leq \we_2 \quad \Longrightarrow \quad T(\fsf_1,\we_1)\leq T(\fsf_2,\we_2),
\end{equation*}
where all inequalities hold in the ``a.e.'' sense. From this, the fact that $\Psi(\usf)\geq 0$ a.e. in $\Omega$ for all $\usf\in \Hs$, and since $\fsf$ is non-negative by initial assumption, we know that for all $\we\in \HL(y^{\alpha},\C)$, it holds that that  $0=T(0,\we)\leq T(\fsf,\we)$ a.e. in $\C$. Furthermore, $T(\fsf,\we)\leq \ue^*$ a.e., where $ \ue^*$ solves the unconstrained version of \eqref{vi}, i.e.,
\begin{equation}\label{pde}
\text{Find }\ue \in \HL(y^{\alpha},\C): a(\ue ,\ve)  =  \langle \fsf , \tr (\ve) \rangle_{-s ,s}, \quad \forall \ve \in \HL(y^{\alpha},\C).
\end{equation}
This latter fact follows since the solutions are monotone with respect to the obstacle: Consider another obstacle map $\tilde{\Psi}$ such that $\Psi(\vsf)\leq \tilde{\Psi}(\vsf)$ a.e. for all $\vsf\in  \Hs$, together with $\tilde{\mathcal{K}}(\cdot)$ defined as \eqref{Kset} but with $\tilde{\Psi}$ instead of $\Psi$. For any $\we$, it follows that if $\ve\in \mathcal{K}(\we)$ and $\tilde{\ve}\in \tilde{\mathcal{K}}(\we)$, then $\min(\ve,\tilde{\ve})\in \mathcal{K}(\we)$ and $\max(\ve,\tilde{\ve})\in\tilde{\mathcal{K}}(\we)$. Define the associated variational inequality solution map $\tilde{T}$ analogous to the map $T$ but where $\mathcal{K}(\cdot)$ is replaced by $\tilde{\mathcal{K}}(\cdot)$. Therefore, we have that $T(\fsf,\ve)\leq \tilde{T}(\fsf,\ve)$ for all $\ve\in \HL(y^{\alpha},\C)$ (see \cite[Theorem 5.1, Chapter 4]{rodrigues1987obstacle}). Hence, for $\tilde{\Psi}\equiv +\infty$, we have $\tilde{T}(\fsf,\ve)=\ue^*$, and the  inequality $T(\fsf,\we)\leq \ue^*$, follows.

The previous paragraph determines that $0$ and $\ue^*$ are sub- and super-solutions of the map $\we\mapsto T(\fsf,\we)$, i.e., $0\leq T(\fsf,0)$ and $T(\fsf,\ue^*)\leq \ue^*$ a.e. in $\C$. Since $\we\mapsto T(\fsf,\we)$ is non-decreasing, this entails that \eqref{local:QVI} admits solutions (see \cite[Chapter 15, Ãâ¬ 15.2, Theorem 3]{aubin1982mathematical}), and for each solution $\ue$, we have $0\leq \ue \leq \ue^*$.

In view of Lemma \ref{lemma:bijections}, we have that $ \tr \ue^*\equiv \usf^*$, where $\usf^*$ solves (weakly) the problem: Find $\usf\in \Hs$ such that  $\mathcal{L}^s \usf=\fsf$. Further, we observe that $\usf^*\in L^\infty(\Omega)$ by the assumptions in Section \ref{s:Linf}.  Finally, since $\tr$ preserves the pointwise order in $\HL(y^{\alpha},\C)$, we have that
\begin{equation}\label{eq:traceLinf}
0\leq \tr \ue \leq \usf^*.
\end{equation}
Hence $ \tr \ue \in L^\infty(\Omega)$ for any solution $\ue$ to \eqref{local:QVI}. 

If in addition to $(i)$, $\Psi$ also satisfies $(ii)$ in Assumption \ref{ass1}, uniqueness of solutions for  \eqref{local:QVI} follows directly by the same arguments as in \cite{Laetsch1975286}.
\end{proof}

Theorem \ref{thm:P} and Lemma \ref{lemma:bijections} amount to the following result:  If the obstacle map $\Psi$ satisfies $(i)$ in Assumption \ref{ass1}. Then, Problems \eqref{fractional:QVI} and \eqref{local:QVI} admit solutions. Moreover, the set of solutions $\mathbf{S}_{\mathbb{P}}$ and $\mathbf{S}_P$ of  \eqref{fractional:QVI} and \eqref{local:QVI}, respectively, have the same cardinality.  If in addition to $(i)$, the obstacle map $\Psi$ satisfies also $(ii)$ in Assumption \ref{ass1}, solutions to  \eqref{fractional:QVI} and \eqref{local:QVI} are unique.

\section{The truncated QVI problem}\label{tQVI} 
The focus of this section is on approximation and numerical methods for problems \eqref{fractional:QVI} and \eqref{local:QVI}. Direct discretization of \eqref{fractional:QVI}, via finite elements, requires dealing with a stiffness matrix $K_{i,j}:=\left\langle \mathcal{L}^s \usf_i, \usf_j\right\rangle_{-s,s}$ which is dense (this can be easily seen by using the equivalent integral representation of $\mathcal{L}^s$ cf.~\cite{MR3489634}), and hence the dimension of the associated discretized problem is bounded by memory limitations (similar situation occurs when we  use the integral definition of fractional operators \cite{antil2016optimal}). In addition, directly using the spectral definition \eqref{def:second_frac} needs access to eigenvalues and eigenvectors of $\mathcal{L}$ which is, again, intractable in general domains. The discretization of problem \eqref{local:QVI} is a more suitable choice for numerical methods. In this case, although the dimension is increased by one, the stiffness matrix $K_{i,j}:=a(\ue_i, \ue_j)$ is sparse. The evident limitation here is that the domain $\C$ associated to \eqref{local:QVI} is not finite. In this vein, we consider a truncation of the domain $\C$, i.e., we define $\C_\tau=\Omega\times (0, \tau)$ and the problem
\begin{equation}\label{trunc:QVI}\tag{$\mathrm{P}_\tau$}
\text{Find }\ue \in \mathcal{K}_\tau(\ue): a_\tau(\ue ,\ue-\ve)  \le  \langle \fsf , \tr (\ue-\ve) \rangle_{-s ,s}, \quad \forall \ve \in \mathcal{K}_\tau(\ue),
\end{equation}
with
\begin{equation*}
\mathcal{K}_\tau(\ve) = \{ \we \in \HL(y^{\alpha},\C_\tau) \; | \; \tr \we \le \Psi(\tr \ve) \quad \text{a.e. in } \Omega\},
\end{equation*}
and where we define $a_\tau(\cdot, \cdot)$ identically as  $a(\cdot, \cdot)$ in \eqref{eq:bilinearform} but where the domain of integration is $\C_\tau$ instead of $\C$.

In this section we study the $\tau$-limiting behavior of solutions to \eqref{trunc:QVI}. We consider an approach that under mild assumptions on the obstacle map guarantees strong convergence of solutions of \eqref{trunc:QVI} to the solution of \eqref{local:QVI}.

Associated to problem \eqref{trunc:QVI} we define the following map. Let $\mathcal{K}$ be a closed, convex and non-empty set in $\HL(y^{\alpha},\C)$ and $\gsf\in\Hsd$, then we define $R(\gsf, \mathcal{K})\in  \HL(y^{\alpha},\C)$ to be the unique solution to 
\begin{equation}\label{trunc:VI2}
\text{Find }\ue \in \mathcal{K}: a(\ue ,\ue-\ve)  \le  \langle \gsf , \tr (\ue-\ve) \rangle_{-s ,s}, \quad \forall \ve \in \mathcal{K}.
\end{equation}
We further define the set valued map $\we\mapsto \mathcal{K}^\tau(\we)$ as
\begin{equation*}
\mathcal{K}^\tau(\we):=\{ \ve \in \mathcal{K}(\we) \; | \; \ve\leq 0 \:\: \text{a.e. in } \Omega\times(\tau,+\infty)\},
\end{equation*}
and we utilize the following shorthand notation: when $\gsf$ is clear in the context, we denote 
\begin{equation}\label{S}
S(\mathcal{K}):=R(\gsf, \mathcal{K}), \quad \text{ and } \quad S^\tau(\we):=S(\mathcal{K}^\tau(\we)),
\end{equation}
for $\we\in  \HL(y^{\alpha},\C)$. The following characterization of the map $S^\tau$ will be 
useful for the remaining paper. 
\begin{lemma}\label{lemma:Increasing}
The map 
\begin{equation*}
\mathbb{R}^+\times \HL(y^{\alpha},\C)\ni (\tau, \we)\mapsto S^\tau(\we)\in \HL(y^{\alpha},\C)
\end{equation*}
is non-decreasing.
\end{lemma}

\begin{proof}
Let $\tau\leq \tau'$ and $\we\leq \we'$ a.e. in $\C$. Hence, $\tr \we\leq \tr \we'$ a.e. in $\Omega$ and implies that $\Psi(\tr \we)\leq \Psi(\tr \we')$ a.e. in $\Omega$ and then for $\ve\in \mathcal{K}(\we)$ and $\ve'\in \mathcal{K}(\we')$, we have $\min (\ve,\ve')\in \mathcal{K}(\we)$ and $\max (\ve,\ve')\in \mathcal{K}(\we')$. Further, for any $\ze\in \HL(y^{\alpha},\C)$, if  $\ue\in \mathcal{K}^\tau(\ze)$ and $\ue'\in \mathcal{K}^{\tau'}(\ze)$, it is straightforward to check that $\min (\ue,\ue')\in \mathcal{K}^\tau(\ze)$ and $\max (\ue,\ue')\in \mathcal{K}^{\tau'}(\ze)$.  Therefore, it follows that
\begin{equation*}
\Y\in \mathcal{K}^\tau(\we), \Y'\in \mathcal{K}^{\tau'}(\we') \Longrightarrow \quad \min (\Y,\Y')\in \mathcal{K}^\tau(\we), \max (\Y,\Y')\in \mathcal{K}^{\tau'}(\we').
\end{equation*}
This yields (see \cite[Theorem 5.1, Chapter 4]{rodrigues1987obstacle}) the non-decreasing property of $ (\tau, \we)\mapsto S^\tau(\we)$.
\end{proof}

We now prove that fixed points of $S^\tau:\HL(y^{\alpha},\C)\to \HL(y^{\alpha},\C)$ can be  equivalently defined as extensions by zero, of solutions to \eqref{trunc:QVI}, to $\C$ (from $\C_\tau$). 
\begin{proposition}\label{prop:equiv}
Let $E: \HL(y^{\alpha},\C_\tau)\to  \HL(y^{\alpha},\C)$ be the extension by zero operator. If $\ue\in \HL(y^{\alpha},\C_\tau)$ is a solution to \eqref{trunc:QVI} then $E\ue\in \HL(y^{\alpha},\C)$ is a fixed point of $S^\tau$. Conversely, if $\ue\in \HL(y^{\alpha},\C)$ is a fixed point of $S^\tau$, then its restriction $\ue|_{\mathcal{C}_\tau}$ belongs to $\HL(y^{\alpha},\C_\tau)$ and solves \eqref{trunc:QVI}.
\end{proposition}

\begin{proof}
Analogously as in the proof of Theorem \ref{thm:P} with the map $T(\cdot, \cdot)$, we have  for any $\we$ that $0\leq R(\fsf,\mathcal{K}^\tau(\we))\leq \ue^*$ where $\ue^* $ solves \eqref{pde} (note that $\fsf\geq 0$ a.e. in $\Omega$). This also implies that if $\ue=S^\tau(\ue)=R(\fsf,\mathcal{K}^\tau(\ue))$ then $\ue\geq 0$. Define $\mathcal{K}^\tau(\we)^+=\mathcal{K}^\tau(\we)\cap\{\we:\we\geq 0 \text{ a.e. in } \C\}$. Since $R(\fsf,\mathcal{K}^\tau(\we))\in \mathcal{K}^\tau(\we)^+$ and $R(\fsf,\mathcal{K}^\tau(\we)^+)\in \mathcal{K}^\tau(\we)$, it is straightforward to prove that for any $\we$, we have $R(\fsf,\mathcal{K}^\tau(\we))=R(\fsf,\mathcal{K}^\tau(\we)^+)$. Since 
\begin{equation*}
\mathcal{K}^\tau(\ue)^+:=\{ \ve \in \mathcal{K}(\ue) \; | \; \ve\geq 0 \:\: \text{a.e. in } \C_\tau, \:\:  \ve= 0 \:\: \text{a.e. in } \Omega\times(\tau,+\infty)\},
\end{equation*}
we have that $\mathcal{K}^\tau(\ue)^+\equiv \{ E\ve \; | \; \ve\geq 0 \:\: \text{a.e. in } \C_\tau, \:\ve\in  \mathcal{K}_\tau(\ue)\}$. Hence, the solution to $\ue=S^\tau(\ue)=R(\fsf,\mathcal{K}^\tau(\ue))$ is equivalently defined as the extension by zero of the solution to
\begin{equation*}
\text{Find }\hat{\ue} \in \mathcal{K}_\tau(\hat{\ue})^+: a_\tau(\hat{\ue} ,\hat{\ue}-\ve)  \le  \langle \fsf , \tr (\hat{\ue}-\ve) \rangle_{-s ,s}, \quad \forall \ve \in \mathcal{K}_\tau(\hat{\ue})^+,
\end{equation*}
where  $\mathcal{K}_\tau(\we)^+=\mathcal{K}_\tau(\we)\cap\{\we:\we\geq 0 \text{ a.e. in } \C_\tau\}$. We are only left to prove that if $\hat{\ue}$ solves the above QVI, then it solves equivalently \eqref{trunc:QVI} and this is done analogously as in the begining of the proof considering that $R(\fsf,\mathcal{K}_\tau(\we))\geq 0$ and that $R(\fsf,\mathcal{K}_\tau(\we))\in  \mathcal{K}_\tau(\we)^+$, for any $\we$.  
\end{proof}

In addition to $(i)$ in Assumption \ref{ass1}, we consider the following assumption on the obstacle mapping to hold true for the rest of the paper.

\begin{ass}[second assumption on $\Psi$]\label{ass2}
\begin{itemize}
\item[$(i)$.] $\Psi(\usf)\geq \nu>0$ a.e. in $\Omega$, for all $\usf\in \Hs$.
\item[$(ii)$.] For $\usf_n, \usf^*\in \Hs$ with $n\in\mathbb{N}$: If $\usf_n\to \usf^*$ in $L^p(\Omega)$, for all $p>1$ then $\Psi(\usf_n)\to \Psi(\usf^*)$ in $L^\infty(\Omega)$.
\end{itemize}
\end{ass}
 
Note that the map $\Psi$ in \eqref{eq:icPsi} does not satisfy $(ii)$ in {Assumption~\ref{ass2}}. However, such property would hold for an appropriate regularization $\tilde{\Psi}$ of $\Psi$ defined as {$\tilde{\Psi}(\usf):=\Psi\circ I(\usf)$} where $I$ is some integral approximation of the identity. 
 
In what follows, we prove convergence of solutions of \eqref{trunc:QVI} to the solution of  \eqref{local:QVI} in a general framework.  First, we define Mosco convergence for closed, convex and non-empty subsets on a reflexive Banach space (see \cite{mosco1969convergence, rodrigues1987obstacle}). 
\begin{definition}[\textsc{Mosco convergence}]\label{definition:MoscoConvergence}
Let $\mathcal{K}$ and $\mathcal{K}_n$, for each $n\in\mathbb{N}$, be non-empty, closed and convex subsets of a reflexive Banach space $X$. We say that the sequence \textit{$\{\mathcal{K}_n\}$ converges to $\mathcal{K}$ in the sense of Mosco} as $n\rightarrow\infty$, if the following two conditions hold:
\begin{description}
  \item[\quad (i)] For each $v\in \mathcal{K}$, there exists $\{v_n\}$ such that $v_n\in \mathcal{K}_n$ and $v_n\rightarrow v$ in $X$.
  \item[\quad (ii)] If $v_n\in \mathcal{K}_n$ and $v_n\rightharpoonup v$ in $X$ 
  along a subsequence, 
  then $v\in \mathcal{K}$.
\end{description}
\end{definition}

The importance of Mosco convergence lies in the fact that if  $\mathcal{K}_n\to \mathcal{K}$ in the sense of Mosco for $\HL(y^{\alpha},\C)$, then it follows that $S(\mathcal{K}_n)\to S(\mathcal{K})$ in  $\HL(y^{\alpha},\C)$ (see \cite{mosco1969convergence, rodrigues1987obstacle}). We now  provide conditions for Mosco convergence for $\{\mathcal{K}^\tau(\ve_n)\}$  and $\{\mathcal{K}^{\tau_n}(\ve_n)\}$ for sequences $\{\ve_n\}$ and $\{\tau_n\}$ in $\HL(y^{\alpha},\C)$ and $\mathbb{R}^+$, respectively.

\begin{lemma}\label{lemma:Mosco}
Let $\{\ve_n\}$ be a bounded sequence in $\HL(y^{\alpha},\C)$ and  satisfy $0\leq \ve_n\leq \ve_{n+1}\leq \Y^*$ for some  $ \Y^*\in \HL(y^{\alpha},\C)$ such that $\tr  \Y^* \in L^\infty(\Omega)$. Then, given $\tau^*\in (0,+\infty]$ and a sequence $\{\tau_n\}$ in $(0,+\infty)$ such $\tau_n\to \tau^*$, there exists $\ve^*\in \HL(y^{\alpha},\C)$ such that
\begin{equation*}
\mathcal{K}^{\tau^*}(\ve_n)\to \mathcal{K}^{\tau^*}(\ve^*), \quad \text{ and } \quad \mathcal{K}^{\tau_n}(\ve_n)\to \mathcal{K}^{\tau^*}(\ve^*),
\end{equation*}
both in the sense of Mosco, as $n\to\infty$,  where $\ve_n \rightharpoonup \ve^*,$   in $\HL(y^{\alpha},\C)$ and $\ve_n \uparrow \ve^*$  pointwise a.e. in $\C$. Here, if $\tau^*=+\infty$,  we denote  $\mathcal{K}^{\tau^*}(\ve^*):= \mathcal{K}(\ve^*)$. 
\end{lemma}

\begin{proof}
Since $\{\ve_n\}$ is a bounded sequence in $\HL(y^{\alpha},\C)$, then $\ve_n \rightharpoonup \ve^*$ in $ \HL(y^{\alpha},\C)$, along a subsequence,  as $n\to\infty$. However, $0\leq \ve_n\leq \ve_{n+1}\leq \Y^*$ which implies $0\leq \tr \ve_n\leq \tr \ve_{n+1}\leq \tr\Y^*\in L^\infty(\Omega)$ and this yields that for some $\ve^*\in  \HL(y^{\alpha},\C)$ we have
\begin{align*}
\ve_n \uparrow \ve^*,& \quad\text{ pointwise a.e. in } \C;\\
\ve_n \rightharpoonup \ve^*,& \quad\text{ in } \HL(y^{\alpha},\C);\\
\Psi (\tr \ve_n)\to \Psi (\tr \ve^*),& \quad \text{ in } L^\infty(\Omega);
\end{align*}
for the whole sequence $\{\ve_n\}$ and not only a subsequence (the gap between the subsequence argument and the entire {sequence is} bridged by the monotonicity of the entire sequence $\{\ve_n\}$). Here we have used Assumption \eqref{ass2} with the fact that $\tr \ve_n, \tr \ve^*\in \Hs$ for $n\in\mathbb{N}$ and $\tr \ve_n\to \tr \ve^*$  in  $L^p(\Omega)$ for all $p>1$: Note that we immediately have that $\tr \ve_n\rightharpoonup \tr \ve^*$  in  $L^p(\Omega)$ for all $p>1$, and by the monotone convergence theorem $\|\tr \ve_n\|_{L^p(\Omega)}\to \|\tr \ve^*\|_{L^p(\Omega)}$ which implies the claim (see  Proposition 3.32 in \cite{MR2759829} which shows that weak convergence {in conjunction with} convergence of the norms implies strong convergence).

Suppose that $\we_n\in \mathcal{K}^{\tau^*}(\ve_n)$ for $n\in\mathbb{N}$ and that $\we_n \rightharpoonup \we^*$  in $\HL(y^{\alpha},\C)$ for some $\we^*$. Since $\tr \HL(y^{\alpha},\C)\equiv \Hs$ and $ \Hs$ compactly embeds into $L^2(\Omega)$, from $\tr \we_n \le \Psi(\tr \ve_n)$ a.e. in $\Omega$  and $\we_n \le 0$ a.e. in $\Omega\times(\tau,+\infty)$, we observe that $\we^*\in \mathcal{K}^{\tau^*}(\ve^*)$ by taking the limit as $n\to \infty$ (along a subsequence) since Assumption \ref{ass2} holds.

Let $\we\in \mathcal{K}^{\tau^*}(\ve^*)$ be arbitrary. Since $\Psi (\tr \ve_n)\geq  \nu>0$, we define $\beta_n:=(1+\|\eta_n-\eta^*\|_{L^\infty(\Omega)}/\nu)^{-1}$, where $\eta_n:=\Psi (\tr \ve_n)$ and $\eta^*:=\Psi (\tr \ve^*)$, and observe that $\beta_n\uparrow 1$ and further $\beta_n\eta_n\leq \eta^*$. Therefore, $\we_n:=\beta_n\we\in \mathcal{K}^{\tau^*}(\ve_n)$ and $\we_n\to \we$ in $\HL(y^{\alpha},\C)$. This proves that $\mathcal{K}^{\tau^*}(\ve_n)\to \mathcal{K}^{\tau^*}(\ve^*)$ in the Mosco sense.

Suppose that $\we_n\in \mathcal{K}^{\tau_n}(\ve_n)$ and that $\we_n \rightharpoonup \we^*$  in $\HL(y^{\alpha},\C)$ for some $\we^*$. Then, we obtain that $\tr \we^* \le \Psi(\tr \ve^*)$ a.e. in $\Omega$ from taking the $n\to\infty$ limit (along a subsequence) in $\tr \we_n \le \Psi(\tr \ve_n)$ a.e. in $\Omega$, i.e., $\we^*\in\mathcal{K}^{\tau^*}(\ve^*)$.

Let $\we\in \mathcal{K}^{\tau^*}(\ve^*)$ be arbitrary and without loss of generality consider $\tau^*=+\infty$ (the case  $\tau^*\in(0,+\infty)$ is handled analogously). Then, as before, $\we_n:=\beta_n\we$ satisfies $\tr \we_n \le \Psi(\tr \ve_n)$ for the same $\beta_n$ defined in the previous paragraphs and also $\we_n\to\we$ in $\HL(y^{\alpha},\C)$ as $n\to\infty$. Define $\rho_n:\C\to\mathbb{R}$ such that $\rho_n(\Omega\times(\tau_n,+\infty))\equiv 0$,  $\rho_n(\Omega\times[0,\tau_n-\epsilon))\equiv 1$, and smooth on $\Omega\times[\tau_n-\epsilon,\tau_n)$ and such that $|\nabla \rho_n|\leq m$ a.e. for some $m>0$. Then, we define $\Y_n:=\rho_n\we_n$ which satisfies that $\Y_n\in \mathcal{K}^{\tau_n}(\ve_n)$. In addition, we observe that
\begin{equation*}
|\we-\Y_n|_{\HL(y^{\alpha},\C)}^2\leq I^1_n(\tau_n)+I^2_n(\tau_n)+I^3(\tau_n),
\end{equation*}
where
\begin{align*}
I^1_n(\tau_n):=\int_{\Omega\times[0,\tau_n-\epsilon))}y^\alpha |\nabla(\we-\we_n)|^2, \:\: I^2_n(\tau_n):=\int_{\Omega\times[\tau_n-\epsilon,\tau_n))}y^\alpha |\nabla(\we-\rho_n\we_n)|^2,
\end{align*}
and $I^3_n(\tau_n):=\int_{\Omega\times[\tau_n,+\infty))}y^\alpha |\nabla\we|^2$. Note that 
\begin{equation*}
I^1_n(\tau_n)\leq\int_{\Omega\times[0,+\infty))}y^\alpha |\nabla(\we-\we_n)|^2\to 0, \quad \text{ as } \quad n\to \infty,
\end{equation*}
given that $\we_n\to\we$ in $\HL(y^{\alpha},\C)$ as $n\to\infty$. Since $\rho_n$ is smooth, $|\nabla \rho_n|\leq m$ a.e. for some $m>0$ and also $\beta_n\in(0,1)$, it follows that 
\begin{equation*}
I^2_n(\tau_n)\leq C \int_{\Omega\times[\tau_n-\epsilon,\tau_n))}y^\alpha |\nabla\we|^2\to 0, \quad \text{ as } \quad n\to \infty,
\end{equation*}
for some $C>0$ independent of $n\in\mathbb{N}$. In addition, we also have that $I^3(\tau_n)\to 0$ as $n\to\infty$. This shows that $\Y_n\to \we$ in $\HL(y^{\alpha},\C)$ as $n\to \infty$ and consequently that $\mathcal{K}^{\tau_n}(\ve_n)\to \mathcal{K}^{\tau^*}(\ve^*)\equiv\mathcal{K}(\ve^*)$ in the Mosco sense.
\end{proof}


We are now in position to provide the proof of how problem \eqref{trunc:QVI} approximates \eqref{local:QVI}. The result is made available at this point by Proposition \ref{prop:equiv}  and Lemma \ref{lemma:Mosco}. From now on, we omit the use of the extension by zero operator $E$ and consider all functions to be defined on $\C$.

\begin{theorem}\label{thm:convergence}
Problem \eqref{trunc:QVI} admits solutions for $\tau\in(0,\infty)$. Further, let $\{\tau_n\}$ be a positive sequence such that $\tau_n\to \infty$ for $n\to \infty$, then there exists  a sequence $\{\ue_{\tau_n}\}$, where $\ue_{\tau_n}$ solves $(\mathrm{P}_{\tau_n})$ for $n\in\mathbb{N}$, such that
\begin{equation*}
\ue_{\tau_n}\to \ue, \:\: \text{ in } \HL(y^{\alpha},\C),
\end{equation*}
as $n\to \infty$ for some $\ue$ that solves  \eqref{local:QVI}. 
\end{theorem}

\begin{proof}
Existence of solutions to \eqref{trunc:QVI} follow from the same arguments as in Theorem \ref{thm:P}. We concentrate on the second part of the statement. 


Since $\tau\mapsto S^\tau(\we)$ is an increasing map, we have for $\tau\leq \tau'$ that
\begin{equation}\label{ineq}
0\leq S^\tau(\we)\leq  S^{\tau'}(\we)\leq T(\we)\leq \ue^*, 
\end{equation}
a.e. where $T(\we)$ denotes the solution to \eqref{vi} and $ \ue^*$ the solution to the unconstrained problem \eqref{pde}. Further, the maps $\we\mapsto S^\tau(\we)$, for $\tau\in (0,+\infty)$ and $\we\mapsto T(\we)$, are non-decreasing, which implies that the sequences $\{\ve_n\}$ and $\{\ue_n\}$ defined as $\ve_n=S_\tau(\ve_{n-1})$,  $\we_n=T(\we_{n-1})$ with $\ve_0=\we_0=0$ are non-decreasing, as well, and located on the interval $[0, \ue^*]$. A simple optimization argument shows that $\{\ve_n\}$ and $\{\we_n\}$ are also bounded in $\HL(y^{\alpha},\C)$, and then the sequences satisfy the assumptions of Lemma \ref{lemma:Mosco}. This implies 
\begin{equation*}
\mathcal{K}^\tau(\ve_n)\to \mathcal{K}^\tau(\ve^*), \:\:\text{ in the sense of Mosco},
\end{equation*}
as $n\to\infty$ where $\ve_n\rightharpoonup \ve^*$ in $\HL(y^{\alpha},\C)$. Therefore, we have that $\ve_n=S^\tau(\ve_{n-1})=S(\mathcal{K}^\tau(\ve_{n-1}))\to S^\tau(\ve^*)$ in $\HL(y^{\alpha},\C)$ as $n\to\infty$, and hence
\begin{equation*}
\ve_n\to \ve^*, \:\: \text{ in }\HL(y^{\alpha},\C),
\end{equation*}
as $n\to\infty$ and $\ve^*$ solves \eqref{trunc:QVI}. The same argument applies to the sequence $\{\we_n\}$ and we obtain also that $\we_n\to \we^*$ in $\HL(y^{\alpha},\C)$ and that $\we^*$ solves \eqref{local:QVI}. 

From the above argument and \eqref{ineq}, we observe that if $\ue_{\tau_n}$ and $\ue_{\tau_{n+1}}$ are solutions, obtained by the above paragraph iteration procedure, to $(\mathrm{P}_{\tau_n})$  and $(\mathrm{P}_{\tau_{n+1}})$, respectively, we have
\begin{equation}\label{ineq2}
0\leq \ue_{\tau_n}\leq  \ue_{\tau_{n+1}}\leq \we^* \leq \ue^*,
\end{equation}
which implies that $\ue_{\tau_n}\uparrow \ue$ pointwise in $\C$ and also $\ue_{\tau_n}\rightharpoonup \ue$ in $\HL(y^{\alpha},\C)$, for some $\ue$. Hence by Lemma \ref{lemma:Mosco} we have that
\begin{equation}\label{eq:Mosco}
\mathcal{K}^{\tau_n}(\ue_{\tau_n})\to \mathcal{K}(\ue), \:\:\text{ in the sense of Mosco}.
\end{equation}
Finally, since $\ue_{\tau_n}=S(\mathcal{K}^{\tau_n}(\ue_{\tau_n}))$ and \eqref{eq:Mosco} holds, $\ue_{\tau_n}=S(\mathcal{K}^{\tau_n}(\ue_{\tau_n}))\to S(\mathcal{K}(\ue))$. Given that  $\ue_{\tau_n}\rightharpoonup \ue$, we have $\ue_{\tau_n}\to \ue$ in $\HL(y^{\alpha},\C)$.  Hence, $\ue= S(\mathcal{K}(\ue))$, i.e., $\ue$ is a solution to \eqref{local:QVI}.
\end{proof}


\section{An algorithm}\label{s:algo} 

In this section we introduce a solution algorithm for \eqref{local:QVI} or \eqref{trunc:QVI}. In particular, 
we consider a sequential minimization solution algorithm that converges under mild assumptions to the solution of \eqref{local:QVI} or \eqref{trunc:QVI}, depending on if the limit of the  truncation parameter sequence is finite or not.
 
\begin{algorithm}
	\caption{Increasing Monotonic Sequential Minimization}
		\label{Algorithm}
	\textbf{Data:} $\fsf\in L^\infty(\Omega)$ and $\fsf\geq 0$ a.e. in $\Omega$, non-negative real valued sequence $\{\tau_m\}_{m=0}^\infty$
	\begin{algorithmic}[1]
	\State Set $\ue_{0}\in \HL(y^{\alpha},\C)$ to satisfy $0\leq \ue_{0}\leq S(\mathcal{K}^{\tau_0}(\ue_{0}))$ and $n=0$.
	\State \textbf{repeat}
	\State \label{step1}\quad Compute  $ \ue_{n+1}:=S(\mathcal{K}^{\tau_n}(\ue_n)).$
	\State \quad Set $n=n+1$.
	\State \textbf{until} some stopping rule is satisfied.
	\end{algorithmic}
\end{algorithm}

Before we establish the convergence result for the above algorithm note that $\ue_{0}$ in Algorithm \ref{Algorithm} can be taken to be zero.

\begin{theorem}\label{thm:algoconv}
Let $\{\ue_n\}_{n=0}^\infty$ be generated by Algorithm \ref{Algorithm} for a monotonically increasing sequence $\{\tau_m\}_{m=0}^\infty$. Then,
\begin{equation*}
\ue_n\to \ue, \:\: \text{ in } \HL(y^{\alpha},\C),
\end{equation*}
where $\ue$ solves \eqref{local:QVI} if $\lim_{m\to\infty}\tau_m= \infty$ and $\ue|_{\C_\tau}$ solves  \eqref{trunc:QVI}  if $\lim_{m\to\infty}\tau_m= \tau <\infty$
  
\end{theorem}
\begin{proof}
The map $(\tau,\ue)\mapsto S( \mathcal{K}^{\tau}(\ue)) $ is non-decreasing for $\tau>0$ and $\ue\geq \ue_0$ (see Lemma \ref{lemma:Increasing}), then we prove that
\begin{equation}\label{eq:IneyInc}
0\leq \ue_0\leq \ue_n\leq \ue_{n+1}\leq \ue^*
\end{equation}
where  $ \ue^*$ denotes the solution to the unconstrained problem \eqref{pde}. We proceed by induction. Since by definition $\ue_0$ is a sub-solution of $\we\mapsto S( \mathcal{K}^{\tau_0}(\we))$, we have that $\ue_{0}\leq S( \mathcal{K}^{\tau_0}(\ue_0))=:\ue_1$. Suppose that for some $n\in \mathbb{N}$, we observe that $\ue_{n-1}\leq \ue_{n}$, then by  the non-decreasing property of the map $(\tau,\ue)\mapsto S( \mathcal{K}^{\tau}(\ue))$ and the definition of $\ue_n$ and $\ue_{n+1}$, we observe
\begin{equation*}
\ue_0\leq \ue_n:=S( \mathcal{K}^{\tau_{n-1}}(\ue_{n-1}))\leq S( \mathcal{K}^{\tau_{n}}(\ue_{n}))=:\ue_{n+1}. 
\end{equation*}
The fact that $\ue_{n+1}\leq \ue^*$ follows since $S( \mathcal{K}^{\tau}(\we))\leq \ue^*$ holds for all $\we\in \HL(y^{\alpha},\C)$.

Suppose that $\lim_{m\to\infty}\tau_m= \infty$. Since $\{\ue_n\}$ is bounded in  $\HL(y^{\alpha},\C)$,  \eqref{eq:IneyInc} holds for $n\in\mathbb{N}_0$ and $\tr \ue^*\in L^\infty(\Omega)$, we observe by Lemma \ref{lemma:Mosco} that
\begin{equation*}
\mathcal{K}^{\tau_n}(\ue_n)\to \mathcal{K}({\ue}),
\end{equation*}
in the sense of Mosco, where $\ue_n\rightharpoonup {\ue}$ (along the entire sequence) in  $\HL(y^{\alpha},\C)$ as $n\to \infty$. Hence, $\ue_{n+1}=S(\mathcal{K}^{\tau_n}(\ue_n))\to S(\mathcal{K}({\ue}))$ in  $\HL(y^{\alpha},\C)$ as $n\to\infty$. This implies that ${\ue}=S(\mathcal{K}({\ue}))$, i.e., ${\ue}$ solves  \eqref{local:QVI} and that $\ue_n\to {\ue}$ in  $\HL(y^{\alpha},\C)$ as $n\to \infty$.

If $\lim_{m\to\infty}\tau_m= \tau^*<+\infty$, then again  by Lemma \ref{lemma:Mosco} we have by the same argument in the previous paragraph that
\begin{equation*}
\mathcal{K}^{\tau_n}(\ue_n)\to \mathcal{K}^{\tau^*}({\ue}),
\end{equation*}
and consequently $\ue_n\to {\ue}$ in  $\HL(y^{\alpha},\C)$ as $n\to \infty$, where ${\ue}|_{\C_{\tau^*}}$ solves \eqref{trunc:QVI}.
\end{proof}

\section{Numerical realization}\label{s:numerics}
Before proceeding further, we shall elaborate on the Step~\ref{step1} of Algorithm~\ref{Algorithm}. In practice, we consider $\tau_n=\tau$ for all $n$ and some large enough $\tau$ in  Algorithm~\ref{Algorithm}; this is a fixed point iteration to approximate a solution to \eqref{trunc:QVI}. If $\{\ue_{n}\}$ is the sequence generated, the stopping criterion considered is satisfied, for this fixed point iteration, as soon as 
\begin{equation}\label{eq:eps1}
\frac{\| \ue_{n+1} - \ue_{n} \|_{\HL(y^\alpha,\C_\tau)}}{\| \ue_{n+1} \|_{\HL(y^\alpha,\C_\tau)}} < \varepsilon_1 ,
\end{equation}
or $n = n_{\max}$. 


As a sub-step, in this fixed point iteration, we need solution to a variational inequality. We employ a Semismooth Newton Algorithm with Regularization, see \cite[pp. 248]{MR2441683} for details:
\begin{algorithm}[H]
\caption{Semismooth Newton Algorithm with Regularization}
\label{algo:ssn}
\begin{algorithmic}[1]
\State \textbf{Input}: $\overline{\mu}$, $\theta$, $\Psi = \Psi(\tr\ue_n)$, $k_{\max}$ and set $k = 0$, $\tau = \tau_n$, $\ue_k = \ue_{n}$ 
\State \textbf{Output}: $\ue_{k+1}$
\State {\bf repeat}  
\State \quad Set $\mathcal{A}_k = \{ x : \left(\overline\mu+\theta(\tr \ue_k - \Psi)(x) > 0 \right) \}$, $\mathcal{I}_k = \Omega \setminus \mathcal{A}_k$. 
\State \label{algo:ssn_step} \quad Solve $\ue_{k+1} \in \HL(y^{\alpha},\C_\tau)$:
       \[
         a(\ue ,\ve) + \left( \overline{\mu} + \theta (\tr \ue - \Psi), \chi_{\mathcal{A}_k} \tr\ve \right)_{L^2(\Omega)} = (\fsf,\tr\ve)_{L^2(\Omega)} ,
       \]
       \quad for all $\ve \in \HL(y^{\alpha},\C_\tau)$.
\State \quad Set 
       \[
        \mu_{k+1} = \left\{\begin{array}{ccc}
                            0 & \mbox{on} & \mathcal{I}_k ,\\
                            \overline{\mu} + \theta (\tr \ue_{k+1} - \Psi) & \mbox{on} & \mathcal{A}_k
                           \end{array}        
                    \right.
       \]
\State \quad $k = k+1$       
\State \textbf{until} some stopping rule is satisfied.       
\end{algorithmic}
\end{algorithm}
The above algorithm converges superlinearly for any $\theta$ 
       provided that iterations are initiated close enough to a solution of the regularized variational inequality (see Theorem 8.25 in \cite{MR2441683}). Further, the sequence of solutions converge, as $\theta\to\infty$, strongly to the solution of the original variational inequality of interest (see Theorem 8.26 in \cite{MR2441683}).

If $\{ \ue_{k+1} \}$ is the sequence generated in Step~\ref{algo:ssn_step} of Algorithm~\ref{algo:ssn}, then the stopping criterion considered is satisfied as soon as: $\mathcal{A}_{k+1} = \mathcal{A}_{k}$ or $$\frac{\| \ue_{k+1} - \ue_{k} \|_{\HL(y^\alpha,\C_\tau)}}{\| \ue_{k} - \ue_{k-1} \|_{\HL(y^\alpha,\C_\tau)}} < \varepsilon_2,$$ or $k = k_{\max}$.

\subsection{Discretization}

The discretization of the linear equation \eqref{pde} was introduced and analyzed in \cite{RHNochetto_EOtarola_AJSalgado_2014a}, see also \cite{MR3393323} for an application to the fractional obstacle problem. Owing to the singular behavior of the solution towards $\Omega \times \{0\}$ it is preferable to use anisotropic meshes, to compensate the singular behavior. In our context such meshes can be defined as follows: Let $\mathscr{T}_\Omega=\{E\}$ be a conforming and quasi-uniform triangulation of $\Omega$, where $E\subset  \mathbb{R}^n$ is an element that is isoparametrically equivalent to either to the unit cube or to the unit simplex in $\mathbb{R}^n$. We assume $\# \mathscr{T}_\Omega \propto M^n$. Thus, the element size $h_{\mathscr{T}_\Omega}$ fulfills $h_{\mathscr{T}_\Omega}\propto M^{-1}$. Furthermore, let $\mathcal{I}_\tau=\{I_k\}_{k=0}^{M-1}$, where $I_k = [y_k,y_{k+1}]$, be a graded mesh of the interval $[0,\tau]$ in the sense that $[0,\tau]=\bigcup_{k=0}^{M-1} I_k$ with
\[
y_k=\left(\frac{k}{M}\right)^\gamma\tau,\quad k=0,\ldots,M,\quad \gamma>3/(1-\alpha)=3/(2s)>1.
\]
We construct the triangulations $\mathscr{T}_\tau$ of the cylinder $\mathcal{C}_\tau$ as tensor product triangulations by using $\mathscr{T}_\Omega$ and $\mathcal{I}_\tau$. Let $\mathbb{T}$ denotes the collection of such 
anisotropic meshes $\mathscr{T}_\tau$.

For each $\mathscr{T}_\tau \in \mathbb{T}$ we define the finite element space $\mathbb{V}(\mathscr{T}_\tau)$ as 
\[
\mathbb{V}(\mathscr{T}_\tau):=\{V\in C^0(\overline{ \mathcal{C}_\tau}):V|_{T}\in\mathcal{P}_1(E)\oplus\mathbb{P}_1(I)\ \forall T=E\times I\in \mathscr{T}_\tau,\ V|_{ \partial_L \C_{\tau}}=0\}.
\]
In case $E$ is a simplex then $\mathcal{P}_1(E)=\mathbb{P}_1(E)$, the set of polynomials of degree at most $1$. If $E$ is a cube then $\mathcal{P}_1(E)$ equals $\mathbb{Q}_1(E)$, the set of polynomials of degree at most 1 in each variable. In our numerical illustrations we shall work with simplices. 


For our numerical examples we consider $n = 2$, $\Omega = (0,1)^2$, $c(x) = 0$, and $A(x) = 1$ in \eqref{trunc:QVI}. We set the force $\fsf(x_1,x_2) = x_1(1-x_1)x_2(1-x_2)$ in first three examples and $\fsf=1$ in the final example. In Algorithm~\ref{Algorithm} we choose a fixed ``large" $\tau$ defined as $\tau = 1 + \frac{1}{3}\log(\#\mathscr{T}_\Omega)$. Such a choice is motivated by the linear equation where it leads to error balance between the truncation $(\tau)$ and the finite element approximation; see \cite[Remark 5.5]{RHNochetto_EOtarola_AJSalgado_2014a}. We further set total degrees of freedom of $\mathscr{T}_\tau$ equal to 12716.

Next we shall study four examples. In all cases we set $\varepsilon_1 = 5e-4$, $n_{\max} = 150$ (see  \eqref{eq:eps1}). Moreover, we set $\varepsilon_2 = 1e-2$, $k_{\max} = 10$ and $\overline{\mu} = 0$ in Algorithm~\ref{algo:ssn}. We notice that the total number of iterations, using a continuation technique for the parameter $\theta$, remained stable under mesh refinements (cf. Algorithm~\ref{algo:ssn}). In our computations, we set $\theta_{\max} = 1e+10$ and increase $\theta$, starting with $10$, such that the ratio between two consecutive values of $\theta$ is 1.5.

Note that the $\Psi$ maps in Examples 1 and 2, both satisfy $(i)$ in Assumption \ref{ass1}, so existence of solutions for the QVI problems is guaranteed. Further, while {$\Psi$ in} Example 2 satisfies Assumption \ref{ass2}, Example 1 only satisfies $(i)$ in Assumption \ref{ass2}. 
Also, Example~3 satisfies Assumption \ref{ass1} and $(i)$ in Assumption \ref{ass2}.
Finally, $\Psi$ in Example 4 has the same structure as Example~2.

\subsection{Example 1}\label{s:ex1}

We first consider the case where the obstacle $\Psi(\usf)$ is given by 
\[
 \Psi(\usf)(x_1,x_2) = 5\left(\sin(x_1) \usf(x_1,x_2)\right)^+ + \delta, 
\]
where $\delta = 1e-10$.
Figures~\ref{f:ex1_sol}, \ref{f:ex1_obs}, and \ref{f:ex1_act} illustrate the final solution, the obstacle, and the active set respectively for different values of $s = 0.2, 0.4, 0.6$, and $s = 0.8$. We clearly notice the solution dependence on $s$. In each case we observe that it takes between 5 to 10 iterations for Algorithm~\ref{algo:ssn} to converge. On the other hand it takes $n = 49, 47, 44$ and $n=42$ iterations for us to achieve the criterion in \eqref{eq:eps1} for $s = 0.2, 0.4, 0.6$ and $s = 0.8$ respectively.

\begin{figure}[h!]
\includegraphics[width=0.45\textwidth]{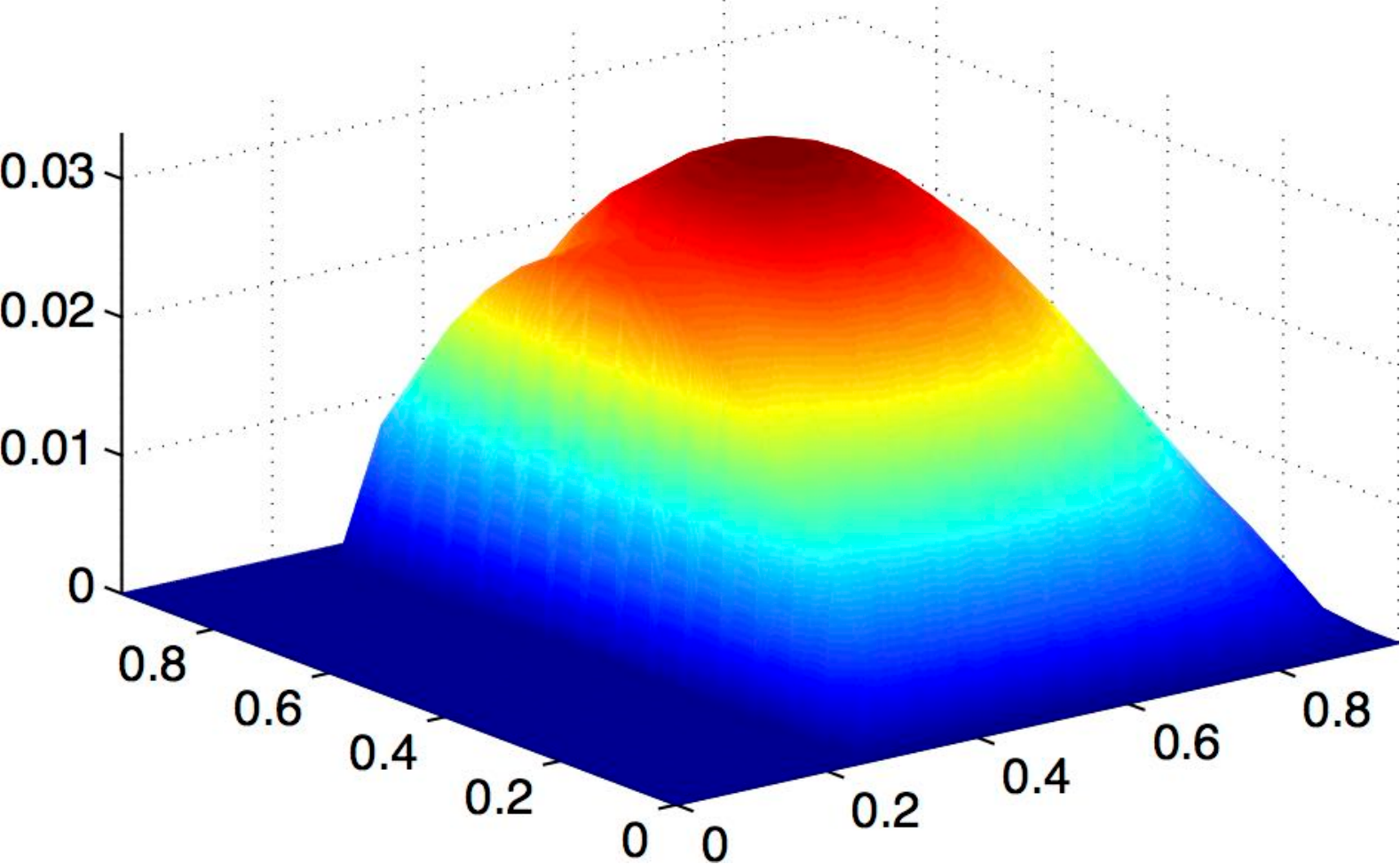}
\includegraphics[width=0.45\textwidth]{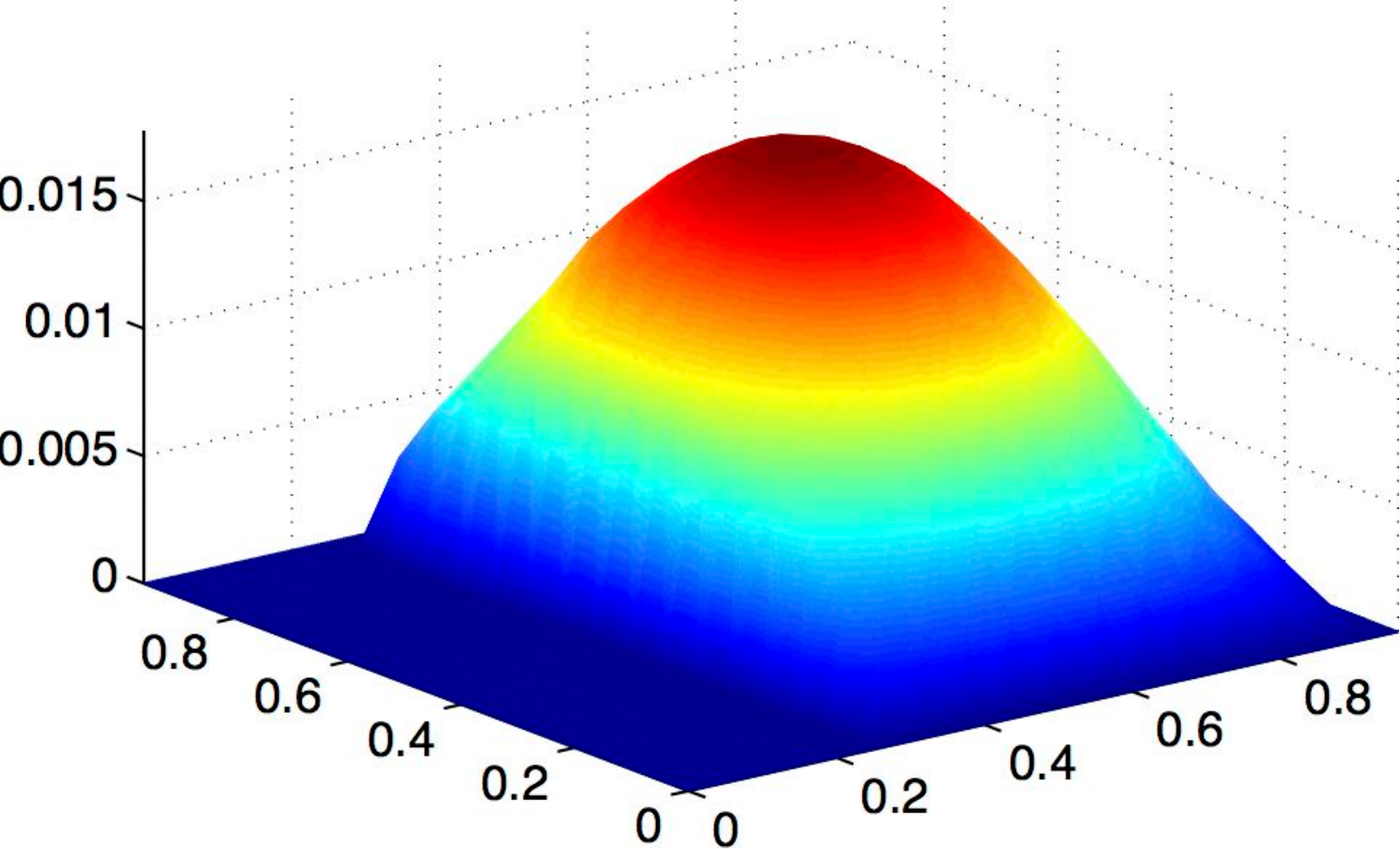}
\includegraphics[width=0.45\textwidth]{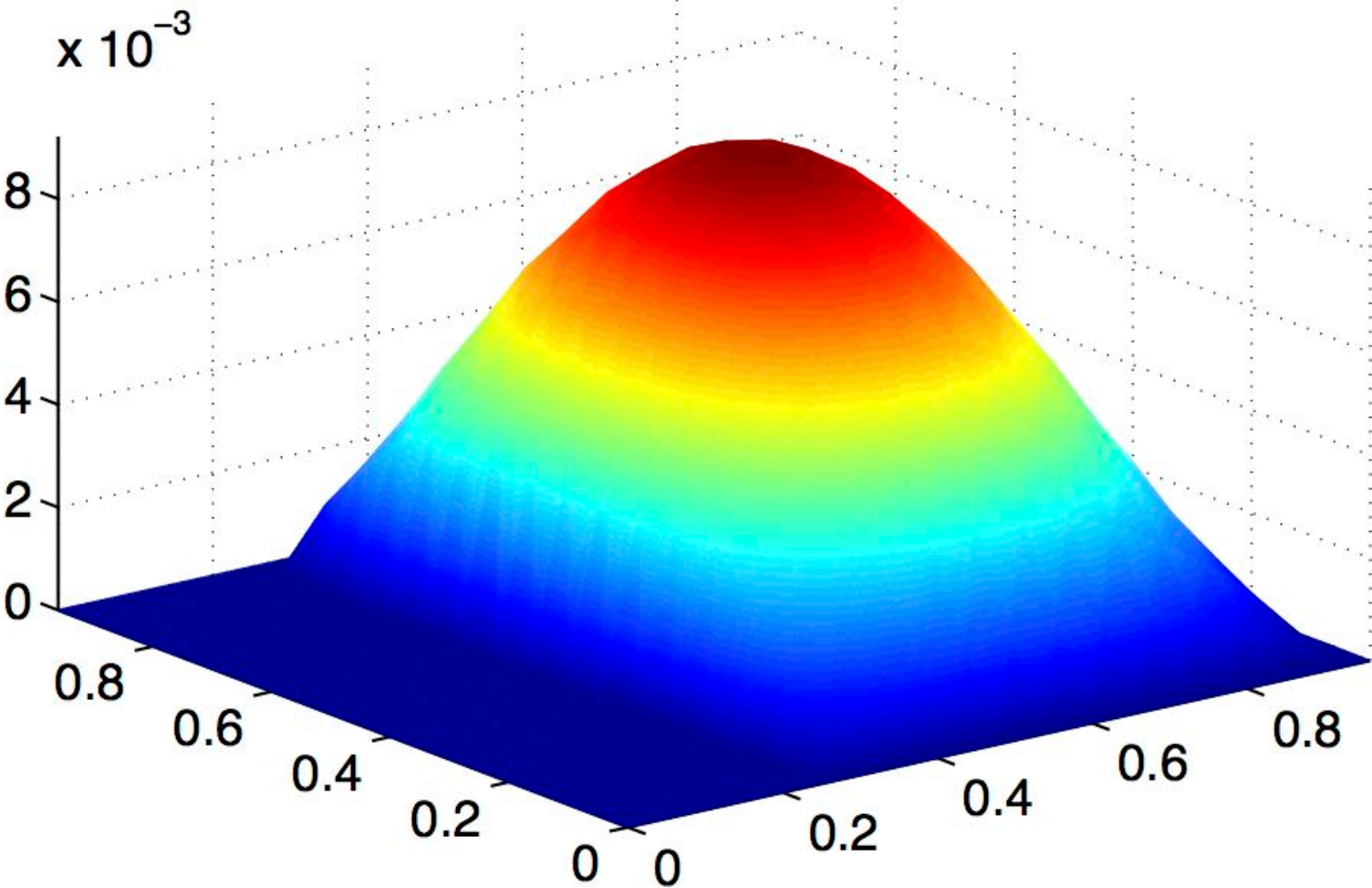}\quad\quad\quad
\includegraphics[width=0.45\textwidth]{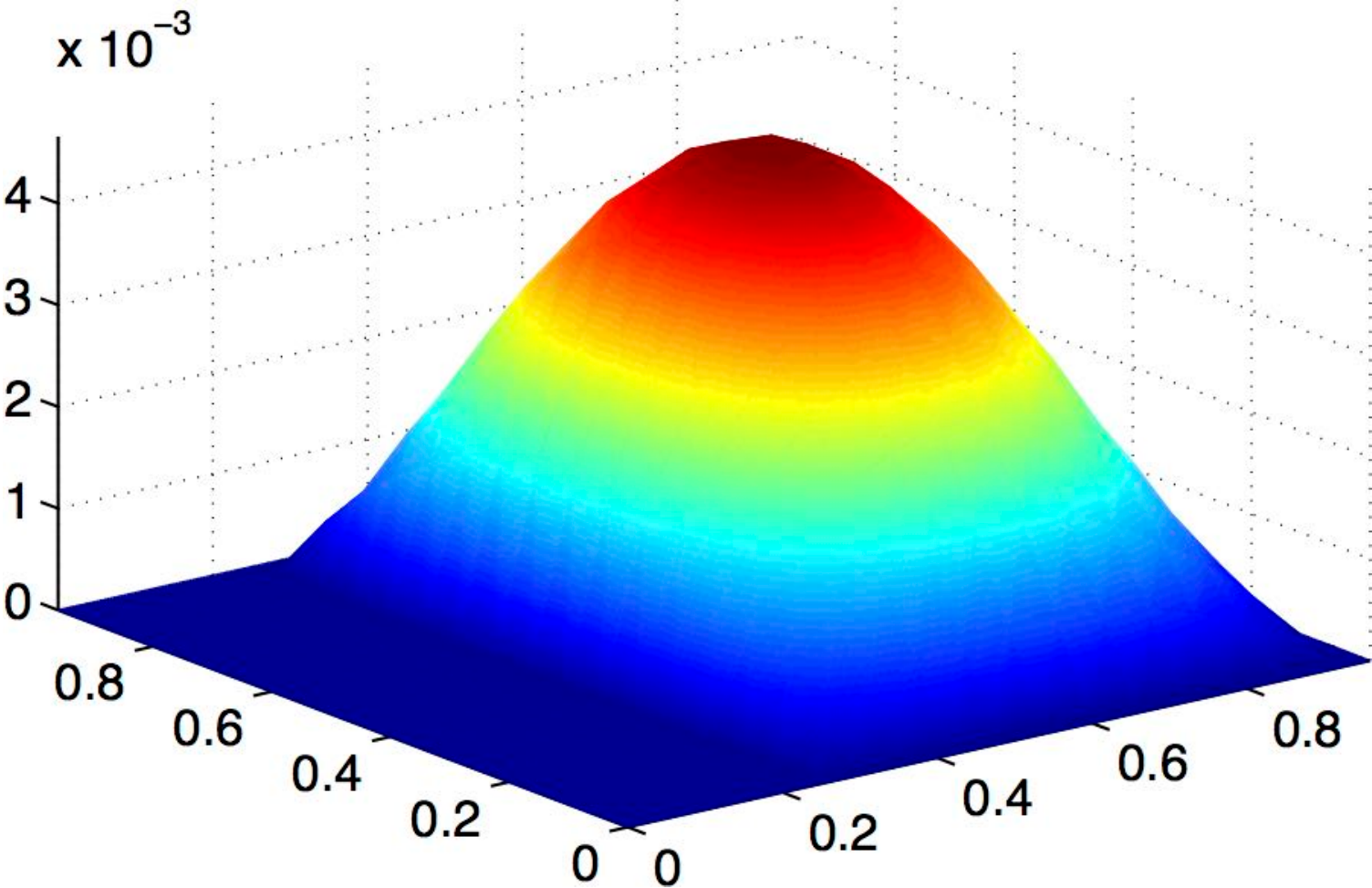}
\caption{\label{f:ex1_sol}Example 1: The top panels illustrate the solution for $s = 0.2$ (left) and $s =0.4$ (right). On the other hand, the bottom panels show the solutions when $s = 0.6$ (left) and $s = 0.8$ (right). The dependence on $s$ is clearly visible.}
\end{figure}

\begin{figure}[h!]
\includegraphics[width=0.45\textwidth]{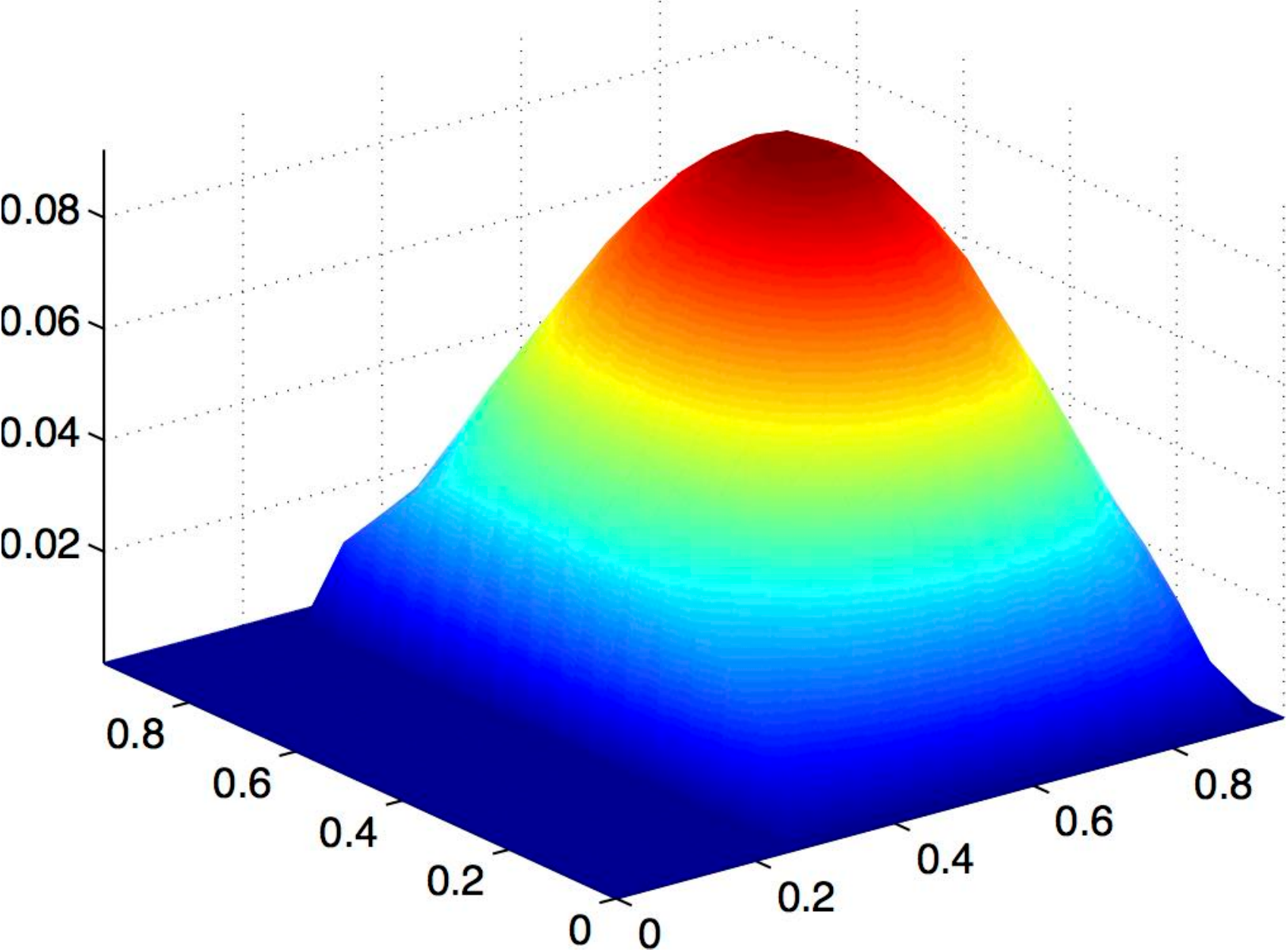}
\includegraphics[width=0.45\textwidth]{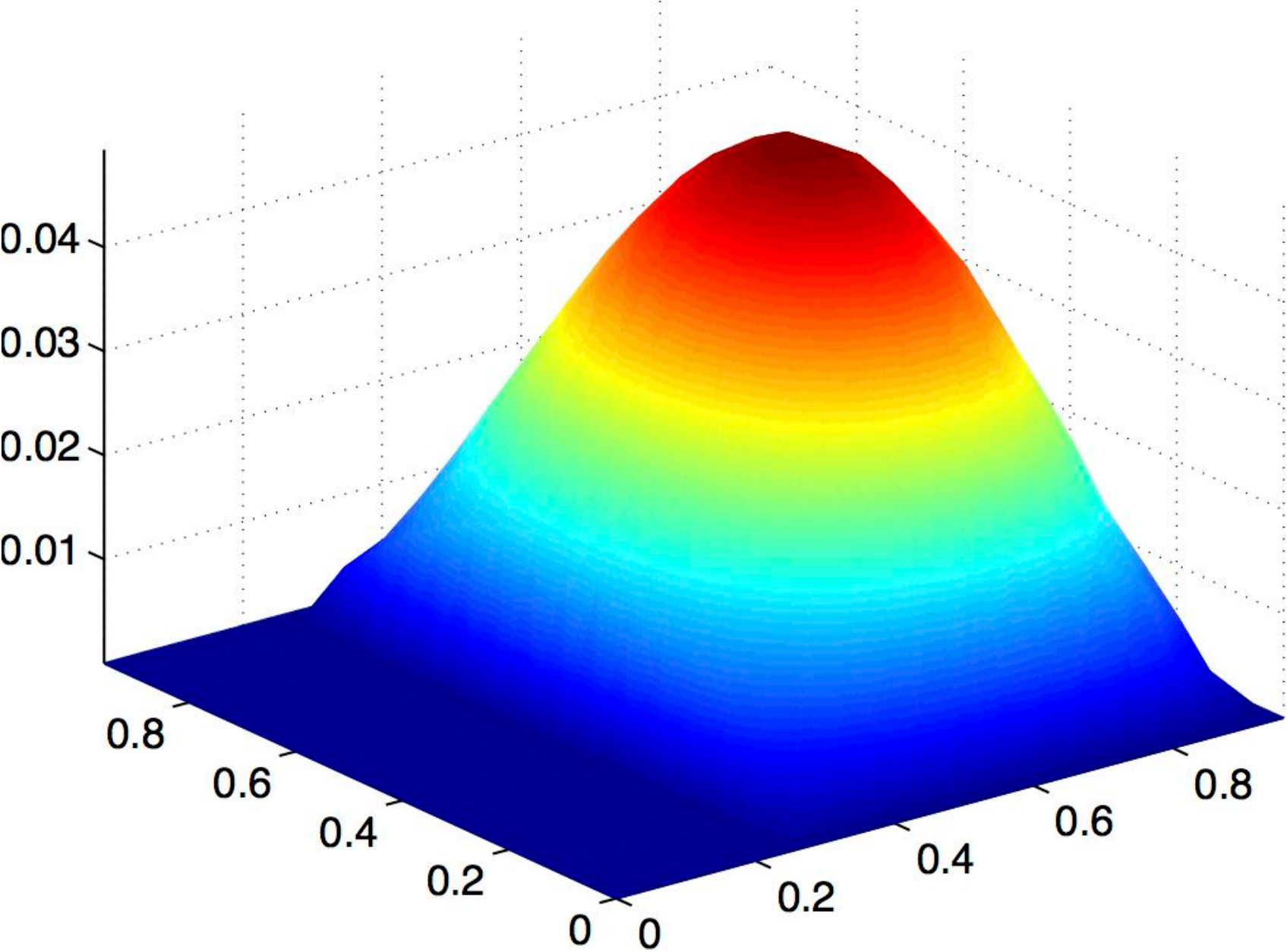}
\includegraphics[width=0.45\textwidth]{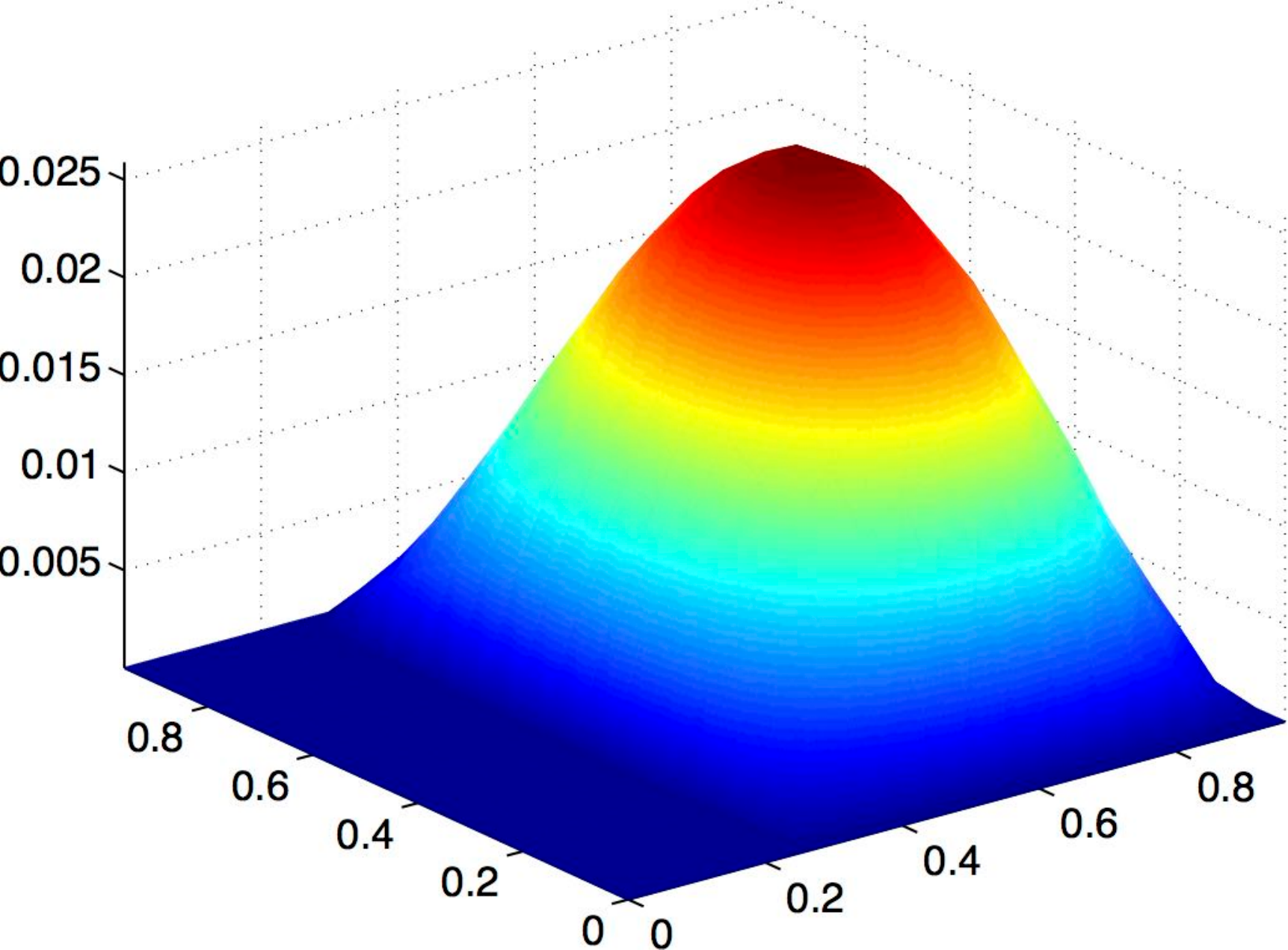}\quad\quad\quad
\includegraphics[width=0.45\textwidth]{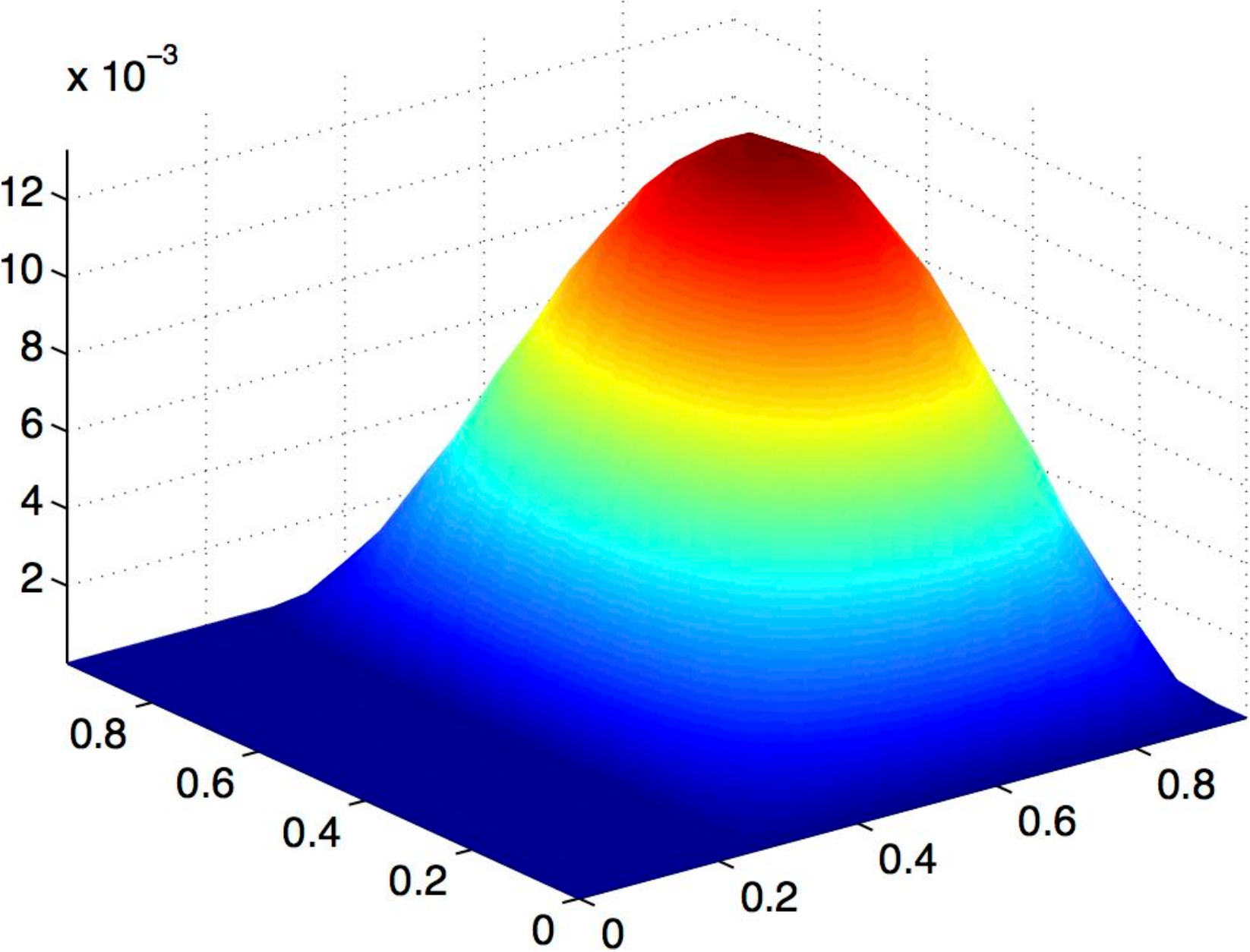}
\caption{\label{f:ex1_obs}Example 1: The top panels illustrate the obstacle $\Psi$ for $s = 0.2$ (left) and $s =0.4$ (right). On the other hand, the bottom panels show the obstacle when $s = 0.6$ (left) and $s = 0.8$ (right). The dependence on $s$ is apparent.}
\end{figure}

\begin{figure}[h!]
\includegraphics[width=0.245\textwidth]{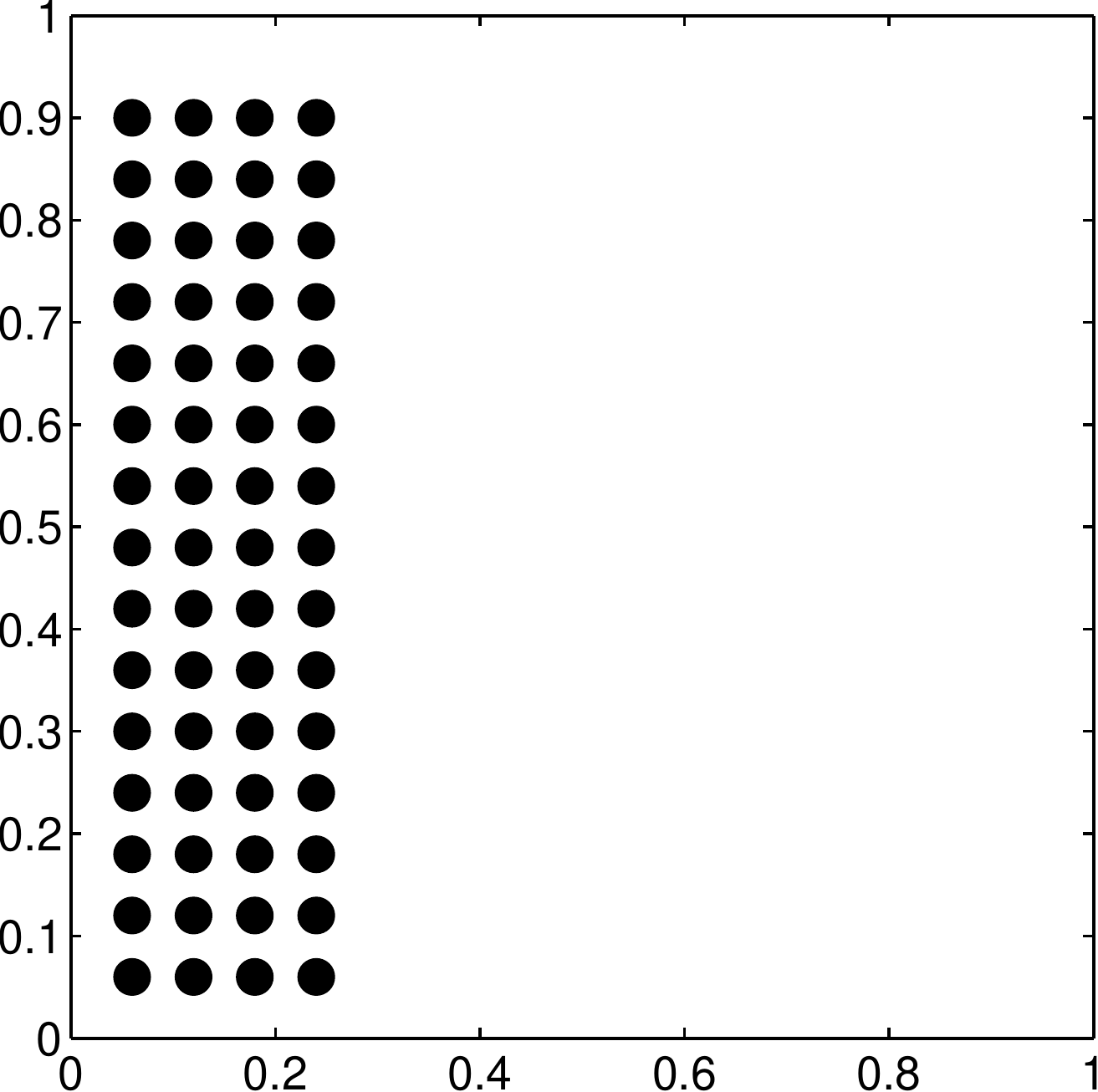}
\includegraphics[width=0.245\textwidth]{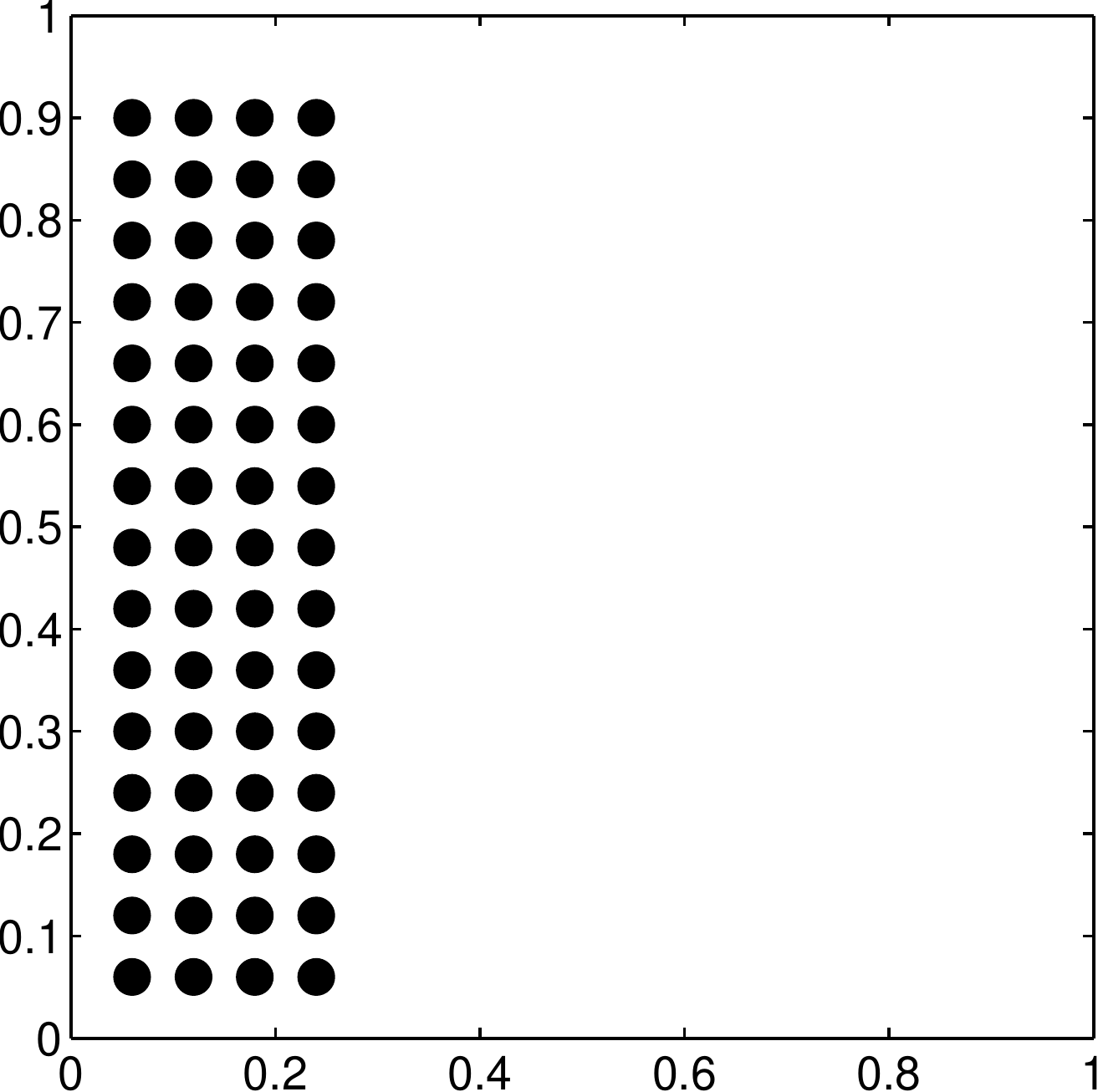}
\includegraphics[width=0.245\textwidth]{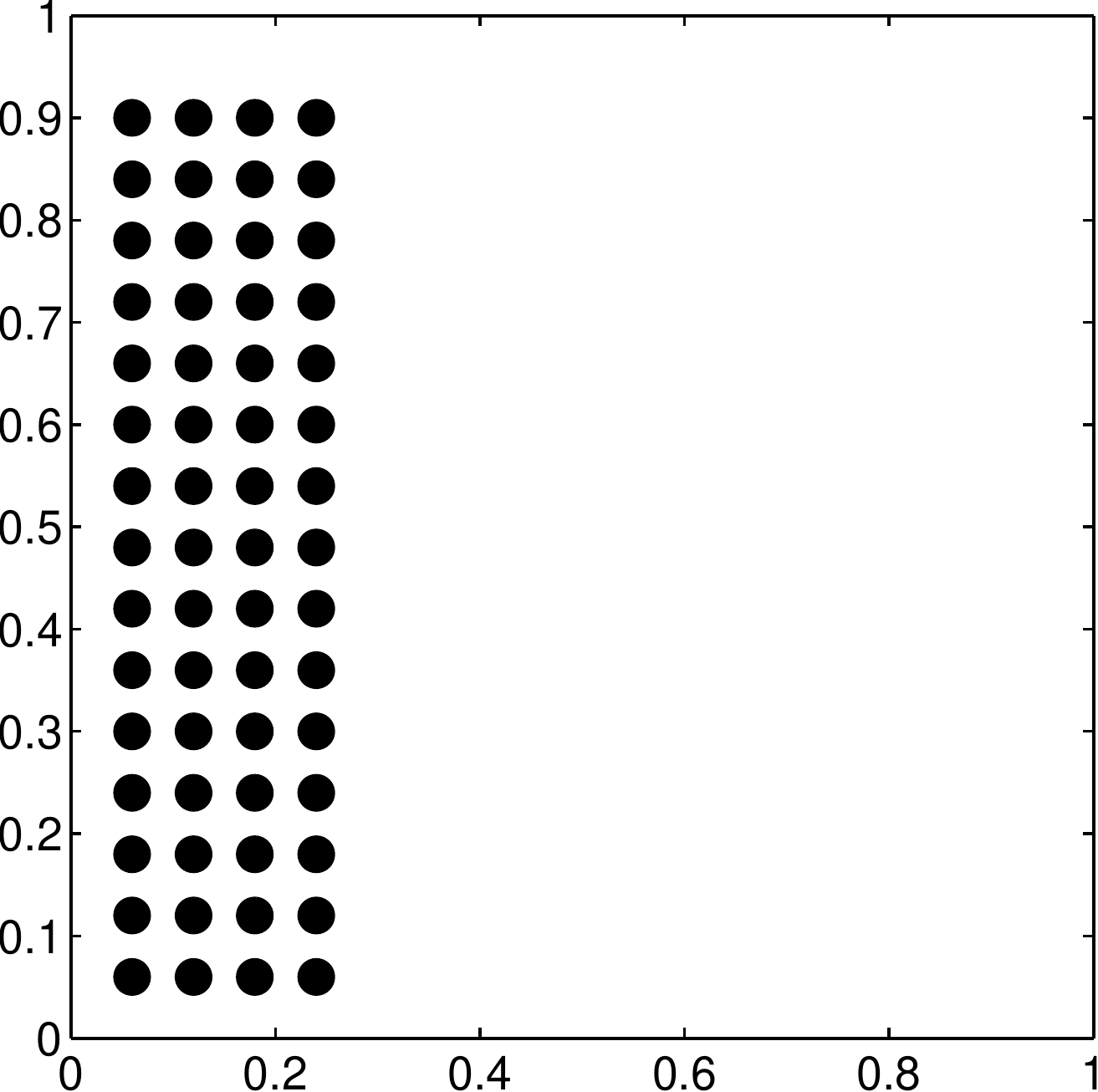}
\includegraphics[width=0.245\textwidth]{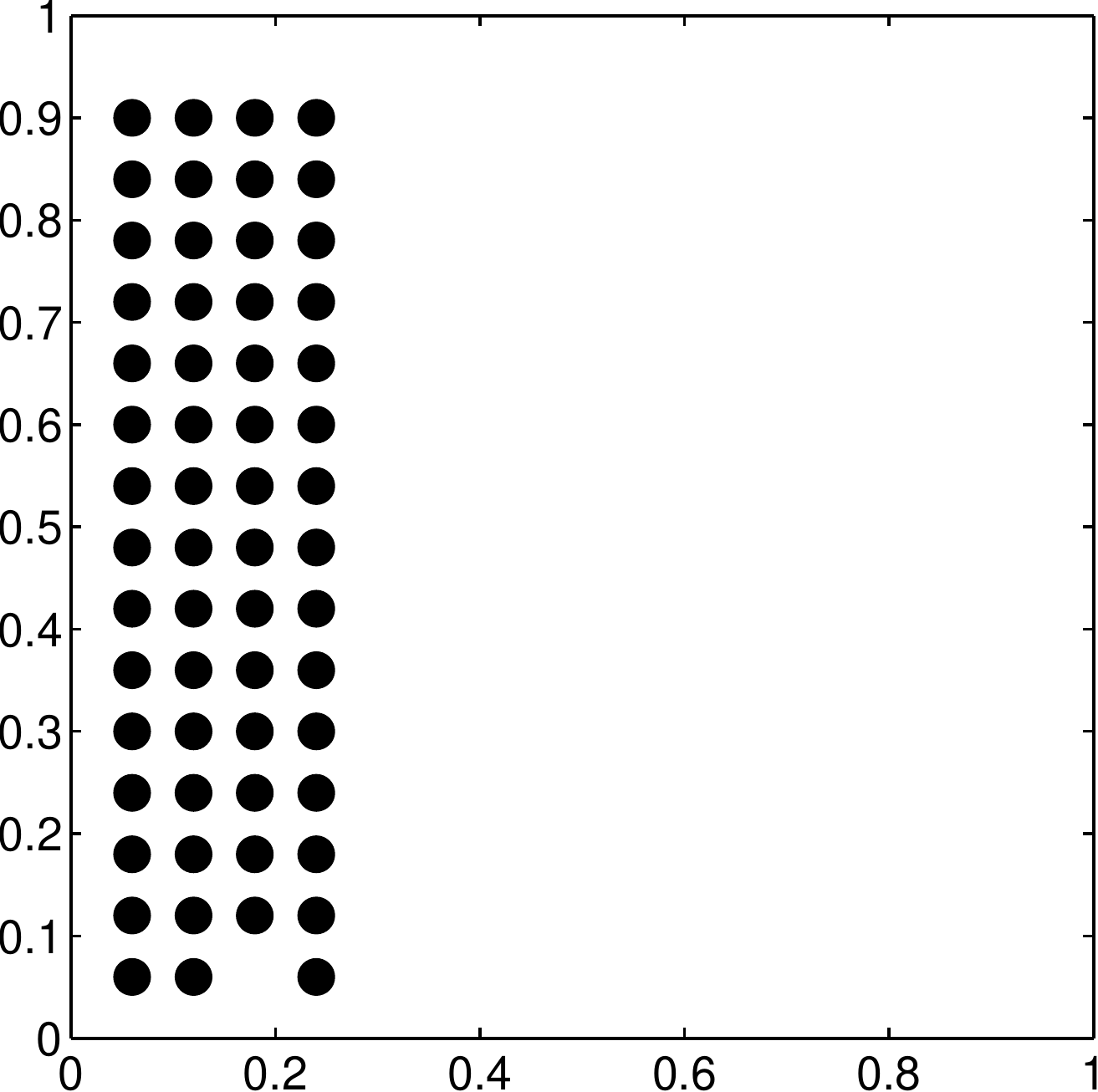}
\caption{\label{f:ex1_act}Example 1: The panels illustrate the active sets for $s = 0.2$, $s =0.4$, $s =0.6$, and $s =0.8$. 
In this example, the active sets for $s = 0.2$, $s =0.4$, $s =0.6$ look similar.}
\end{figure}

\subsection{Example 2}\label{s:ex2}

As a second example we take a nonlocal $\Psi$, i.e., we set
\[
 \Psi(\usf) = 2 \left|\int_\Omega \usf \ dx \right| + \delta , 
\]
where $\delta = 1e-10$.
Figure~\ref{f:ex2_sol} and \ref{f:ex2_act} illustrate the solution and the active set respectively for different values of $s = 0.2, 0.4, 0.6$, and $s = 0.8$. As the final $\Psi$ is a constant in this case so we decided not to plot it here. Again we clearly notice a different solution behavior with respect to $s$. We further notice that it takes between 5 to 10 iterations for Algorithm~\ref{algo:ssn} to converge. On the other hand it takes $n = 47, 48, 50, 52$ when $s = 0.2, 0.4, 0.6, 0.8$, respectively, for us to achieve the criterion in \eqref{eq:eps1}.

\begin{figure}[h!]
\includegraphics[width=0.45\textwidth]{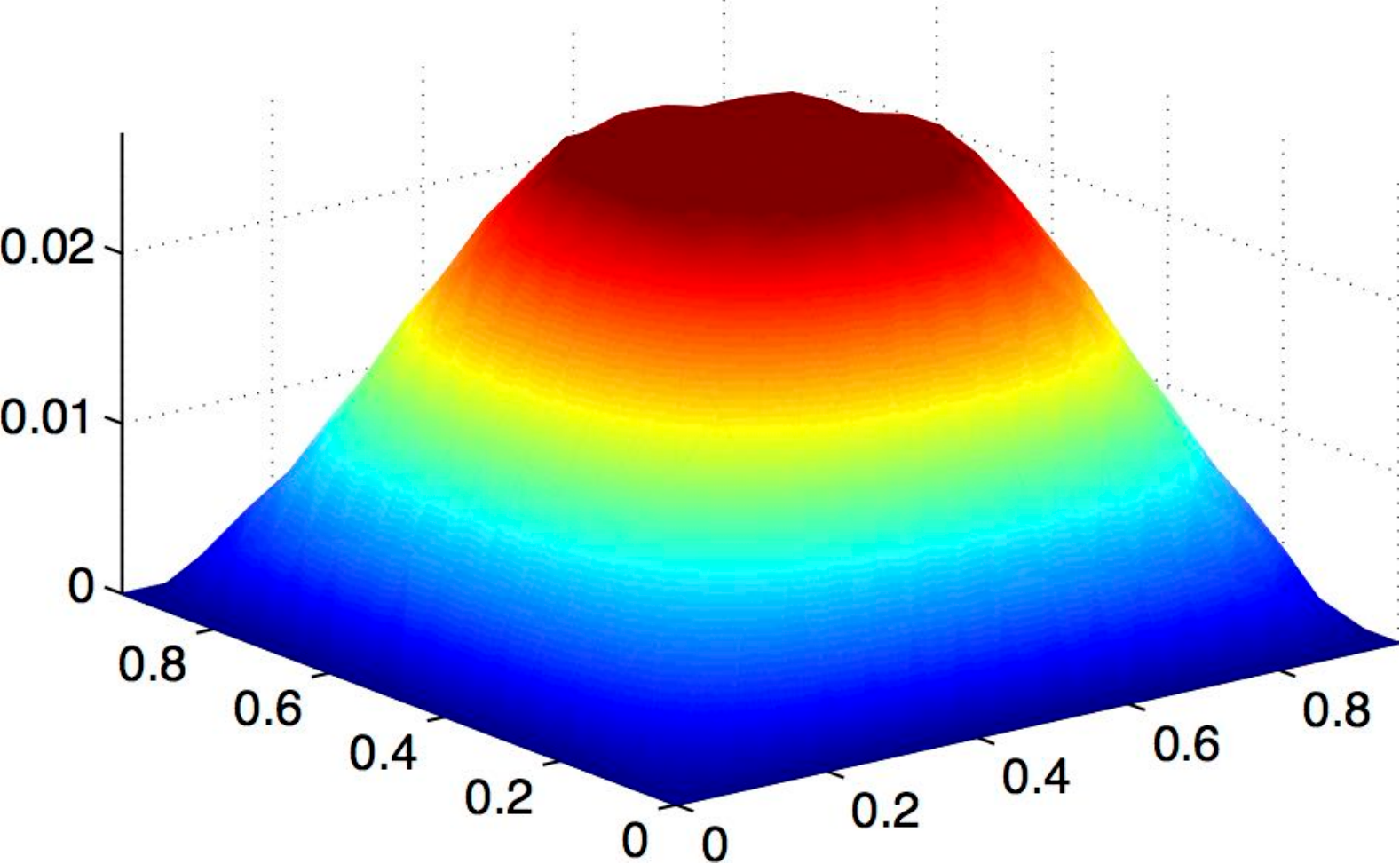}
\includegraphics[width=0.45\textwidth]{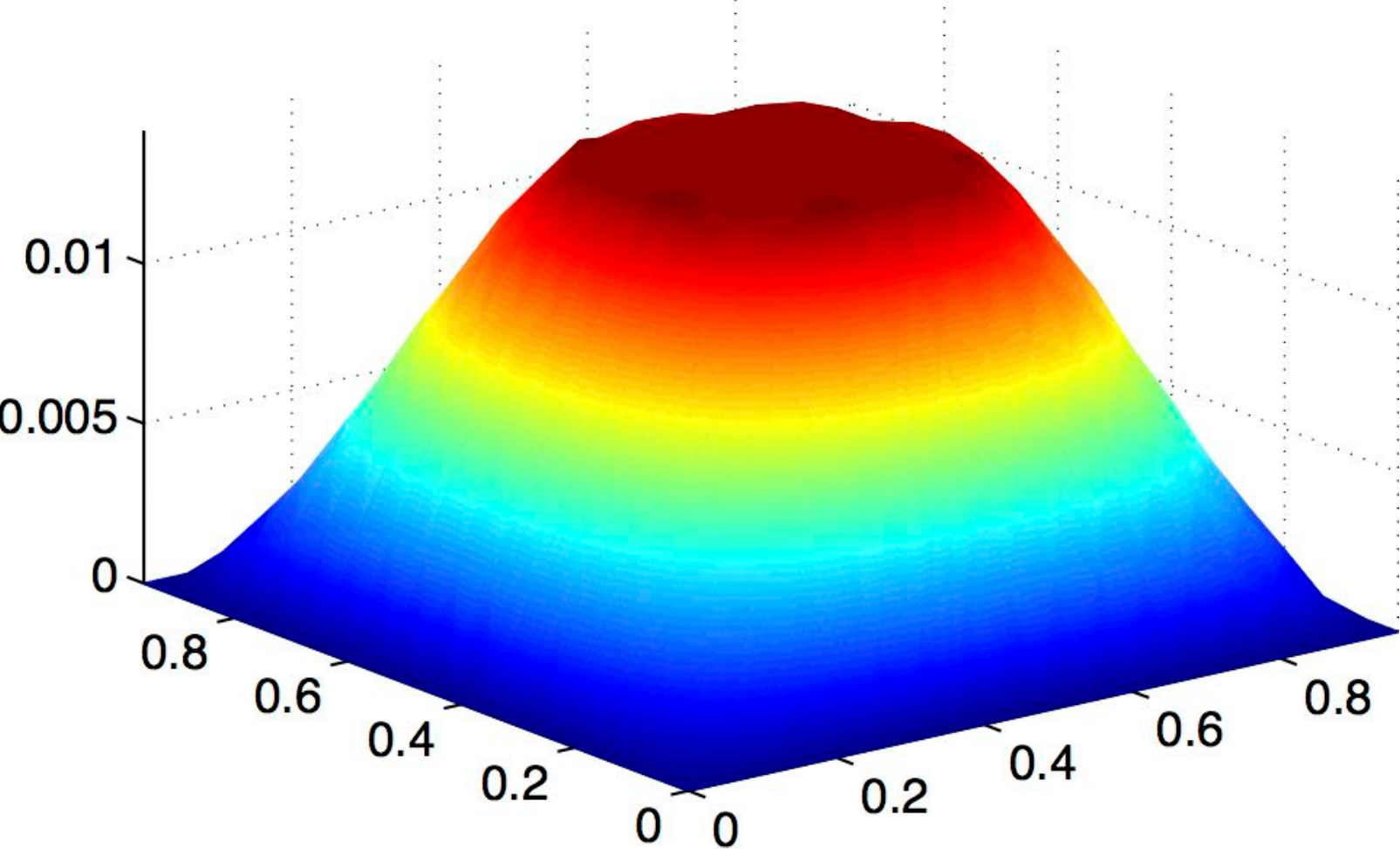}
\includegraphics[width=0.45\textwidth]{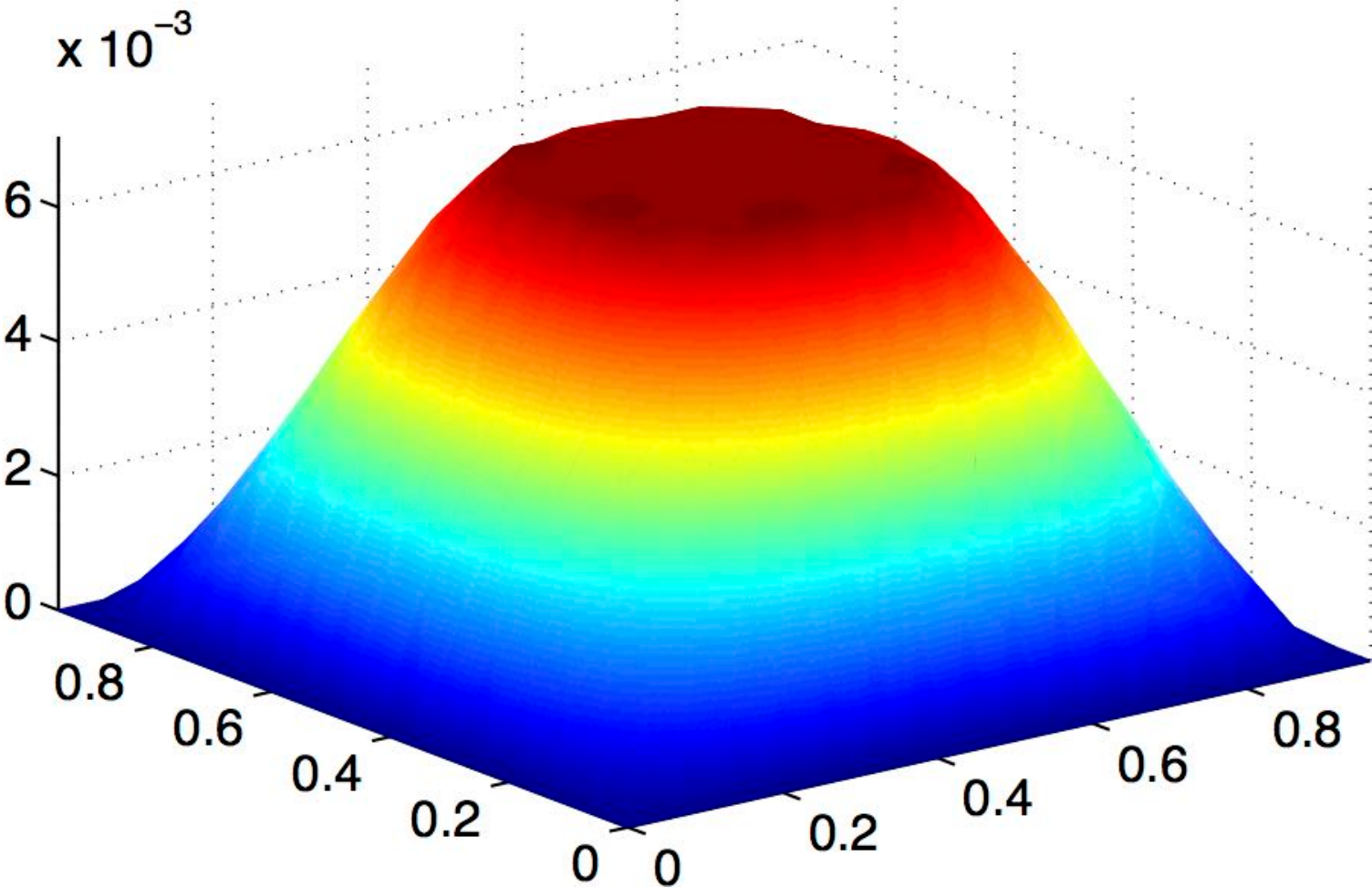}\quad\quad\quad
\includegraphics[width=0.45\textwidth]{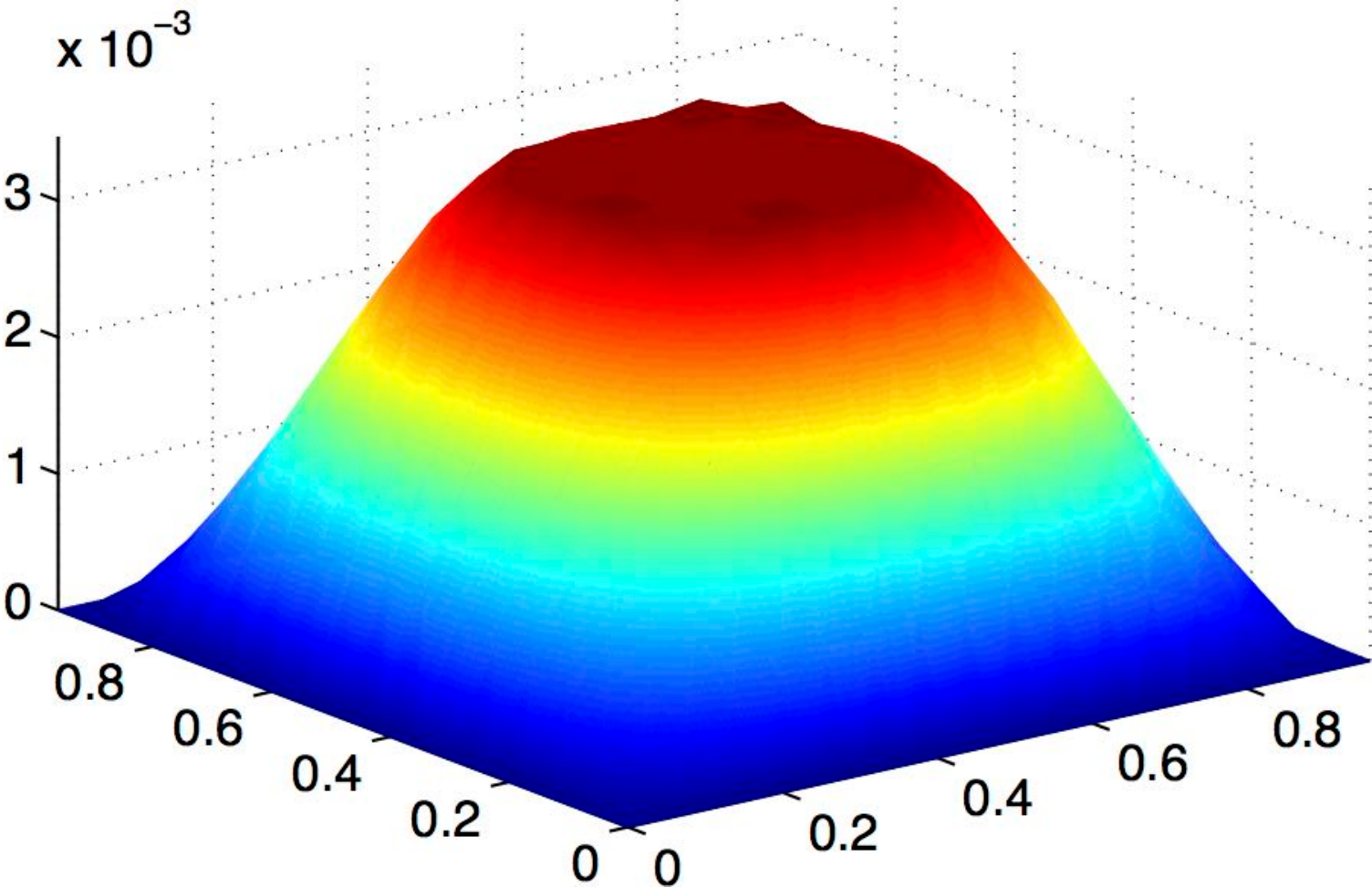}
\caption{\label{f:ex2_sol}Example 2: The top panels illustrate the solution for $s = 0.2$ (left) and $s =0.4$ (right). On the other hand, the bottom panels show the solutions when $s = 0.6$ (left) and $s = 0.8$ (right).}
\end{figure}

\begin{figure}[h!]
\includegraphics[width=0.245\textwidth]{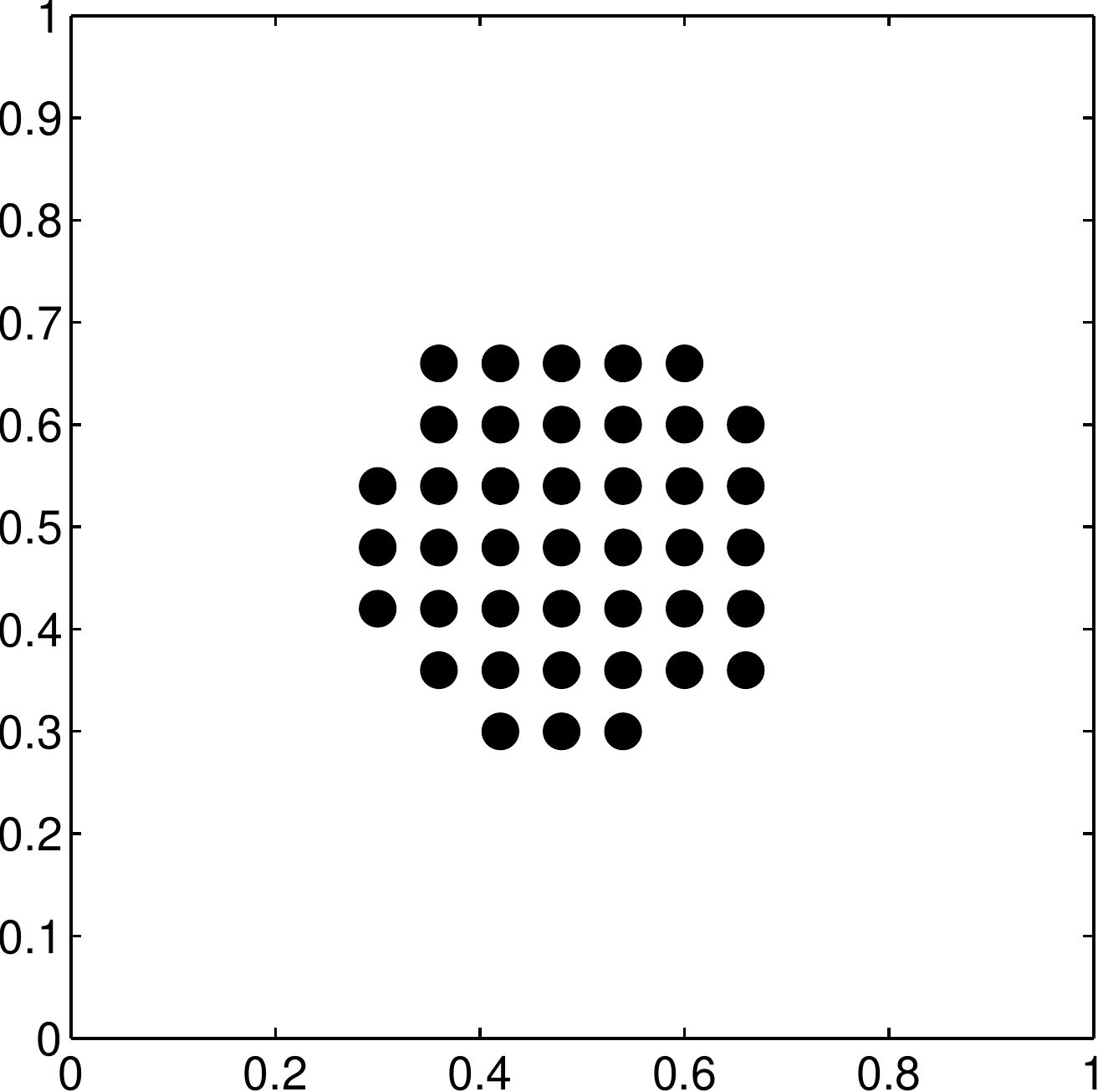}
\includegraphics[width=0.245\textwidth]{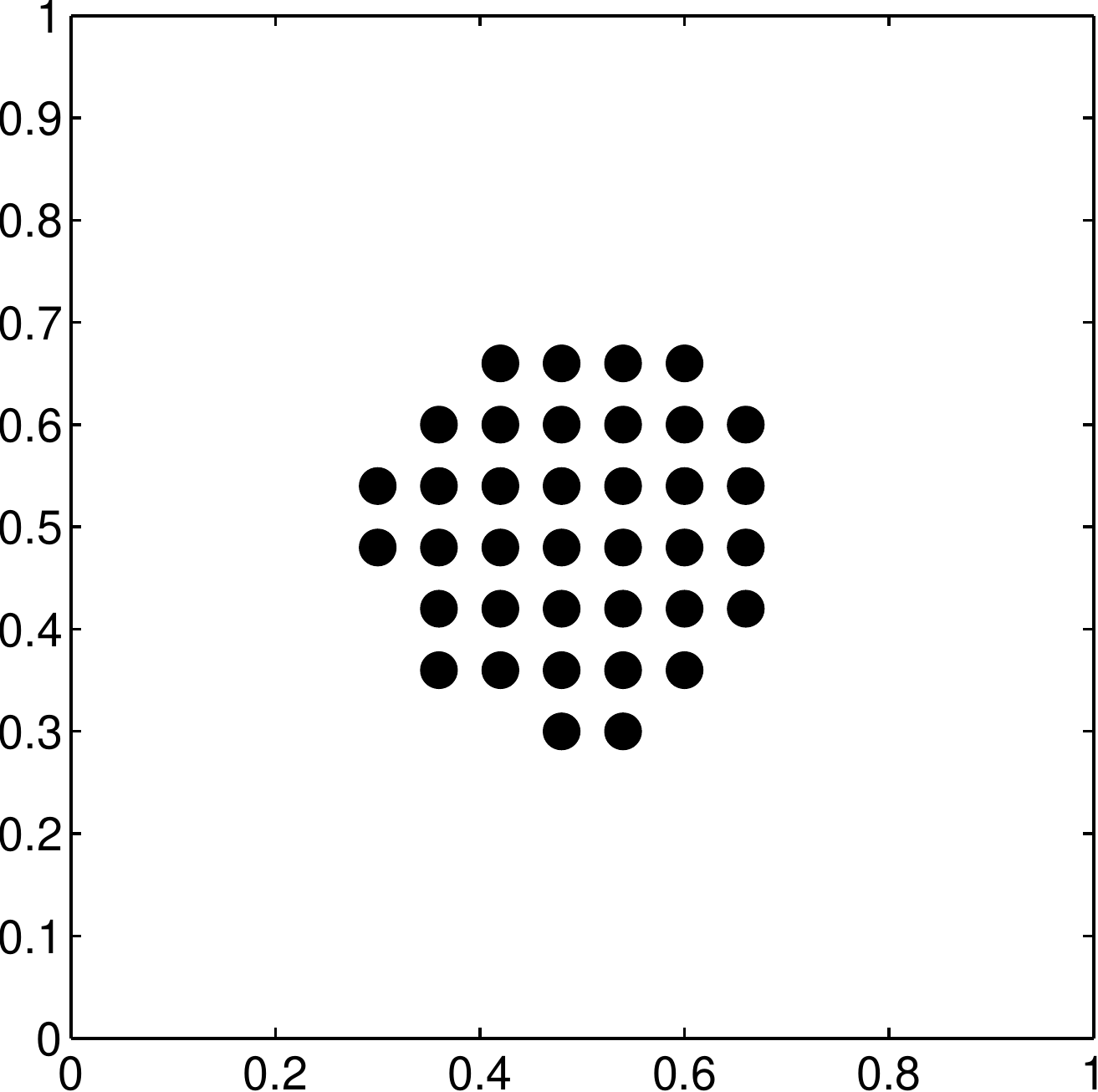}
\includegraphics[width=0.245\textwidth]{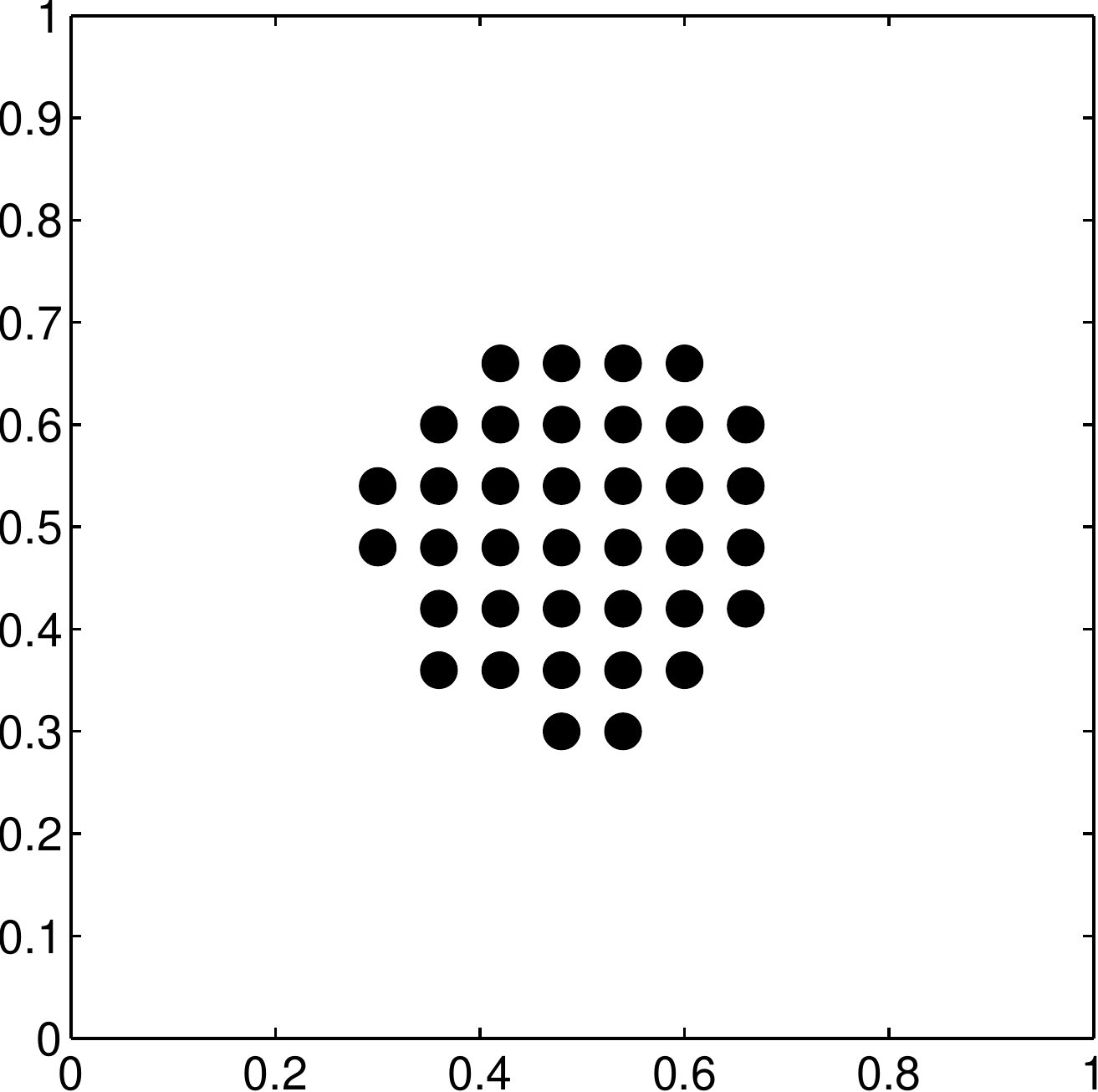}
\includegraphics[width=0.245\textwidth]{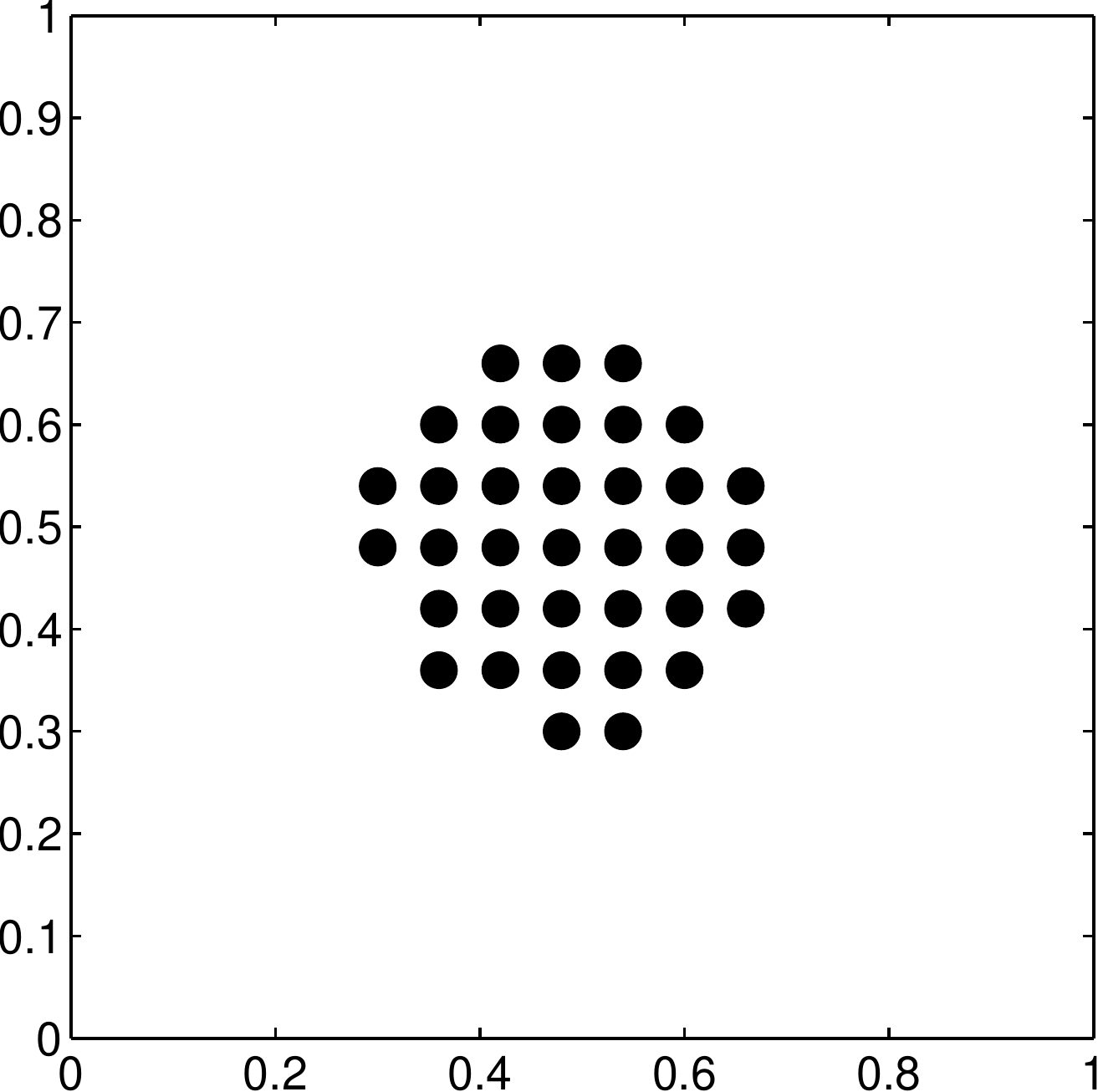}
\caption{\label{f:ex2_act}Example 1: The panels illustrate the active sets for $s = 0.2$, $s =0.4$, $s =0.6$, and $s =0.8$. }
\end{figure}

\subsection{Example 3}\label{s:ex3}

Finally we present an example where $\Psi$ is as given in \eqref{eq:icPsi} with $\nu = 5e-3$. We illustrate the solution in Figure~\ref{f:ex3_sol} and the active set in Figure~\ref{f:ex3_act} for $s = 0.2, 0.4, 0.6$, and 
$s = 0.8$. As in section~\ref{s:ex2} we again observe that the $\Psi$ is a constant under the current configuration therefore we chose not to plot it here. 
\begin{figure}[h!]
\includegraphics[width=0.45\textwidth]{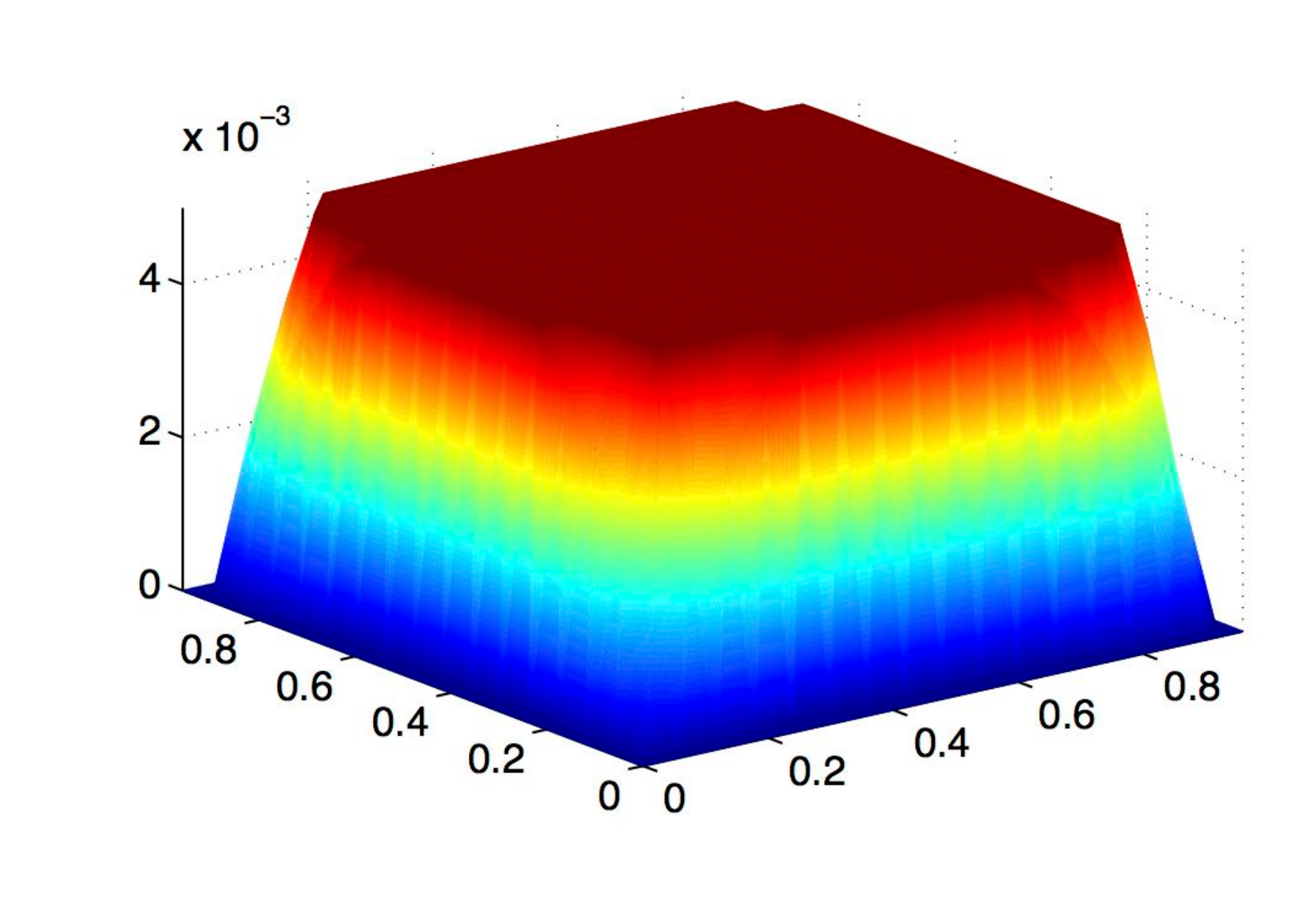}
\includegraphics[width=0.45\textwidth]{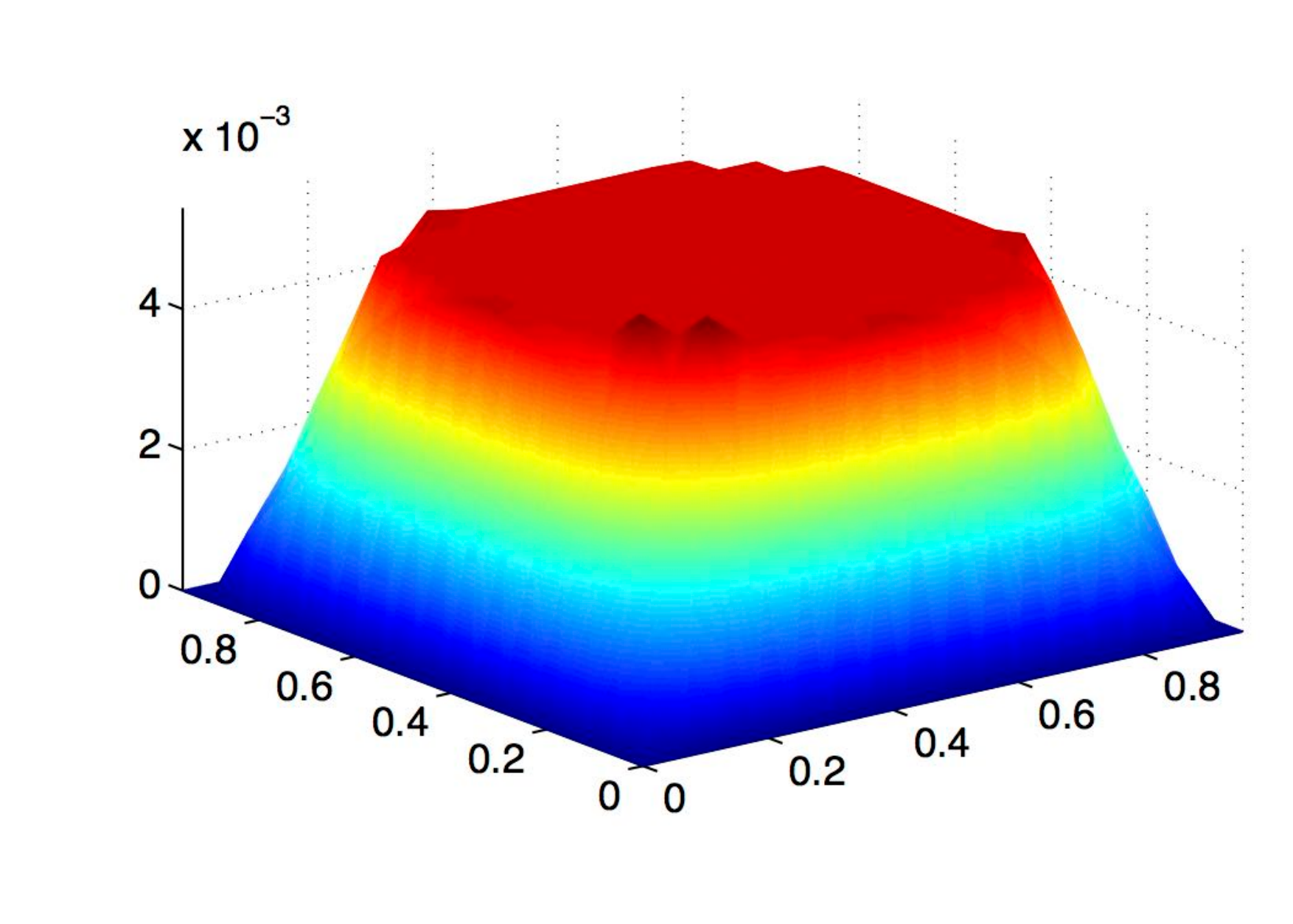}
\includegraphics[width=0.45\textwidth]{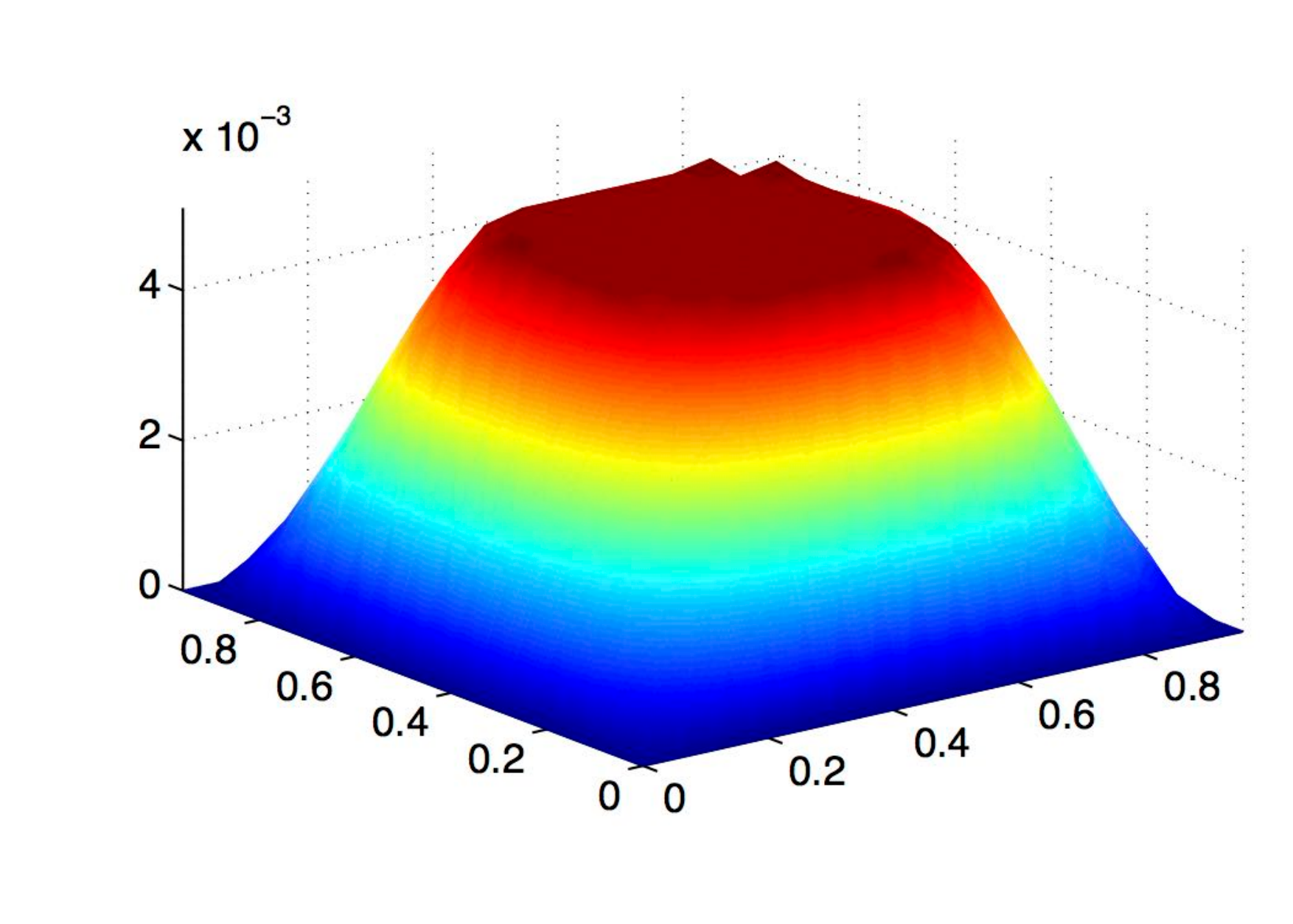}\quad\quad\quad
\includegraphics[width=0.45\textwidth]{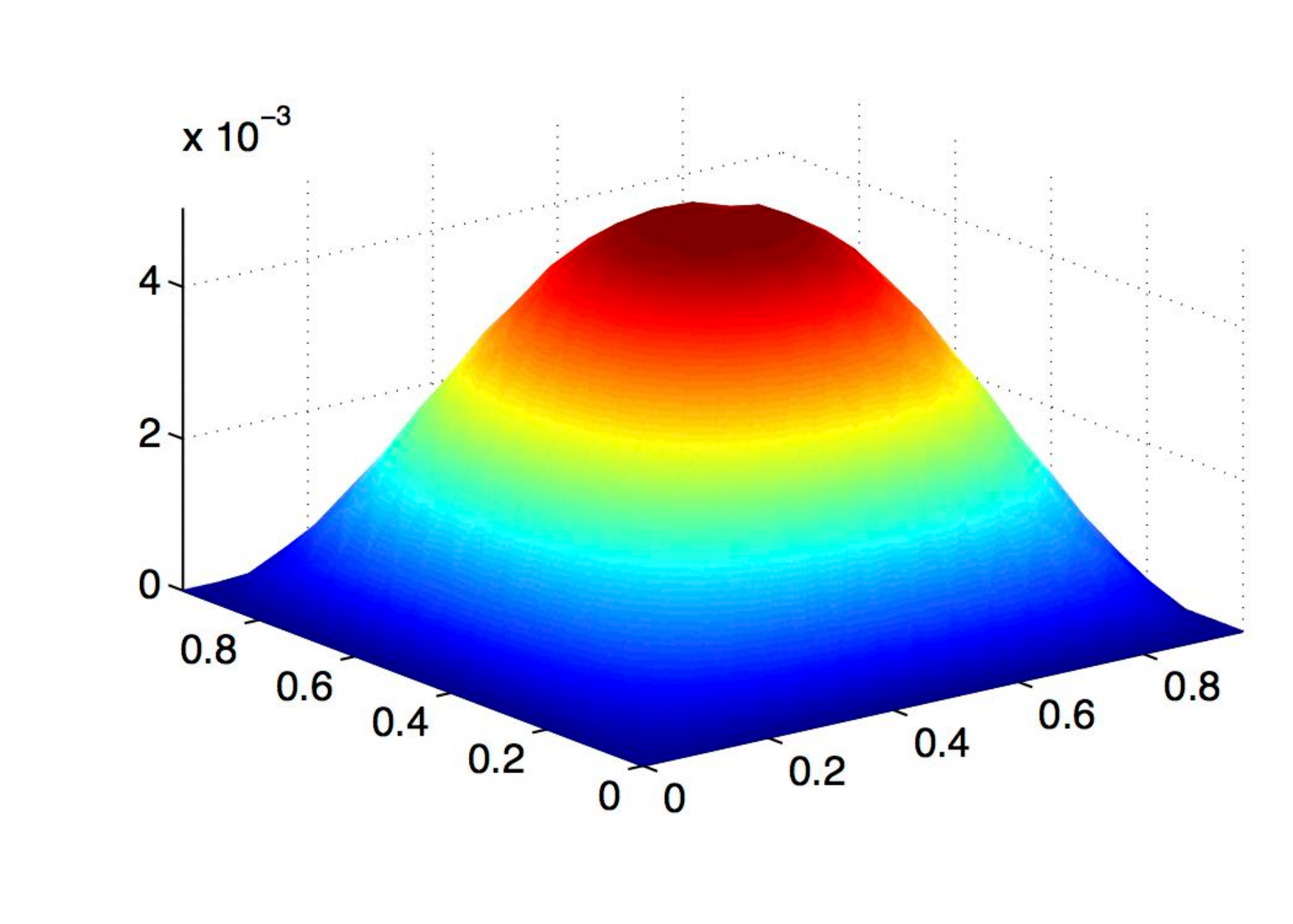}
\caption{\label{f:ex3_sol}Example 3: The top panels illustrate the solution for $s = 0.2$ (left) and $s =0.4$ (right). On the other hand, the bottom panels show the solutions when $s = 0.6$ (left) and $s = 0.8$ (right). The changes with respect to $s$ are significant.}
\end{figure}
\begin{figure}[h!]
\includegraphics[width=0.245\textwidth]{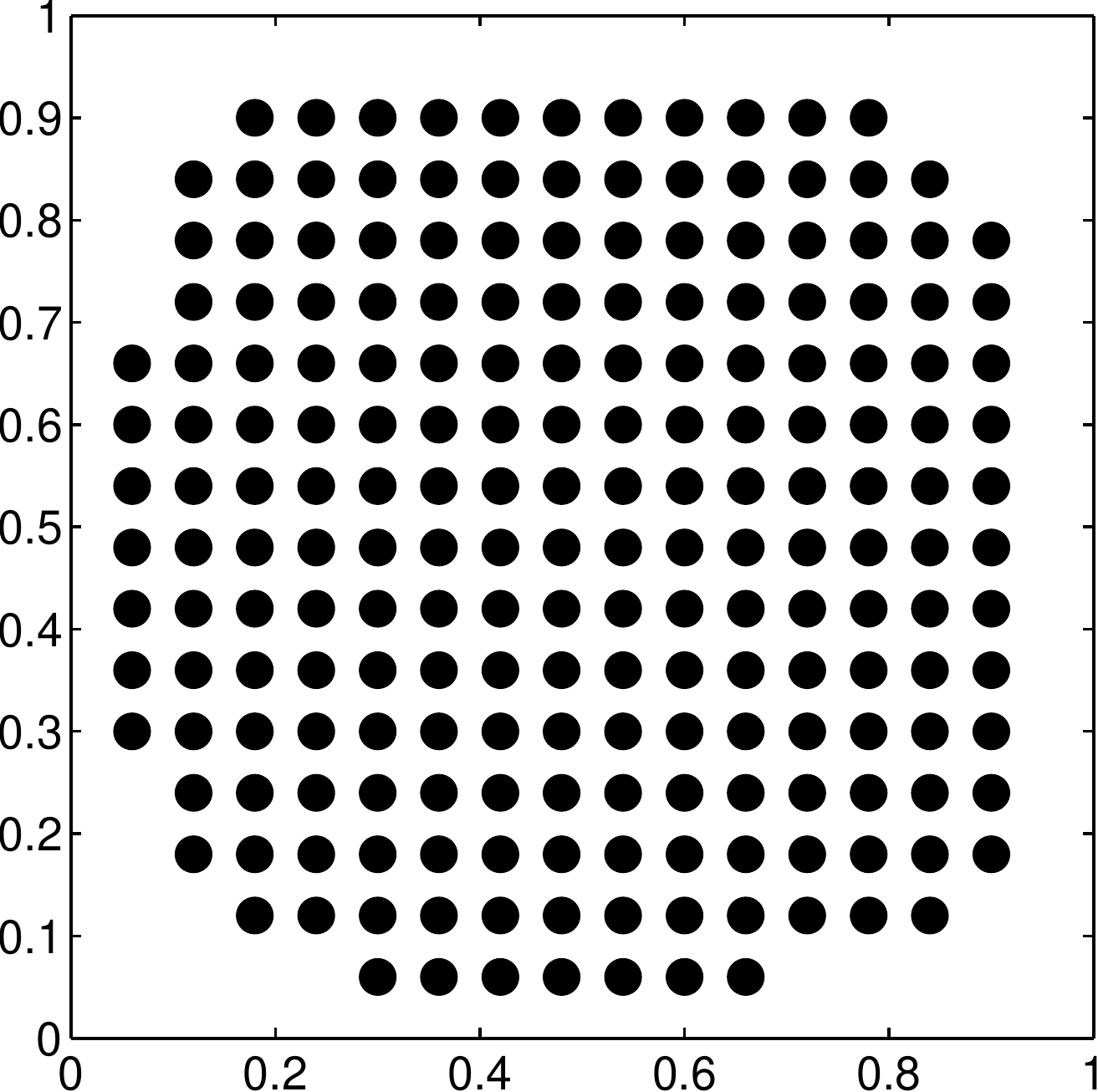}
\includegraphics[width=0.245\textwidth]{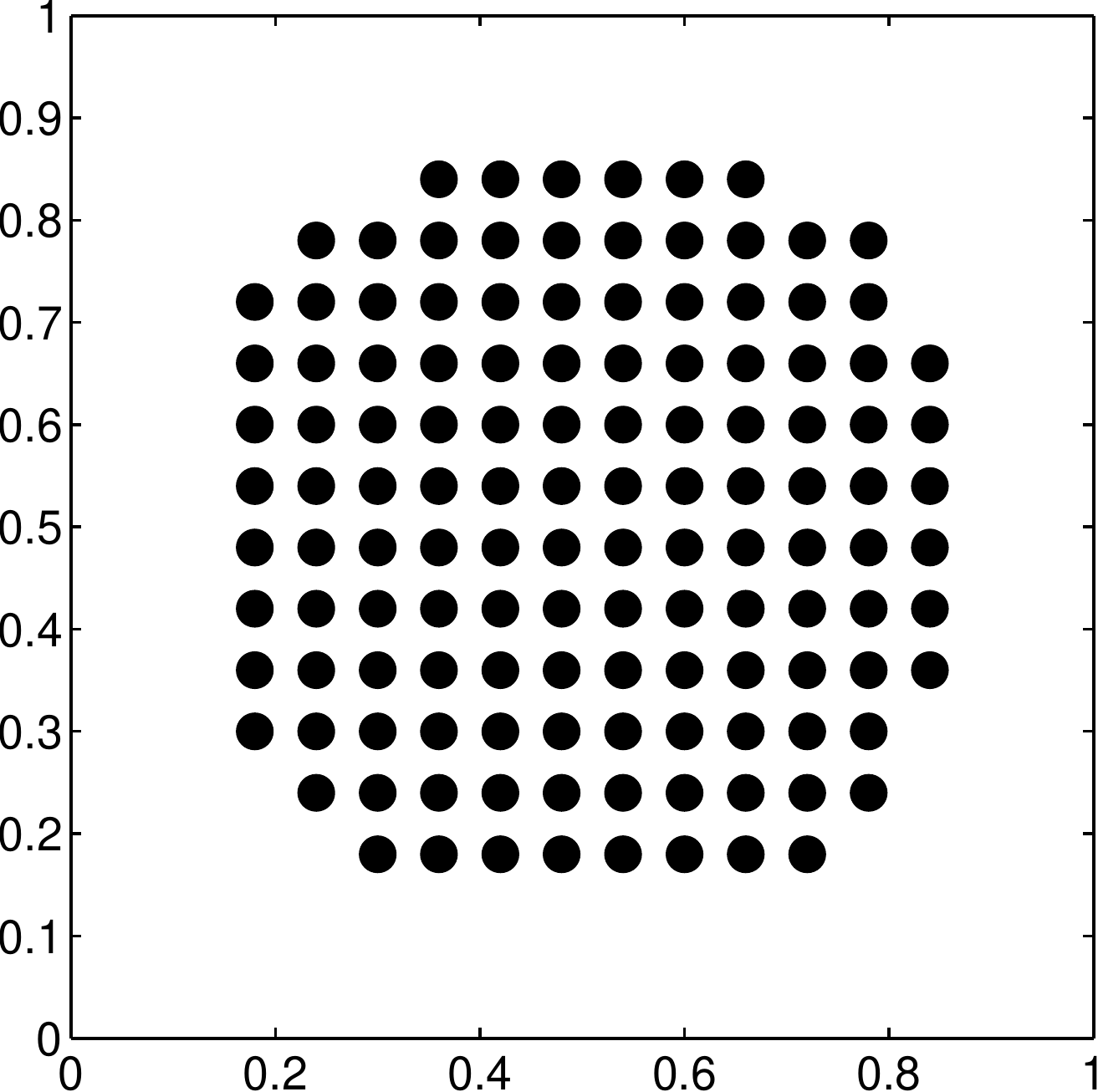}
\includegraphics[width=0.245\textwidth]{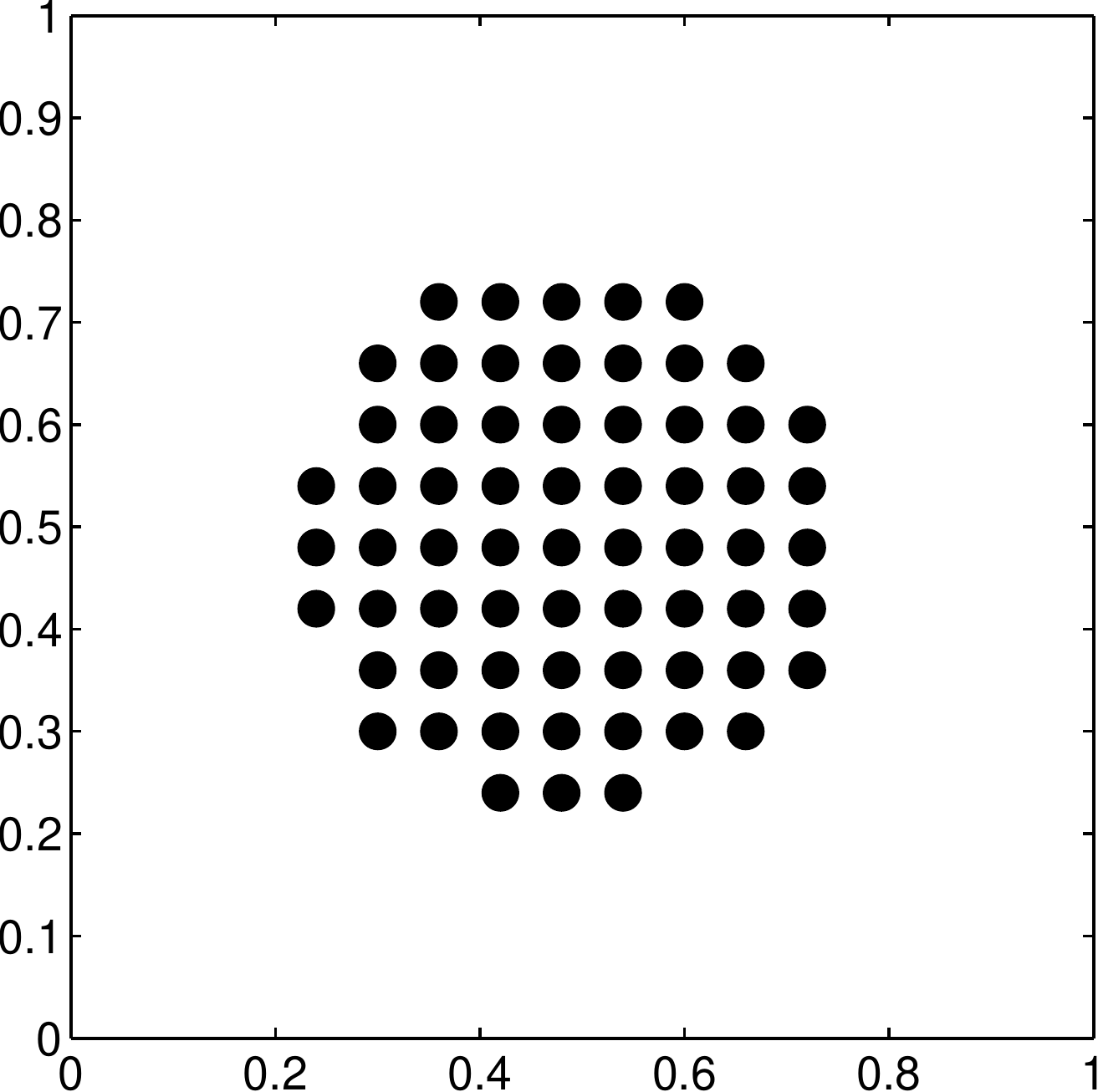}
\includegraphics[width=0.245\textwidth]{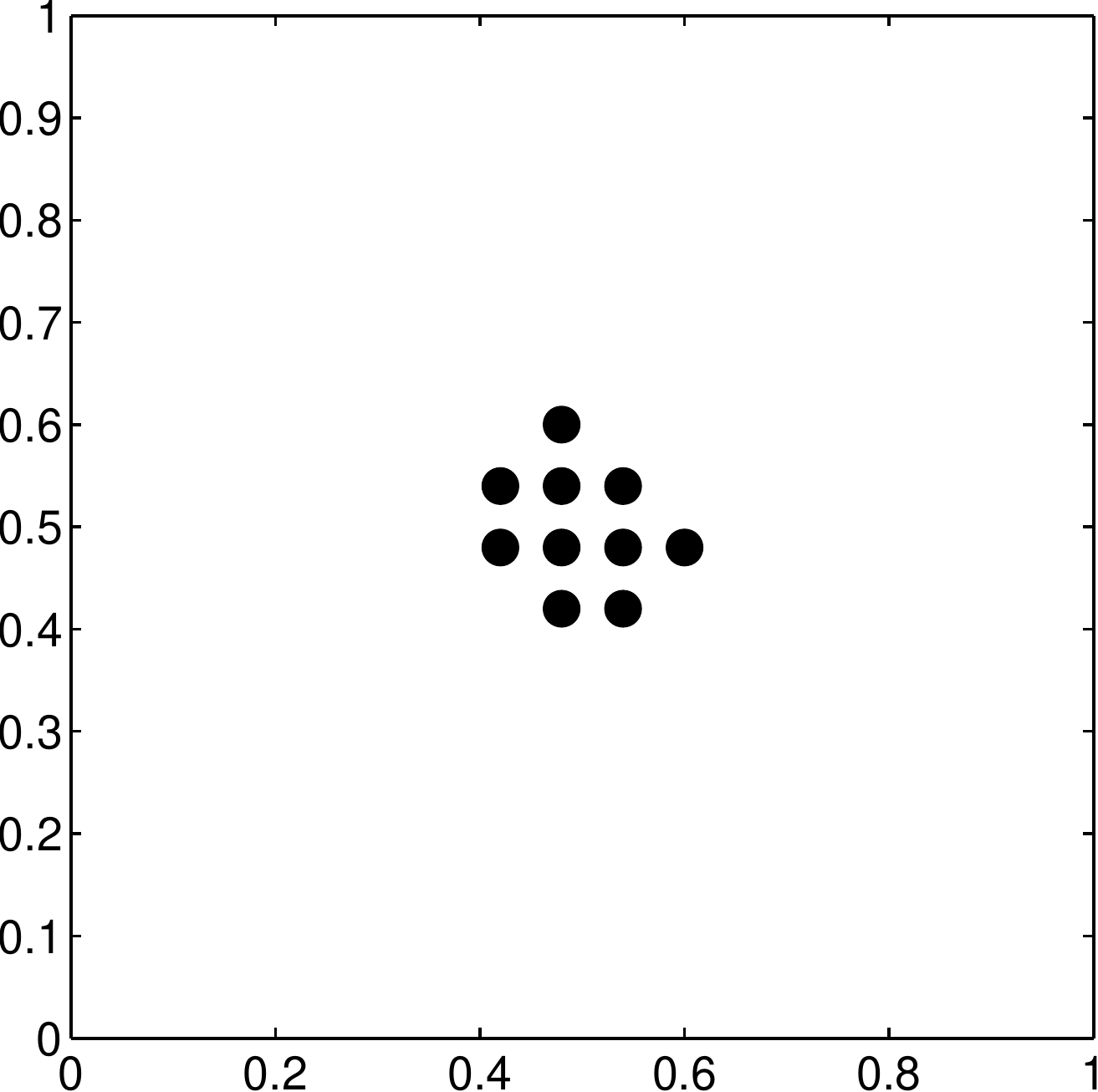}
\caption{\label{f:ex3_act}Example 3: The panels illustrate the active sets for $s = 0.2$, $s =0.4$, $s =0.6$, and $s =0.8$.}
\end{figure}
Clearly the solution behaves differently with respect to $s$. We observe that Algorithm~\ref{algo:ssn} takes 2 iterations to achieve \eqref{eq:eps1}.

\subsection{Example 4: $\fsf = 1$}

In the final example we consider $\fsf=1$. Notice that $1 \not\in \mathbb{H}^{1-s}(\Omega)$ when $s < \frac12$. We let 
\[
 \Psi(\usf) = 1.45 \left|\int_\Omega \usf \ dx \right| + \delta , 
\]
where $\delta = 1e-10$. 
We illustrate the solution in Figure~\ref{f:ex4_sol} and the active set in Figure~\ref{f:ex4_act} for $s = 0.2, 0.4, 0.6$, and $s = 0.8$. We have omitted the plots of $\Psi$ since it is constant. 
We again notice that it takes between 5 to 10 iterations for Algorithm~\ref{algo:ssn} to converge. On the other hand it takes $n = 131, 130, 130, 130$ iterations when $s = 0.2, 0.4, 0.6, 0.8$, respectively, for us to achieve the criterion in \eqref{eq:eps1}.

\begin{figure}[h!]
\includegraphics[width=0.45\textwidth]{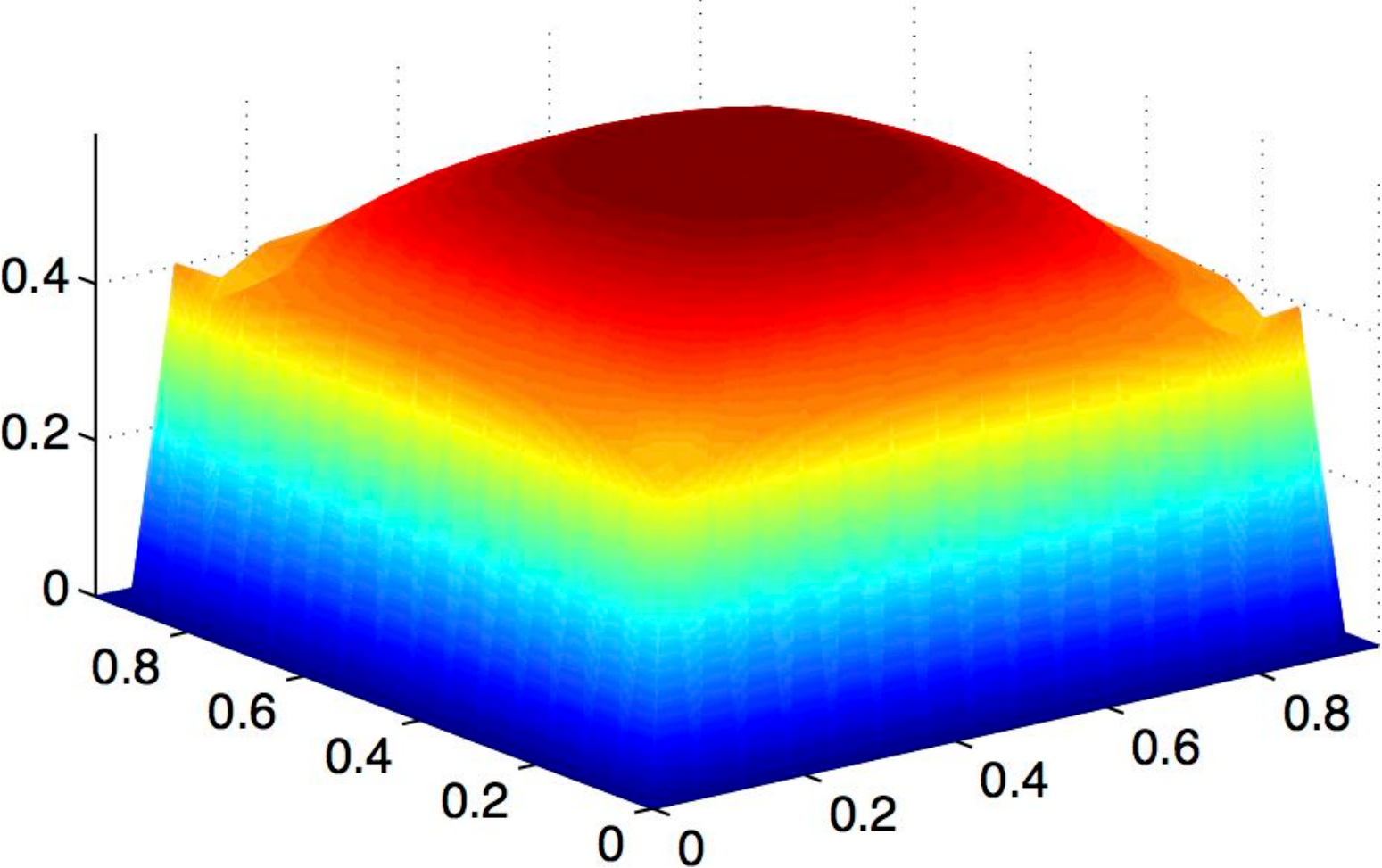}
\includegraphics[width=0.45\textwidth]{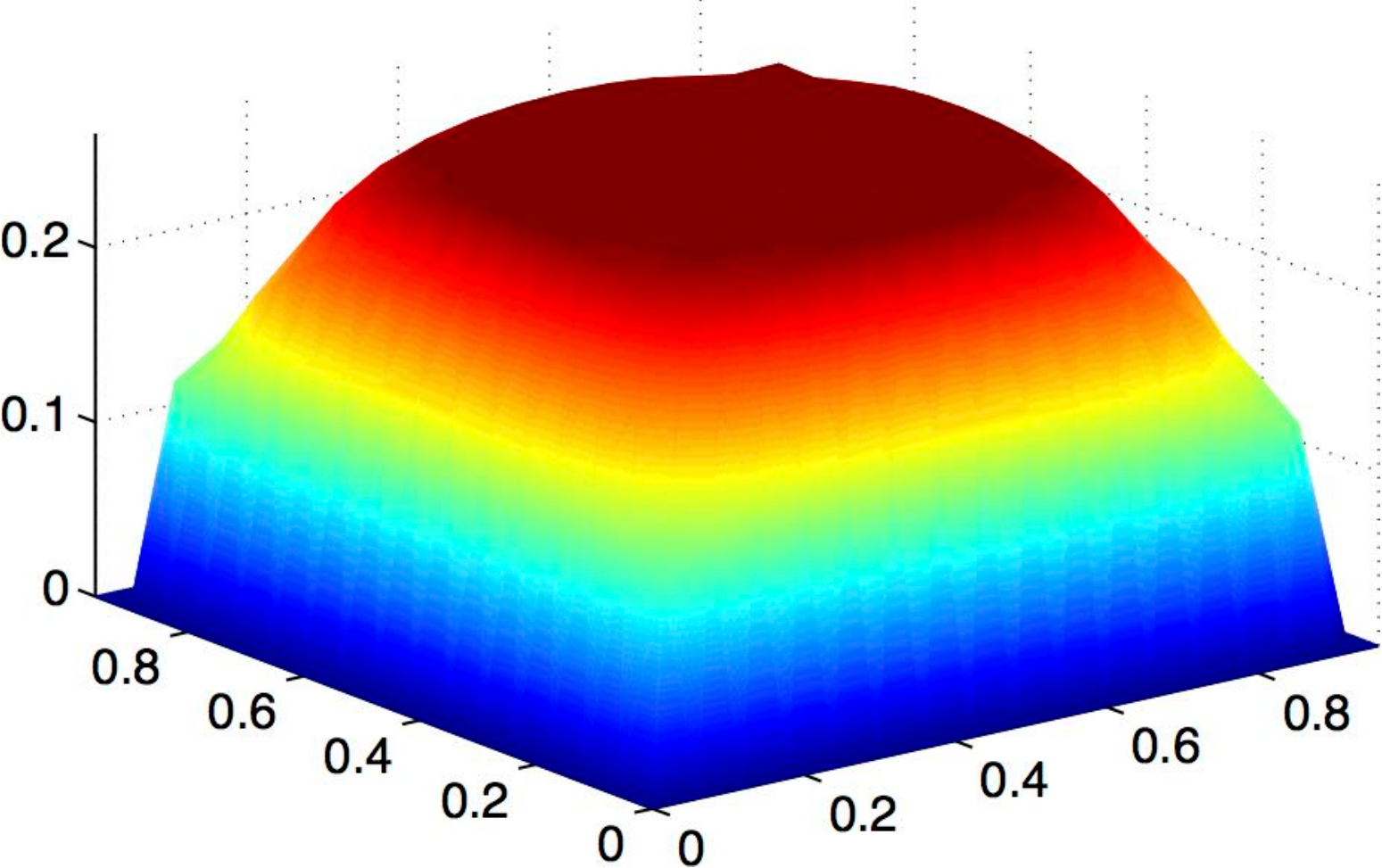}
\includegraphics[width=0.45\textwidth]{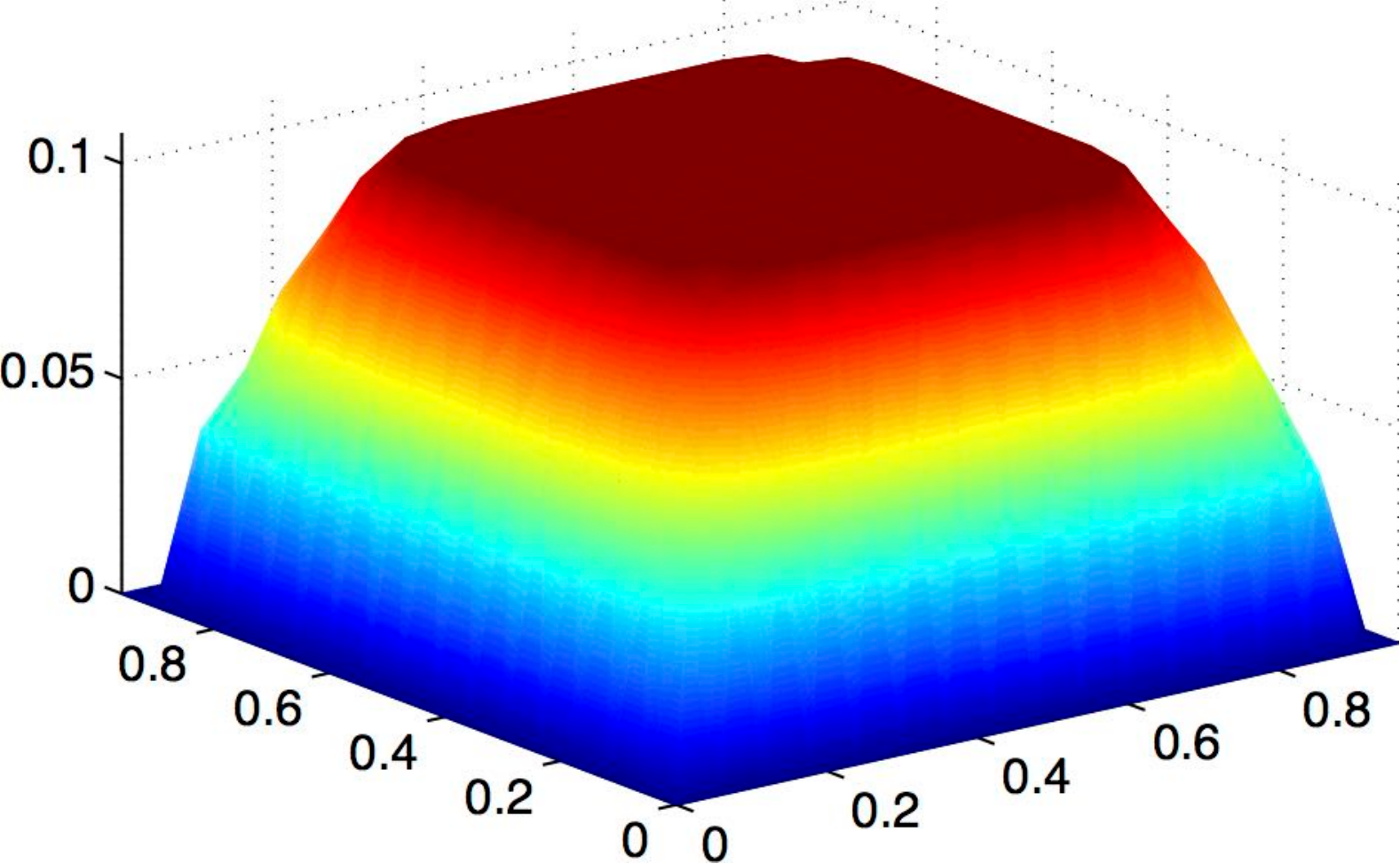}\quad\quad\quad
\includegraphics[width=0.45\textwidth]{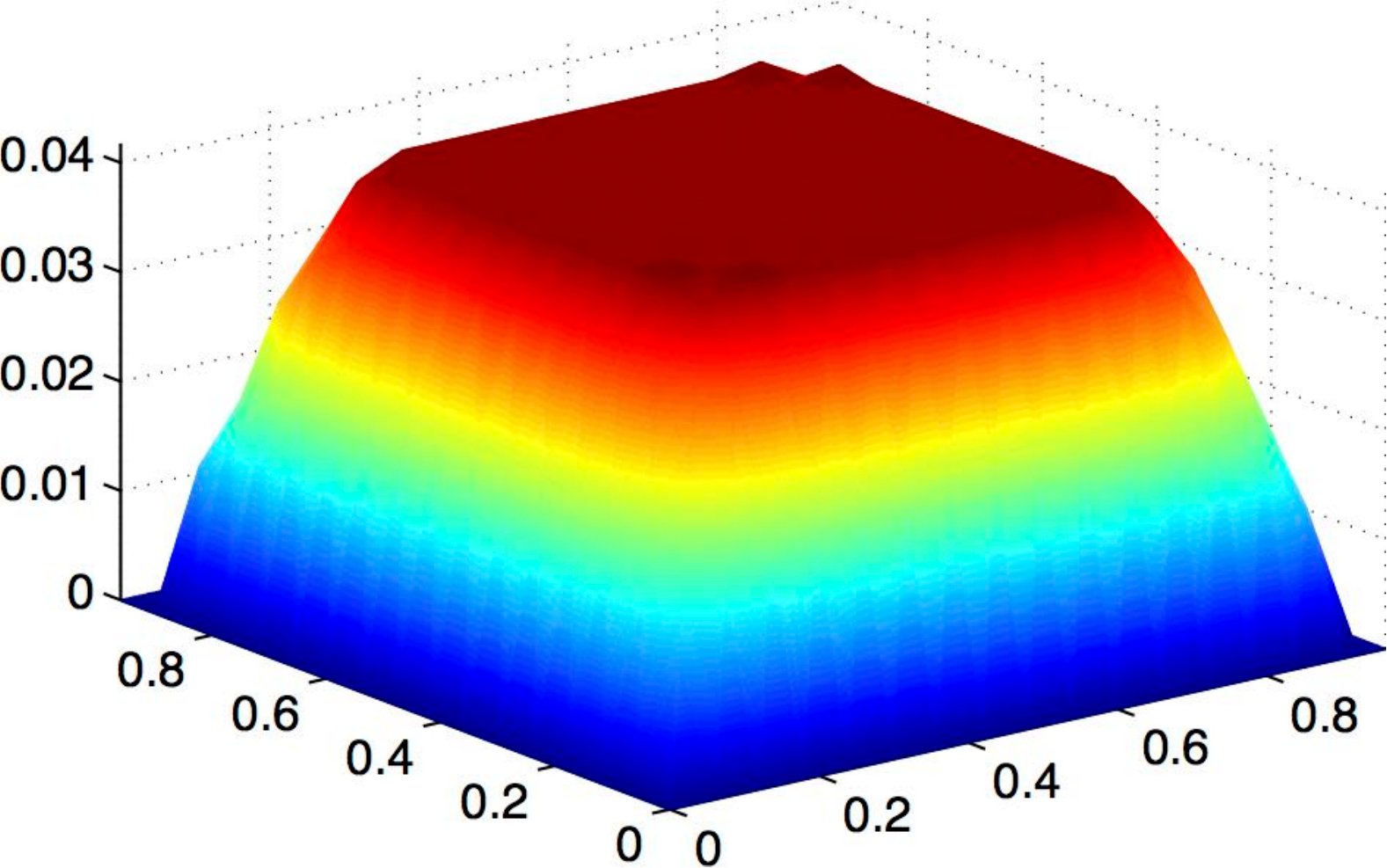}
\caption{\label{f:ex4_sol}Example 4: The top panels illustrate the solution for $s = 0.2$ (left) and $s =0.4$ (right). On the other hand, the bottom panels show the solutions when $s = 0.6$ (left) and $s = 0.8$ (right).}
\end{figure}

\begin{figure}[h!]
\includegraphics[width=0.245\textwidth]{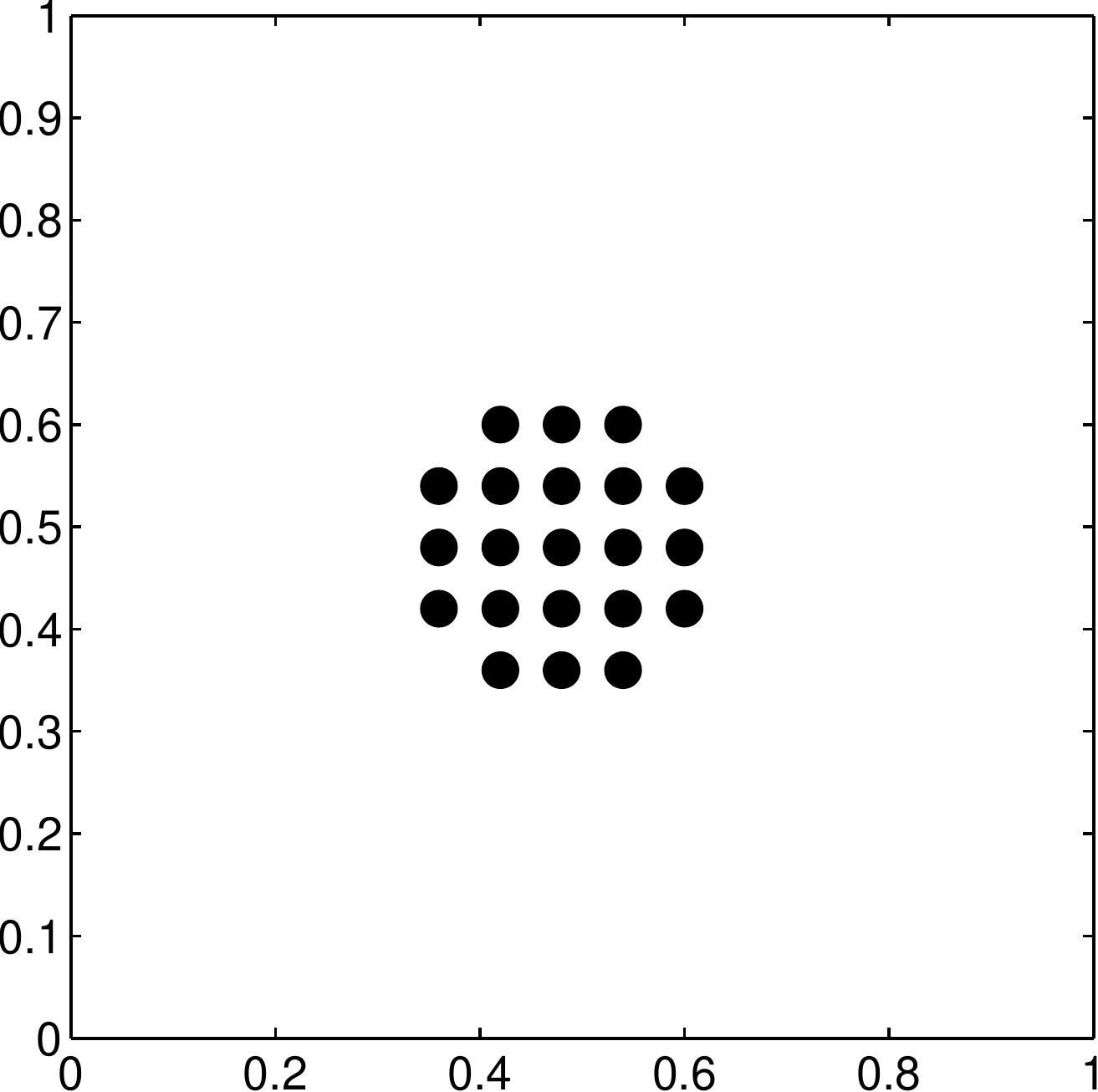}
\includegraphics[width=0.245\textwidth]{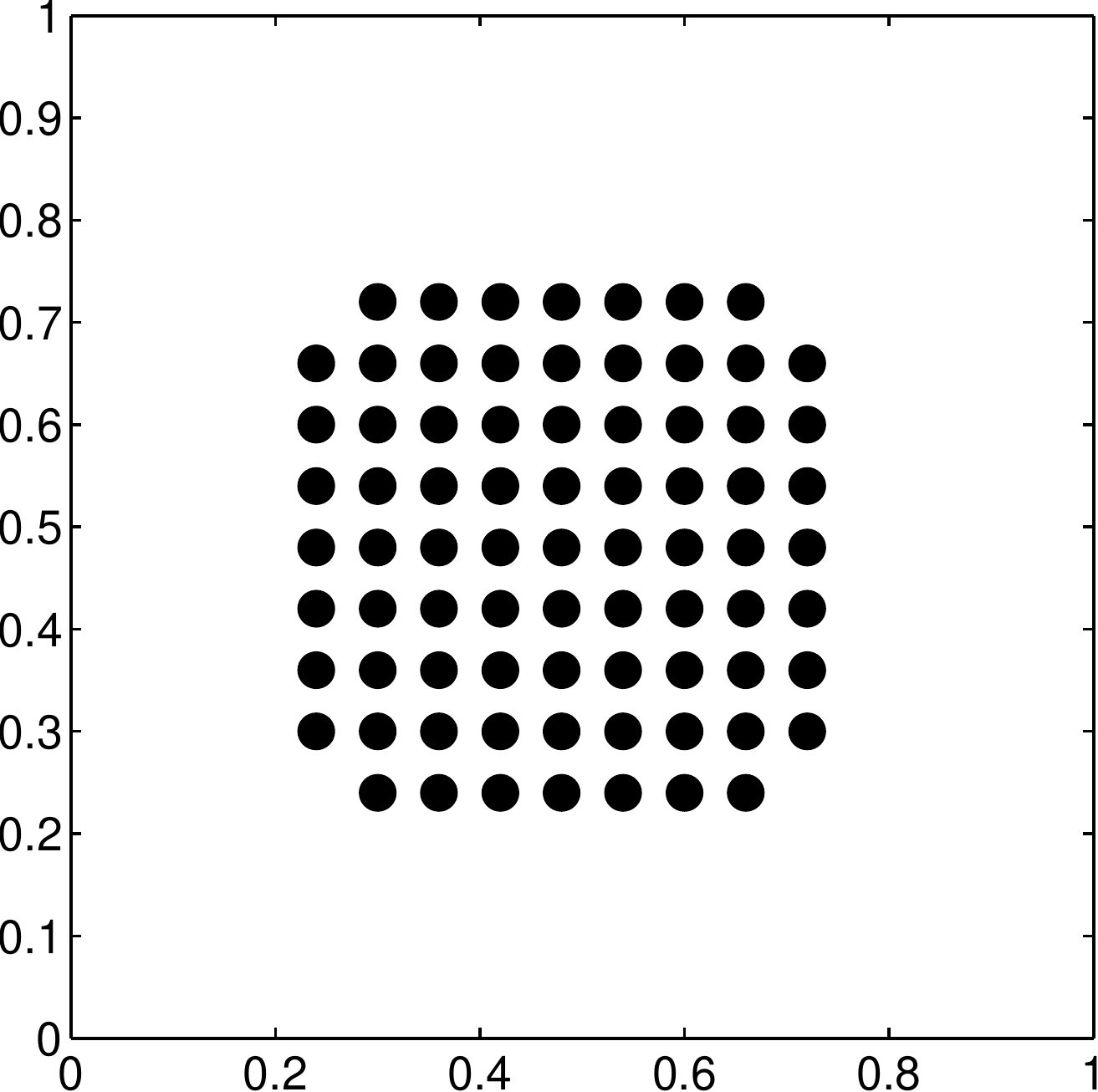}
\includegraphics[width=0.245\textwidth]{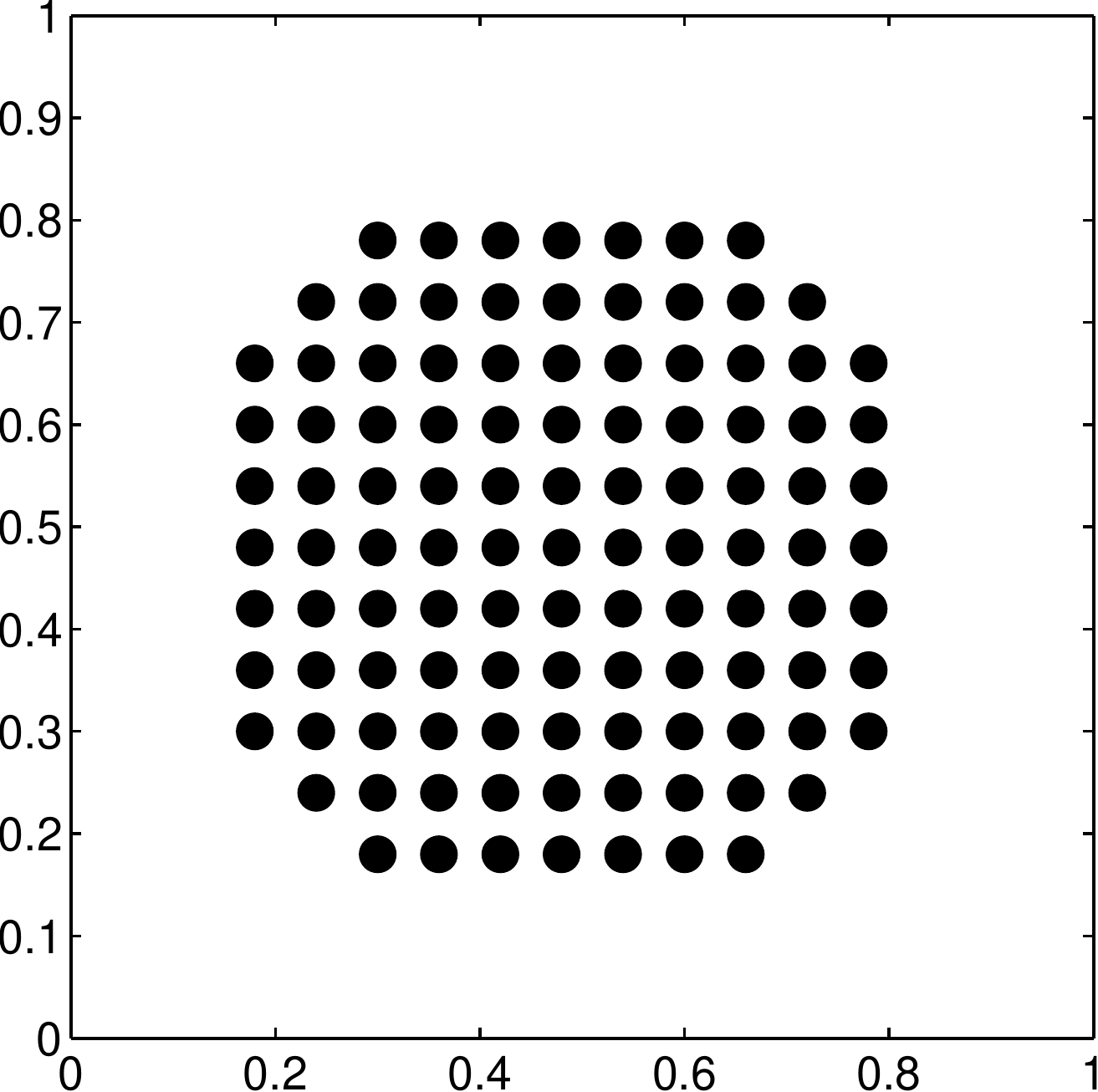}
\includegraphics[width=0.245\textwidth]{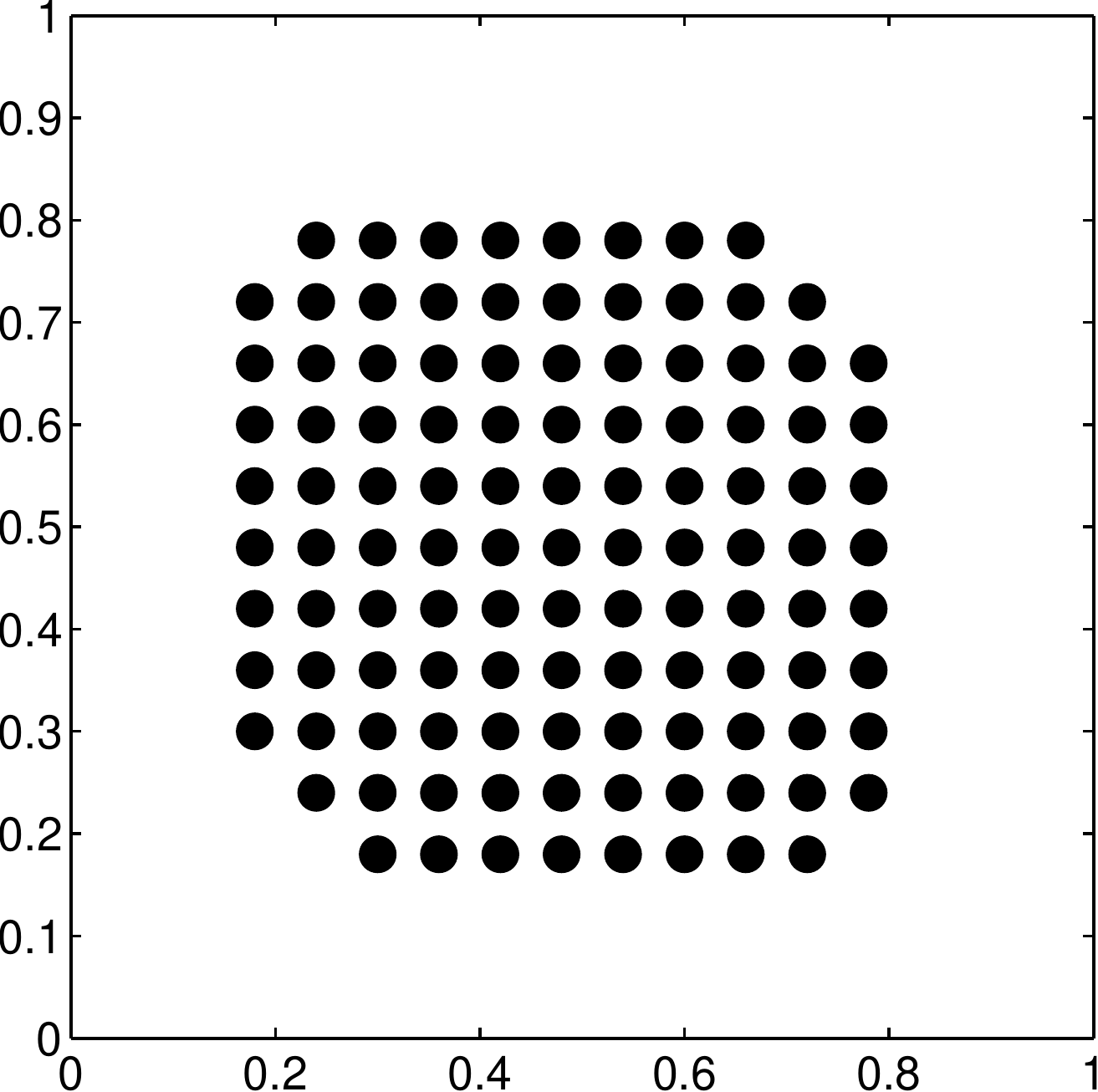}
\caption{\label{f:ex4_act}Example 4: The panels illustrate the active sets for $s = 0.2$, $s =0.4$, $s =0.6$, and $s =0.8$. }
\end{figure}

\subsection*{Acknowledgments}
HA has been supported in part by NSF grant DMS-1521590. CNR has been supported via the framework of {\sc Matheon} by the Einstein Foundation Berlin within the ECMath project SE15/SE19 and acknowledges the
support of the DFG through the DFG-SPP 1962: Priority Programme ``Non-smooth and Complementarity-based
Distributed Parameter Systems: Simulation and Hierarchical Optimization''  within Project 11.


\normalsize
\baselineskip=17pt


\bibliographystyle{abbrv}
\bibliography{biblio}

\end{document}